\documentclass[12pt]{amsart}
\usepackage{amsmath,amsfonts,amsthm,amssymb,epsfig}
\usepackage[T1]{fontenc} 
\usepackage{esvect}
\usepackage{hyperref}
\usepackage{verbatim}
\usepackage{latexsym}
\usepackage{mathrsfs}
\usepackage{bbm}
\input colordvi
\usepackage{bm}
\usepackage[]{graphicx}
\usepackage{amssymb,amsthm,amsmath,amsfonts,epsfig}
\usepackage{epsf,pstricks,multicol,pst-grad,color,pst-node,pst-text,pst-3d,pst-tree}
\newcommand{\D}{\displaystyle}

\newcommand{\rev}[1]{#1}
\newcommand{\deltainv}{\rev{(-\Delta)^{-1}}}
\newcommand{\deltainvh}{\rev{(-\Delta_h)^{-1}}}

\newcommand{\revdim}[1]{#1}
\definecolor{darkviolet}{rgb}{0.58, 0.0, 0.83}
\newcommand{\revp}[1]{#1}
\newcommand{\mfrakj}{\ell}
\newcommand{\ellch}{\mathfrak{l}}

\newcommand{\XCH}{X_{\tt CH}}

\newcommand{\eps}{\varepsilon}
\newcommand{\normal}{{n}}
\newcommand{\hsnormal}{{\revised{n_\Gamma}}} 
\newcommand{\rnvec}[1]{{\revised{#1}}} 

\newcommand{\revised}[1]{#1}

\newcommand{\xx}{x} 

\newcommand{\cPsi}{c_F}
\newcommand{\Vh}{[\mathbb{V}_h(I)]^2}
\newcommand{\Wh}{\mathbb{V}_h(I)}
\newcommand{\RZ}{{\mathbb R} \slash {\mathbb Z}}

\newtheorem{thm}{Theorem}[section]
\newtheorem{definition}{Definition}[section]
\newtheorem{scheme}{Scheme}[section]
\newtheorem{cor}[thm]{Corollary}
\newtheorem{lem}[thm]{Lemma}
\newtheorem{rem}[thm]{Remark}

\numberwithin{equation}{section}
\begin{document}
{\black{
\title[]
{Numerical approximation of the Stochastic Cahn-Hilliard Equation near the Sharp Interface Limit}
%
%
%
%
%
%

\author{Dimitra Antonopoulou}
\address{Department of Mathematics, University of Chester, Thornton Science Park, CH2 4NU, UK.}
\address{Institute of Applied and Computational Mathematics, FO.R.T.H., GR-711 10 Heraklion, Greece.}
\email{d.antonopoulou@chester.ac.uk}
\author{\v{L}ubom\'{i}r Ba\v{n}as}
\address{Department of Mathematics, Bielefeld University, 33501 Bielefeld, Germany}
\email{banas@math.uni-bielefeld.de}
\author{Robert N\"urnberg}
\address{Department of Mathematics, Imperial College London, London, SW7 2AZ, UK}
\email{robert.nurnberg@imperial.ac.uk}
\author{Andreas Prohl}
\address{Mathematisches Institut der Universit\"at T\"ubingen,
                Auf der Morgenstelle 10,
                D-72076 T\"ubingen, Germany.}
\email{prohl@na.uni-tuebingen.de}

%
%

\begin{abstract}
We consider the stochastic Cahn-Hilliard equation
with additive noise term $\eps^\gamma g\, \dot{W}$ ($\gamma >0$)
that scales with the interfacial width parameter $\eps$.
%
We verify strong error estimates for a gradient \revised{flow} structure-inheriting time-implicit discretization, where $\eps^{-1}$
only enters {\em polynomially}; the proof is based on
{higher-moment estimates for}  iterates, 
and a (discrete) spectral estimate for its deterministic counterpart.
For $\gamma$ sufficiently large, convergence {in probability} of iterates
towards the deterministic Hele-Shaw/Mullins-Sekerka problem
in the sharp-interface limit $\eps \rightarrow 0$ is shown.
These convergence results are partly generalized to a fully discrete finite element based discretization. 

We complement the theoretical results by computational studies to provide practical evidence
concerning the effect of noise (depending on its 'strength' $\gamma$) on the geometric
evolution in the sharp-interface limit.
For this purpose we compare the simulations with those
from a fully discrete finite element numerical scheme for the (stochastic) Mullins-Sekerka problem.
The computational results indicate that the limit for $\gamma\geq 1$ is the deterministic problem,
and for $\gamma=0$ we obtain agreement with a (new)
stochastic version of the Mullins-Sekerka problem.

\end{abstract}

%


%

%
%
%
%
%
%
%
%
%

\maketitle

\section{Introduction}

We consider the stochastic Cahn-Hilliard equation with additive noise
\begin{subequations}\label{p1}
\begin{align}\label{p1_1}
&\mathrm{d} u =\Delta\Big{(}-\varepsilon\Delta u+\frac{1}{\varepsilon}f(u)\Big{)}\mathrm{d} t
+\varepsilon^{\gamma}g\, \mathrm{d} W\;\;\;\;&&\mbox{in}\;\; {\mathcal D}_T := (0,T) \times \mathcal{D}\,\\
&\partial_\normal u = \partial_\normal \Delta u =0\;\;\;\;&&\mbox{on}\;\; (0,T) \times \partial \mathcal{D}\
,\\
&u(0, \cdot)=u_0^\eps\;\;\;\;&&\mbox{on}\;\;\mathcal{D}\, .
\end{align}
\end{subequations}
We fix $T>0$,
$\gamma>0$, and $\varepsilon>0$ is a (small) interfacial width parameter.
For simplicity, we assume $\mathcal{D}\subset \mathbb{R}^{2}$ to
be a convex, bounded polygonal domain, with
$\normal \in {\mathbb S}^2$ the outer unit normal along $\partial \mathcal{D}$, and
 $W \equiv \{ W_t;\, 0 \leq t \leq T\}$ \revised{to} be an ${\mathbb R}$-valued Wiener process on a 
filtered probability space $(\Omega, {\mathcal F}, \{ {\mathcal F}_t\}_t, {\mathbb P})$.
The function $g \in C^{\infty}({\mathcal D})$ \revised{is} such that $\int_{{\mathcal D}} g \, {\rm d}x = 0$ to enable conservation of mass in \eqref{p1},
 and $\revdim{\partial_\normal g =0}$ on $\partial \mathcal{D}$.
Furthermore, we assume {$u^\varepsilon_0 \in {\mathbb H}^1$}, {and impose $\int_{\mathcal D} u^{\varepsilon}_0\, {\rm d}x = 0$, for simplicity; generalization
for arbitrary mean values is straightforward}.
%

The nonlinear drift part $f$ in  (\ref{p1}) is the derivative of the
double-well potential $F(u):=\frac{1}{4}(u^2-1)^2$, i.e., 
\revised{$f(u)=F'(u)=u^3-u$}.
Associated to the system (\ref{p1}) is the Ginzburg-Landau free energy 
\begin{equation*}\label{glener}
\mathcal{E}(u) = \int_{{\mathcal D}} \Big(\frac{\eps}{2} |\nabla u|^2 + \frac{1}{\eps}F(u)\Big)\, {\rm d}\xx\, .
\end{equation*}
The particular case $g \equiv 0$ in (\ref{p1}) leads to the deterministic Cahn-Hilliard equation
which can be interpreted as the $\mathbb{H}^{-1}$-gradient flow of the Ginzburg-Landau free energy.
It is convenient to reformulate \eqref{p1} as
\begin{subequations}\label{p3}
\begin{align}
\label{p3_1}
\mathrm{d} u & =\Delta w \mathrm{d}t + \varepsilon^{\gamma}g\, \mathrm{d}W   & \, &\mbox{in}\;\;\mathcal{D}_T,\\
\label{p3_2}
w &=-\varepsilon\Delta u+\frac{1}{\varepsilon}f(u) &\, &\mbox{in}\;\; \mathcal{D}_T\, ,\\
\label{p3_3}
\partial_\normal u &=\partial_\normal w=0 &\, &\mbox{on}\;\;(0,T) \times \partial \mathcal{D}\, ,\\
\label{p3_4}
u(0,\cdot) & = u_0^\eps &&\mbox{on}\;\; \mathcal{D}\, ,
\end{align}
\end{subequations}
where $w$ denotes the chemical potential. 

The Cahn-Hilliard equation has been derived as a phenomenological model for phase separation of binary alloys.
The stochastic version of the Cahn-Hilliard equation, also known as the Cahn-Hilliard-Cook equation, has been proposed in \cite{cook, hohenberg, langer}:
here, the noise term is used to model effects of external fields, impurities in
the alloy, or may describe thermal fluctuations or external mass
supply.
We also mention \cite{gam05},
where 
 computational studies for \eqref{p1} show a better agreement with experimental data in the presence of noise.
For a theoretical analysis of various versions of the stochastic Cahn-Hilliard equation  we refer to \cite{Weber,Weber1,Debu2,elezovic}. Next to its relevancy in materials sciences, \eqref{p1} is used as an approximation to the Mullins-Sekerka/Hele-Shaw problem; 
%
%
by the classical result \cite{abc},
the solution of the deterministic Cahn-Hilliard equation is known to converge to
the solution of the Mullins-Sekerka/Hele-Shaw problem in the sharp interface limit $\eps\downarrow 0$. 
A partial convergence result for the stochastic Cahn-Hilliard equation (\ref{p1})
has been obtained recently in \cite{abk18} for a sufficiently large exponent $\gamma$. 
{We extend this work to eventually validate {\it uniform} convergence of {iterates of the time discretization Scheme \ref{scheme_time} to the sharp-interface limit 
of (\ref{p1}) for vanishing 
numerical (time-step $k$), and regularization (width $\varepsilon$) parameters:
hence, the zero level set of the solution to the geometric interface of the Mullins-Sekerka problem is accurately resolved via Scheme \ref{scheme_time} in the asymptotic limit.}
\\

It is well-known that
an energy-preserving discretization, along with
a proper balancing of numerical parameters and the interface width parameter $\varepsilon$\revised{,}
is required for accurate simulation of the deterministic Cahn-Hilliard equation; see e.g.~\cite{CHFP}:  analytically, this balancing of scales allows to circumvent  a straight-forward
application of Gronwall's lemma in the error analysis, which would otherwise 
cause a factor in
a corresponding error estimate that grows exponentially in $\varepsilon^{-1}$.
 {The present paper pursues a corresponding goal for a  structure-preserving
 discretization of the stochastic Cahn-Hilliard equation \eqref{p1}; we identify
 proper discretization scales which allow a resolution of interface-driven evolutions,
 and thus avoid a Gronwall-type argument in the corresponding strong error analysis.
 {This allows for practically relevant scaling scenarios of involved numerical parameters to accurately approximate solutions of \eqref{p1} even in
 the asymptotic regime where $\varepsilon \ll 1$.}
\\

The proof of a strong error estimate for a space-time discretization of (\ref{p1})  which {causes  only {\em polynomial dependence on $\varepsilon^{-1}$} in involved stability constants} uses the following ideas:
\begin{itemize}
\item[(a)] We use the time-implicit Scheme \ref{scheme_time}, whose iterates inherit 
the {basic energy bound (see Lemma \ref{lem_xch}, i))} from \eqref{p1}.
         We benefit from a {weak
         monotonicity property of the drift operator} {in the proof of Lemma \ref{thma2}} to effectively handle the cubic nonlinearity in the drift part.
\item[(b)] For $\gamma >0$ sufficiently large, we view \eqref{p1} as a stochastic perturbation of the deterministic Cahn-Hilliard equation (i.e., \eqref{p1} with \revised{$g \equiv 0$}),
and proceed analogically also in the discrete setting. We then benefit in the proof of Lemma \ref{thma2} from (the discrete version of) the spectral estimate (\ref{genspec}) from \cite{Ch:94,AlFu:93} 
for the deterministic Cahn-Hilliard equation (see Lemma~\ref{lem_xch}, v)).
\item[(c)] For the deterministic setting \cite{CHFP}, an induction argument is used on the discrete level, 
{which addresses the cubic error term (scaled by $\varepsilon^{-1}$) in Lemma \ref{thma2}.}
%
%
This argument may not be generalized in a straightforward way
to the current stochastic setting where the discrete solution {is a 
sequence of random variables allowing for (relatively) large temporal variations}.
For this reason we consider the propagation of errors on
two complementary subsets of $\Omega$: on the large subset
$\Omega_2$ we verify the error estimate (Lemma \ref{lem_patherr}), while
we benefit from the higher-moment estimates for iterates of Scheme \ref{scheme_time} from (a) to
derive a corresponding estimate on the small set $\Omega \setminus \Omega_2$ (see Corollary~\ref{cor_err_xxa}). 
A combination of both results then establishes our first main result:
{a {strong} error estimate for the numerical approximation of the stochastic Cahn-Hilliard equation}
(see Theorem \ref{thmfin}), {avoiding Gronwall's lemma}.
\item[(d)] 
Building on the results from (c), and
using an $\mathbb{L}^\infty$-bound for the solution of Scheme \ref{scheme_time}
(Lemma~\ref{lem_linftybnd}), along with error estimates in stronger norms (Lemma~\ref{lem_err_l2l4}),
we show {uniform} convergence {of iterates} 
on large subsets of $\Omega$ (Theorem~\ref{thm_err_linfty}).
This intermediate result then implies the second main result of the paper:
the convergence in probability of iterates of Scheme \ref{scheme_time} to the sharp interface limit
in Theorem~\ref{cor_sharp} {for sufficiently large $\gamma$}.
In particular, we show {that \revised{the} numerical solution of (\ref{p1})} uniformly converges in probability to
$1$, $-1$ in the interior and exterior of the geometric interface of the deterministic Mullins-Sekerka problem (\ref{eq:MS}), respectively.
As a consequence we obtain uniform convergence of the zero level set of the numerical solution to the geometric interface of the Mullins-Sekerka problem in probability; cf.~Corollary~\ref{cor_gamma}.
\end{itemize}
}
The error analysis below in particular identifies  proper balancing strategies
of numerical parameters with the interface width that allow to approximate the limiting sharp interface model 
for realistic problem setups, {and motivates the use of space-time adaptive meshes for numerical simulations; see e.g.~\cite{ps19}}.  In Section \ref{sec_numer},
we present computational studies which evidence asymptotic properties of the solution for different scalings of the noise term.
{
Our studies suggest the deterministic Mullins-Sekerka problem as sharp-interface limit
already for $\gamma \geq 1$; we \revised{observe} this in simulations
for spatially colored, as well 
as for the space-time white noise. In contrast, corresponding simulations for $\gamma = 0$
indicate that the sharp-interface limit is a stochastic version of the Mullins-Sekerka problem; {see Section \ref{hsms-1}}.  \\
}

{To sum up, the convergence analysis presented in this paper 
is a combination of a perturbation and discretization error analysis.
The latter depends on stability properties
of the proposed numerical scheme: higher-moment energy estimates for the
Scheme \ref{scheme_time},   
a discrete spectral estimate for the related deterministic
variant, {and a local error analysis on the sample set $\Omega$} are crucial ingredients of our approach.
The techniques developed in this paper constitute a general framework which can be used to treat
different and/or more general phase-field models including the stochastic Allen-Cahn equation, and {apply to settings which}
involve multiplicative noise, driving trace-class Hilbert-space-valued Wiener processes, and bounded polyhedral  domains ${\mathcal D} \subset {\mathbb R}^3$, as well.
%
%
}

The paper is organized as follows. 
Section~\ref{sec_cont_ch} is dedicated to the analysis of the continuous problem.
The time discretization Scheme \ref{scheme_time} is proposed in Section~\ref{asec_time} and rates of
convergence are shown, while
{Section~\ref{sec_space_time} extends this convergence analysis to
its finite-element discretization.}
The convergence of the numerical discretization to the sharp-interface limit
is studied in Section~\ref{sec_sharp}.
Section~\ref{sec_numer} contains the details of the implementation of the numerical schemes for the stochastic
Cahn-Hilliard and the stochastic Mullins-Sekerka problem, respectively, as well as
computational experiments which {complement} the analytical results.

\section{The stochastic Cahn-Hilliard equation}\label{sec_cont_ch}
\subsection{Notation}\label{sec_cont_ch_1}
For $1\leq p \leq \infty$, we denote by $\bigl( \mathbb{L}^p, \Vert \cdot \Vert_{\mathbb{L}^p}\bigr)$ the standard spaces of $p$-th order  integrable functions on $\mathcal{D}$.
By $(\cdot,\cdot)$ we denote the $\mathbb{L}^2$-inner product,
{and let $\Vert\cdot\Vert = \Vert \cdot \Vert_{\mathbb{L}^2}$.
For} $k\in \mathbb{N}$
we write $\bigl(\mathbb{H}^k, \Vert \cdot \Vert_{\mathbb{H}^k}\bigr)$ for 
usual Sobolev spaces on $\mathcal{D}$, and
$\mathbb{H}^{-1} = (\mathbb{H}^1)^\prime$. 
We define $\mathbb{L}^2_0 := \{ \phi \in \mathbb{L}^2; \,\, \int_\mathcal{D} \phi \,\mathrm{d}\xx = 0\}$,
and for $v \in \mathbb{L}^2$
we denote its zero mean counterpart as $\overline{v} \in \mathbb{L}^2_0$, i.e.,  $\overline{v} := v - \frac{1}{|\mathcal{D}|}\int_{\mathcal{D}}v\,\mathrm{d}\xx$.
We frequently use the isomorphism $\deltainv: \mathbb{L}^2_0 \rightarrow {\mathbb{H}^2} \cap \mathbb{L}^2_0$, where ${w} = \deltainv \overline{v}$ is
the unique solution of
\begin{equation*}
-\Delta {w}  =  \overline{v} \quad \mathrm{in}\,\, \mathcal{D}, \qquad
\displaystyle  \partial_\normal {w}  =  0 \quad \mathrm{on}\,\, \partial \mathcal{D}.
\end{equation*}
In particular, $(\nabla \deltainv \overline{v}, \nabla \varphi) = \rev{(\overline{v}, \varphi)}$ for all $\varphi \in \mathbb{H}^1$, $\overline{v}\in \mathbb{L}^2_0$.
\rev{Below, we denote $\Delta^{-1/2} \overline{v}:= \nabla \deltainv \overline{v}$
and note that norms $\|\overline{v} \|_{\mathbb{H}^{-1}}$ and $ \Vert\Delta^{-1/2}\overline{v} \Vert$ are equivalent for all $\overline{v}\in \mathbb{L}^2_0$.}
Throughout the paper, $C$  denotes a generic positive constant that may depend on $\mathcal{D}$, $T$, but is independent of $\eps$.

\subsection{The Problem}\label{sec_cont_ch_2}
We recall the definition of a strong variational solution of the stochastic Cahn-Hilliard equation (\ref{p1}); its existence, uniqueness, and regularity properties have been obtained in 
\rev{\cite[Thm. 8.2]{elezovic}, \cite[Prop. 2.2]{Debu2}.}
\begin{definition}\label{def_varsol}
Let $u_0^\eps \in L^2(\Omega, \mathcal{F}_0, \mathbb{P}; \mathbb{H}^1) \cap L^4(\Omega, \mathcal{F}_0, \mathbb{P}; \mathbb{L}^4)$
\rev{and denote $\underline{\mathbb{H}}^2 = \{\varphi \in \mathbb{H}^2,\,\, \partial_\normal \varphi = 0\,\,\mathrm{on}\,\, \partial\mathcal{D} \}$}. 
Then, the process 
$$u\in L^2\bigl(\Omega, \{ \mathcal{F}_t\}_t, \mathbb{P}; C([0,T]; \mathbb{H}^1)\cap \rev{L^2(0,T; \underline{\mathbb{H}}^2)}\bigr) \cap L^4\bigl(\Omega, \{ \mathcal{F}_t\}_t, \mathbb{P}; C([0,T]; \mathbb{L}^4)\bigr)$$ 
is called a strong solution of (\ref{p1}) 
if it satisfies $\mathbb{P}$-a.s.~and for all $0 \leq t \leq T$
$$
\bigl(u(t), \varphi \bigr) = (u_0^\eps, \varphi) + \rev{\int_0^t \Big(-\eps \Delta u + \frac{1}{\eps} f(u), \Delta \varphi\Big)\mathrm{d}s}
+ \eps^\gamma \int_0^t (\varphi,g) \, \mathrm{d}W(s) \quad \forall \varphi \in \underline{\mathbb{H}}^2\, .
$$
\end{definition}
%
The following lemma establishes existence and bounds for the strong solution $u$ of (\ref{p1}) and 
for the chemical potential $w$ from (\ref{p3_2});
cf.~\cite[Section~2.3]{Debu2} for a proof of i), while ii)
follows similarly as part i) by the It\^o formula and the Burkholder-Davis-Gundy inequality.
\begin{lem}\label{lem_ener_cont}
{Let $T>0$. There exists a unique}
strong  solution $u$ of (\ref{p1}), and there hold
\begin{itemize}
\item[i)] \quad
$
\D \mathbb{E}\big[ \mathcal{E}\bigl(u(t)\bigr)\big]
+ \mathbb{E}\Big[ \int_0^t\|\nabla w(s)\|^2\, \mathrm{d}s\Big] 
\leq 
C \big( \mathcal{E}(u_0^\eps) + 1\big)
\qquad \forall \, t\in [0,T]\, ,
$
\item[ii)] \quad
For any $p\in \mathbb{N}$ there exists $C\equiv C(p)>0$ such that
\quad 
$$\D \mathbb{E}\big[ \sup_{t\in [0,T]} \mathcal{E}\bigl(u(t) \bigr)^p\big] 
\leq  C\bigl(  \mathcal{E}(u_0^\eps)^p  + 1\bigr)\, .
$$
\end{itemize}
\end{lem}

{{
\subsection{Spectral estimate}\label{spectral estimate}


We denote by $u_{\tt CH}: \mathcal{D}_T \rightarrow {\mathbb R}$ the solution of the deterministic Cahn-Hilliard equation,
i.e., \eqref{p1} with $g \equiv 0$.
{Let $\eps_0 \ll 1$; throughout the paper we assume that for every $\eps \in (0, \eps_0)$, there exists an arbitrarily close  approximation $u_{\tt A}\in C^2(\overline{\mathcal{D}}_T)$ of 
$u_{\tt CH}$ which satisfies the spectral estimate (cf.~\cite[relation (2.3)]{abc})
\begin{equation}\label{genspec}
\inf_{0\leq t\leq T}\inf_{\psi\in
\mathbb{H}^{1}, \; w=\deltainv\psi} \frac{\eps\|\nabla\psi\|^2+\frac{1}{\eps}\bigl(f'(u_{\tt
A})\psi,\psi \bigr)}{\|\nabla w\|^2}\geq -C_0\, ,
\end{equation}
where the constant $C_0 >0$ does not depend on $\varepsilon >0$; cf.~\cite{abc,AlFu:93,Ch:94}.}

\subsection{Error bound between $u$ of (\ref{p1}) and $u_{\tt CH}$ of (\ref{p1}) with $g \equiv 0$.}\label{dimitra-bound}
%
%

In \cite{abk18} the authors study the convergence of the solution of the stochastic Cahn-Hilliard equation (\ref{p1})
to the deterministic sharp-interface limit. In particular, they show the convergence in probability
of the solution $u$ of (\ref{p1}) to the approximation $u_{\tt A}$ of $u_{\tt CH}$  for sufficiently large $\gamma>0$.
Apart from the spectral estimate (\ref{genspec}), a central ingredient of their analysis 
is the use of a stopping time argument to control the drift nonlinearity. The stopping time which, in our setting, is defined as
$$
T_\eps:=\inf\Big\{ t\in[0,T]:\,\,\frac{1}{\eps}\int_0^t\|u(s)-u_{\tt CH}(s)\|_{\mathbb{L}^3}^3\, {\rm d}s >\eps^{\sigma_0}\Big\} $$
for some constant $\sigma_0>0$, enables the derivation of the estimates in Lemma~\ref{lem_rest}  below  up to the stopping time $T_\eps$ on a large sample subset 
$$
\Omega_1 := \Big\{ \omega\in \Omega:\,\, \eps^\gamma \revdim{\sup_{t\in [0,T_\eps]}}\Big|\int_0^{t}\big(
\revised{u(s)-u_{\tt CH}(s)}, \deltainv g\, \mathrm{d}W(s)\big)\Big| \leq \eps^{\kappa_0}\Big\} 
$$
that satisfies $\mathbb{P}[\Omega_1] \rightarrow  1$ for 
\revised{$\eps \downarrow 0$, for some constant $\kappa_0$}.
On specifying the condition {\bf (A)}  below  it can be shown that $T_\eps\equiv T$, which yields Lemma~\ref{lem_rest}.
In this section we extend the work \cite{abk18} by showing a strong error estimate for $u-u_{\tt CH}$ in Lemma~\ref{RL2est}.

In Section \ref{asec_time} we perform an analogous analysis on the discrete level 
by using a stopping index $J_\eps$, and a set $\Omega_2$ which are discrete counterparts of $T_\eps$ \revised{and} $\Omega_1$, respectively.
Both approaches require a lower bound for the noise strength $\gamma$ to ensure, in particular, positive probability of the sets $\Omega_1$ \revised{and} $\Omega_2$, respectively.

\medskip

For the analysis in this section we require the following assumptions to hold.
\begin{itemize}
\item[{\bf (A)}] Let  $\mathcal{E}(u^\varepsilon_0) \leq C$. Assume that {the triplet $(\sigma_0, \kappa_0, \gamma) \in \bigl[{\mathbb R}^+\bigr]^3$ satisfies} 
$$
{\sigma_0 > 12\,, \qquad \sigma_0 > \kappa_0 > \frac{2}{3}\sigma_0 + 4\,, \qquad \gamma >  \max\big\{ \frac{23}{3}, \frac{\kappa_0}{2}\big\}}\, .
$$
\end{itemize}
}
{Assumption {\bf (A)} ensures positivity of all exponents in the estimates in the lemmas of this section.}
The following lemma relies on the spectral estimate (\ref{genspec}) and is a consequence of \cite[Theorem 3.10]{abk18} for $p=3$, $d=2$, {where a slightly different notational setup is used.}
\begin{lem}\label{lem_rest}
{Suppose}
${\bf (A)}$. {There exists}
$\eps_0 \equiv \eps_0(\sigma_0, \kappa_0) >0$ such that
for any $\eps \leq \eps_0$ and {sufficiently large $\ellch >0$}
\begin{eqnarray*}\label{rest0}
{\rm i)} &&{\mathbb P}\bigl[\|u -u_{\tt A}\|^2_{L^\infty(0,T; \mathbb{H}^{-1})}  \leq {C}  {\eps^{\kappa_0}}\bigr] \geq 1- 
C \varepsilon^{(\gamma + \frac{\sigma_0+1}{3} - \kappa_{0})\ellch}\,,\\
\label{rest1}
{\rm ii)}&&{\mathbb P}\bigl[\eps \|\nabla [u -u_{\tt A}]\|^2_{L^2(0,T; \mathbb{L}^{2})}  \leq {C} {\eps^{\frac{2\sigma_0}{3}}} \bigr] \geq 1- 
C \varepsilon^{(\gamma + \frac{\sigma_0+1}{3} - \kappa_{0})\ellch}\, ,
\end{eqnarray*}
where $\ellch$ 
and $C \equiv C(\ellch)>0$ are independent of  $\gamma$, $\sigma_0$, $\kappa_0$ and $\eps$.
\end{lem}
\rev{A closer inspection of the proofs in \cite{abk18} (cf. \cite[Lemma 4.3]{abk18} in particular) reveals that the
parameter $\ellch$ can be chosen arbitrarily large in the above theorem.}

{We now use Lemma \ref{lem_rest} to show bounds for the difference $u-u_{\tt CH}$ in different norms.}
\begin{lem}\label{RL2est}
{Suppose {\bf (A)}, and $\varepsilon \leq \varepsilon_0$, for
$\varepsilon_0 \equiv \varepsilon_0(\sigma_0, \kappa_0)>0$ sufficiently small.}
There exists
$C>0$  such that
\begin{equation*}\label{L2estnew}
\begin{split}
\mathbb{E} \Big{[} \|u -u_{\tt CH}\|_{L^\infty(0,T;\mathbb{H}^{-1})}^2 
+ {\eps \|\nabla [u -u_{\tt CH}]\|^2_{L^2(0,T; \mathbb{L}^{2})}}
\Big{]} 
\leq 
C\eps^{\frac{2\sigma_0}{3}}\, .
\end{split}
\end{equation*}

\end{lem}
\begin{proof}
{By \cite[Theorem~2.1]{abc} (see also \cite[Theorem 4.11 and Remark 4.6]{abc})
{there exists} $u_{\tt A} \in C^2(\overline{\mathcal{D}}_T)
{\cap {\mathbb L}^2_0}$ which satisfies (\ref{genspec}) and} 
\begin{equation}\label{est_uauch}
\|u_{\tt A}-u_{\tt CH}\|_{L^\infty(0,T;\mathbb{H}^{-1})}^2 + {\|u_{\tt A}-u_{\tt CH}\|_{L^2(0,T;{\mathbb H}^{1})}^2}
\leq C\eps^{2\gamma}\,,
\end{equation}
and, cf. \cite[Theorem~2.3]{abc},
\begin{equation}\label{est_c1_uauch}
\|u_{\tt A}-u_{\tt CH}\|_{C^1(\mathcal{D}_T)} \leq C\eps\,.
\end{equation}
By using the energy bound for $u_{\tt CH}$ and (\ref{est_c1_uauch}) we get
$\|{{u}_{\tt A}}\|_{L^\infty(0,T;\mathbb{H}^{1})} \leq C$.

Consider the subset {$\widetilde{\Omega}_1 \subset \Omega$} (cf. \cite[Lemma~4.5,~Lemma~4.6]{abk18}),
\revised{%
\[
\widetilde{\Omega}_1 :=  \D\big\{ \omega \in \Omega:\, \|u-u_{\tt A}\|^2_{L^\infty(0,T, \mathbb{H}^{-1})} 
+ {\eps \|\nabla [u -u_{\tt A}]\|^2_{L^2(0,T; \mathbb{L}^{2})}  \leq {C}\eps^{\frac{2\sigma_0}{3}}  } \big\}\,.
\]}%
By {Lemma \ref{lem_rest}, ii)}, we have  $\mathbb{P}[\widetilde{\Omega}_1^c] \leq C\eps^{\big(\gamma + \frac{\sigma_0+1}{3} - \kappa_{0}\big)\ellch} \rev{<1}$,
for sufficiently large $\ellch>0$. 
Then using Lemma~\ref{lem_ener_cont}, ii)  and (\ref{est_c1_uauch}),
%
%
we  estimate the error 
$$
{{\tt Err}_{\tt A}} :=\|u-u_{\tt A}\|_{L^\infty(0,T;\mathbb{H}^{-1})}^2  + {\eps \|\nabla [u -u_{\tt A}]\|^2_{L^2(0,T; \mathbb{L}^{2})} }\, ,
$$ 
as 
\begin{equation*}
\begin{split}
\mathbb{E}\big[{\tt Err}_{\tt A}\big] & =\int_{\Omega}\mathbbm{1}_{\widetilde{\Omega}_1}{\tt Err}_{\tt A} \, {\rm d}\omega +
\int_{\Omega}\mathbbm{1}_{\widetilde{\Omega}_1^c}{\tt Err}_{\tt A} \, {\rm d}\omega\\
&\leq  {C}{\eps^{\frac{2\sigma_0}{3}}} + 
C {\bigl( {\mathbb P}[ {\widetilde{\Omega}_1^c}] \bigr)^{1/2}} 
\Bigl(\mathbb{E}\Big[\sup_{[0,T]}\mathcal{E}\bigl(u(t)\bigr)^2\Big] 
+ \|u_{\tt A}\|_{L^\infty(0,T;\mathbb{H}^{1})}^2
 \Bigr)^{1/2}\\
&\leq  C\bigl( {\eps^{\frac{2\sigma_0}{3}}}
+ {\eps^{(\gamma + \frac{\sigma_0+1}{3} - \kappa_{0})\frac{\ellch}{2}}} \bigr) \, .
\end{split}
\end{equation*}
It is due to {\bf (A)} that $\gamma + \frac{\sigma_0+1}{3} - \kappa_{0} > 0$. We
now choose $\ellch$ sufficiently large  such that $\big(\gamma + \frac{\sigma_0+1}{3} - \kappa_{0}\big)\frac{\ellch}{2} > \frac{2}{3}\sigma_0$
and the statement follows from the estimate for ${\tt Err}_{\tt A}$ and (\ref{est_uauch}) by the triangle inequality.

%
%
\qed
\end{proof}

\section{A time discretization Scheme for (\ref{p1})}\label{asec_time}


For fixed $J \in {\mathbb N}$, let $0=t_0<t_1<\cdots<t_J=T$ be an
equidistant partition of $[0,T]$ with step size $k = \frac{T}{J}$, and $\Delta_j W := W(t_j) - W(t_{j-1})$, $j=1,\dots, J$. 
We  approximate \eqref{p1} by the following scheme:
\begin{scheme}\label{scheme_time} For every $1 \leq j \leq J$, find 
a $[{\mathbb H}^1]^2$-valued r.v.~$(X^j, w^j)$
 such that ${\mathbb P}$-a.s.
\begin{equation*}
\begin{split}
&(X^j-X^{j-1},\varphi)+k(\nabla w^{j},\nabla
\varphi)=\varepsilon^{\gamma}\bigl(g,\varphi
\bigr)\Delta_j W
\;\;\;\; \, \, \quad \, \forall\, \varphi \in {\mathbb H}^1\,,\\
&\varepsilon(\nabla X^j,\nabla \psi)+\frac{1}{\varepsilon}
\bigl(f(X^j),\psi \bigr)=(w^j,\psi) \qquad \qquad \quad
\quad \ \, \forall\, \psi \in
{\mathbb H}^1\, ,\\
&X^0 =u_0^\eps \in {\mathbb H}^1\, .
\end{split}
\end{equation*}
\end{scheme}
The solvability  and uniqueness 
of $\{(X^j, w^j)\}_{j\geq 1}$,
as well as the $\mathbb{P}$-a.s.~conservation of mass
of $\{X^j\}_{j\geq 1}$ are immediate.

For the error analysis of Scheme~\ref{scheme_time}, we use the 
iterates $\bigl\{ (X^{j}_{\tt CH}, w^{j}_{\tt CH})\bigr\}_{j=0}^J \subset \bigl[ {\mathbb H}^1\bigr]^2$  which solve Scheme~\ref{scheme_time} for $g \equiv 0$. 
\rev{The following lemma collects the properties of these iterates from \cite{CHFP,fp_ifb05}. 
We remark that, compared to \cite{CHFP,fp_ifb05}, the results are stated in a simplified (but equivalent) form, which is more suitable for the subsequent analysis.}
\begin{lem}\label{lem_xch}
Suppose ${\mathcal E}(u^{\varepsilon}_0) \leq C$.
Let $\bigl\{ (X^{j}_{\tt CH}, w^{j}_{\tt CH})\bigr\}_{j=0}^J \subset \bigl[ {\mathbb H}^1\bigr]^2$ be the solution of Scheme \ref{scheme_time} for $g \equiv 0$. For every 
$0 <\beta < \frac{1}{2}$, $\varepsilon \in (0, \eps_0) $, $k \leq \eps^3$, 
and ${\mathfrak p}_{\tt CH} >0$, there exist
${\mathfrak m}_{\tt CH}, {\mathfrak n}_{\tt CH}, C>0$, {and ${\mathfrak l}_{\tt CH} \geq 3$} such that
\begin{eqnarray*}
\ \quad \quad {\rm i)} & \quad& \D \max_{1 \leq j \leq J}{\mathcal E}(X^j_{\tt CH}) \leq 
 {\mathcal E}(u_0^{\varepsilon})\,.   \quad \qquad \qquad \qquad \qquad \qquad \qquad \qquad \qquad \qquad \qquad \quad \,   
\end{eqnarray*}
Assume moreover $\|u_0^{\varepsilon}\|_{\mathbb{H}^2} \leq C\eps^{-\mathfrak{p}_{\tt CH}}$, then
\begin{eqnarray*}
\ \quad \quad {\rm ii)} & \quad & \D \max_{1 \leq j \leq J} \|X^{j}_{\tt CH}\|_{\mathbb{H}^2} \leq C \eps^{-\mathfrak{n}_{\tt CH}}\,,
\quad \qquad \qquad \qquad \qquad \qquad \qquad \qquad \qquad \qquad \qquad \,\,\,
\\
\ \quad \quad {\rm iii)} & \quad & \D \max_{1 \leq j \leq J} \|X^{j}_{\tt CH}\|_{\mathbb{L}^\infty} \leq C\,
\quad\mathrm{for} \quad k \leq C \varepsilon^{{\mathfrak l}_{\tt CH}}.
\qquad \qquad \qquad \qquad\,\,\,
\end{eqnarray*}
{Assume in addition  
$\|u_0^{\varepsilon}\|_{\mathbb{H}^3} \leq C\eps^{-\mathfrak{p}_{\tt CH}}$.}  Then for $k \leq C \varepsilon^{{\mathfrak l}_{\tt CH}}$, and $C_0 >0$ from
(\ref{genspec}) it holds
\begin{eqnarray*}
{\rm iv)} \quad& \D \max_{1 \leq j \leq J} \Vert u_{\tt CH}(t_j) - X^j_{\tt CH}\Vert^2_{{\mathbb H}^{-1}} 
  + \sum_{j=1}^J k^{1+\beta} \big\Vert \nabla \bigl[u_{\tt CH}(t_j) - X^j_{\tt CH} \bigr]\big\Vert^2 \leq C  \frac{k^{{2}-\beta}}{ \varepsilon^{{\mathfrak m}_{\tt CH}}}\, , \\
{\rm v)} \quad& \D \inf_{0\leq t\leq T}\inf_{\psi\in
\mathbb{H}^{1}, \; w=\deltainv\psi} \frac{\eps\|\nabla\psi\|^2+\frac{1-\eps^3}{\eps}\bigl(f'(X^j_{\tt CH})\psi,\psi \bigr)}{\|\nabla w\|^2}\geq -(1-\eps^3)(C_0+1)\, .
\end{eqnarray*}
\end{lem}
\begin{proof}
The proof of i), ii), iv), v) is a direct consequence of \cite[Lemma~3, Corollary~1, Proposition~2]{CHFP}.

To show iii), we use the Gagliardo-Nirenberg inequality and \cite[inequality (76)]{CHFP},
ii),~iv) to get the following $\mathbb{L}^{\infty}$-error estimate
for $k \leq C\eps^{\mathfrak{l}_{\tt CH}}$, and some $\mathfrak{l}_{\tt CH}>0$,
$$
\max_{1 \leq j \leq J}\|\XCH^j - u_{\tt CH}(t_j)\|_{\mathbb{L}^\infty} \leq \eps^2\, .
$$
Hence, 
$\|\XCH^j\|_{\mathbb{L}^\infty}\leq C$
since {$\|u_{\tt CH}\|_{\mathbb{L}^\infty}\leq C$}; cf.~\cite[proof of Theorem.~2.3]{abc} and \cite[Lemma~2.2]{fp_ifb05}.
\qed
\end{proof}
%
%
The numerical solution of Scheme \ref{scheme_time}
satisfies the discrete counterpart of the energy estimate in Lemma \ref{lem_ener_cont}, i).
The time-step constraint in the lemma below is a consequence of the implicit treatment of the nonlinearity; see the last term in (\ref{fest0}),
its estimate (\ref{est_l2inc}), and (\ref{est_step1}); the lower bound for admissible $\gamma$ has the same origin.

\begin{lem}\label{lem_energy}
Let $\gamma > \frac{3}{2}$, $\eps \in (0,\eps_0)$
and $k \leq \varepsilon^3$. Then the solution of Scheme \ref{scheme_time} conserves mass along every path $\omega\in\Omega$, and there exists $C > 0$ such that
\begin{itemize}
\item[i)]\quad
$\D \max_{1\leq j\leq J} \mathbb{E}\bigl[ {\mathcal E}(X^j)\bigr]  
+ \frac{k}{2}  \sum_{i=1}^J\mathbb{E}\bigl[\|\nabla w^i\|^2\bigr]
\leq
C \,\bigl(   {{\mathcal E}(u^\varepsilon_0)} +1\bigr)\,,$
\item[ii)]\quad
$\D \mathbb{E}\big[\max_{1\leq j\leq J}{\mathcal E}(X^j) 
\big] 
\leq
C \bigl(  {\mathcal E}(u^\varepsilon_0)  + 1\bigr)\,.$
\end{itemize}
For every $p = 2^r$, $r \in {\mathbb N}$, there exists $C \equiv C(p, T) > 0$ such  that
\begin{itemize}
\item[iii)]\quad $
\D \max_{1\leq j\leq J} \mathbb{E}\bigl[ \vert{\mathcal E}(X^j)\vert^p\bigr] 
\leq C 
\D \bigl( \vert {\mathcal E}(u^\varepsilon_0)\vert^p +1\bigr)\, ,$
\item[iv)]\quad $
 \D \mathbb{E}\big[ \max_{1\leq j\leq J} \vert{\mathcal E}(X^j)\vert^p\big] \leq C 
\bigl(  \vert{\mathcal E}(u^\varepsilon_0)\vert^p +1\bigr)\, .$
\end{itemize}
\end{lem}
\begin{proof} 
{\bf i)} For  $\omega \in \Omega$ fixed,
we choose $\varphi=w^j(\omega)$ and $\psi=[X^j-X^{j-1}](\omega)$ in
Scheme \ref{scheme_time}. Adding both equations then leads to ${\mathbb P}$-a.s.
\begin{equation}\label{test1}
\begin{array}{lll}
\displaystyle 
\revised{\frac{\eps}{2} \|\nabla X^j\|^2 - \frac{\eps}{2} \|\nabla X^{j-1}\|^2}
 + \frac{\eps}{2} \|\nabla[X^j - X^{j-1}]\|^2 + k\|\nabla w^j\|^2
\\
 \displaystyle \qquad + \frac{1}{\eps}\bigl(f(X^j), X^j-X^{j-1}\bigr)
= \eps^\gamma(g, w^j)\Delta_jW\, .
\end{array}
\end{equation}
{Note that the third term on the left-hand side reflects the numerical dissipativity in the scheme.}
We can estimate the nonlinear term as (cf.~\cite[Section~3.1]{ACFP}), 
\begin{eqnarray}\nonumber
 \bigl(f(X^j), X^j-X^{j-1} \bigr) &\geq& 
\revised{\frac{1}{4} \|{\mathfrak f}(X^j)\|^2-\frac{1}{4}\|{\mathfrak f}(X^{j-1})\|^2}
 \\ \label{fest0}
&& + \frac{1}{4}\|{\mathfrak f}(X^{j})-{\mathfrak f}(X^{j-1})\|^2
- \frac{1}{2}\|X^j-X^{j-1}\|^2\, ,
\end{eqnarray}
where we employ the notation ${\mathfrak f}(u) := |u|^2 -1$, i.e., $f(X^j)= {\mathfrak f}(X^j)X^j$. The \revised{third} term on the right-hand side again reflects
 numerical dissipativity.

By  $\omega \in \Omega$ fixed, and $\varphi = \deltainv[X^j-X^{j-1}](\omega)$
in Scheme \ref{scheme_time}, we eventually have ${\mathbb P}$-a.s.,
$$
\|\Delta^{-1/2}[X^j-X^{j-1}]\|^2\leq \Big(k\|\nabla w^j\| + \eps^{\gamma}\|\Delta^{-1/2} g\||\Delta_j W |\Big)\|\Delta^{-1/2}[X^j-X^{j-1}]\| \, ,
$$
which together with $\|\Delta^{-1/2} g\|\leq C$ yields the estimate
\begin{equation*}\label{est_inc0}
\|\Delta^{-1/2}[X^j-X^{j-1}]\|^2\leq  2k^2\|\nabla w^j\|^2  + C\eps^{2\gamma}|\Delta_j W|^2\, .
\end{equation*}
Hence, using this estimate, {and exploiting again the inherent numerical
dissipation of the scheme} we can estimate
\begin{equation}\label{est_l2inc}
\begin{split}
\frac{1}{2\eps}\|X^j-X^{j-1}\|^2 & = \frac{1}{2\eps}\bigl(\nabla\deltainv[X^j-X^{j-1}], \nabla [X^j-X^{j-1}] \bigr) 
\\
& \leq \frac{1}{4\eps^3}\|\Delta^{-1/2}[{X^j-X^{j-1}}]\|^2 + \frac{\eps}{4}\|\nabla[X^j-X^{j-1}]\|^2
\\
& \leq 
\frac{k^2}{2 \eps^3}\|\nabla w^j\|^2  + C\eps^{2\gamma-3}|\Delta_j W|^2 + \frac{\eps}{4}\|\nabla[X^j-X^{j-1}]\|^2\, .
\end{split}
\end{equation}
We substitute (\ref{fest0}) along with the last inequality into (\ref{test1})
and get
\begin{equation}\label{est_step1}
\begin{split}
& \frac{\eps}{2} \bigl(\|\nabla X^j\|^2 - \|\nabla X^{j-1}\|^2\bigr) + \frac{\eps}{4} \|\nabla[X^j - X^{j-1}]\|^2 
\\ 
 &\quad +\frac{1}{4\eps} \Bigl(\|{\mathfrak f}(X^j)\|^2- \|{\mathfrak f}(X^{j-1})\|^2 + \|{\mathfrak f}(X^{j})-{\mathfrak f}(X^{j-1})\|^2\Bigr) + \big(k-\frac{k^2}{2 \eps^3}\big)\|\nabla w^j\|^2
\\
&\qquad   
\rev{\leq} \eps^\gamma(g, w^j)\Delta_jW + C\eps^{2\gamma-3}|\Delta_j W|^2\, ,
\end{split}
\end{equation}
which motivates time-steps $k < 2 \varepsilon^3$.
Next, by using the second equation in Scheme \ref{scheme_time}, we can rewrite the first term on the right-hand side as
\begin{equation}\label{lc3}
\begin{split}
\varepsilon^{\gamma} \bigl( g, w^j\bigr) \Delta_j W 
&= \varepsilon^{\gamma+1}
\Bigl[\bigl( \nabla [X^j-X^{j-1}], \nabla g\bigr)  + \bigl( \nabla X^{j-1}, \nabla g\bigr)
\Bigr]\Delta_j W \\
&\qquad  + \varepsilon^{\gamma-1} \Bigl[ \bigl( f(X^j) - f(X^{j-1}), g\bigr) + \bigl( f(X^{j-1}), g\bigr)
 \Bigr]\Delta_j W\,
\\
&=: A_1+A_2+A_3+A_4\, .
 \end{split}
\end{equation}
Note that $\mathbb{E}[A_2]= \mathbb{E}[A_4] = 0$.
Next, we obtain
\begin{equation}\label{est_a1}
\begin{split}
A_1 = \eps^{\gamma+1}\bigl( \nabla [X^j-X^{j-1}], \nabla g\bigr) \Delta_j W 
& \leq 
\frac{\eps}{8}\|\nabla [X^j-X^{j-1}]\|^2 + C\varepsilon^{2\gamma+1}\|\nabla g\|^2|\Delta_j W|^2
\\
& \leq 
\frac{\eps}{8}\|\nabla [X^j-X^{j-1}]\|^2  + C\varepsilon^{2\gamma+1}|\Delta_j W|^2\, .
\end{split}
\end{equation}
On recalling $f(X^j)= {\mathfrak f}(X^j)X^j$, we rewrite the remaining term as
\begin{equation}\label{est_a3}
\begin{split}
A_3  & =  \varepsilon^{\gamma-1} \bigl( f(X^j) - f(X^{j-1}), g\bigr)\Delta_j W
\\
& = \varepsilon^{\gamma-1} \Big( \big[ {\mathfrak f}(X^j) - {\mathfrak f}(X^{j-1})\big] X^{j}, g\Big)\Delta_j W
 + \varepsilon^{\gamma-1} \Big( {\mathfrak f}(X^{j-1})\big[X^{j}-X^{j-1}\big], g\Big)\Delta_j W
\\
&
=: A_{3,1}+A_{3,2}\, .
\end{split}
\end{equation}
Thanks to the embeddings ${\mathbb L}^s \hookrightarrow  {\mathbb L}^r$ ($r \leq s$),
and the Cauchy-Schwarz and Young's inequalities,
\begin{equation}\nonumber
\begin{split}
A_{3,1}
& \leq \frac{1}{16\eps}\|{\mathfrak f}(X^j) - {\mathfrak f}(X^{j-1})\|^2 + C\varepsilon^{2\gamma-1}\||\rev{X^{j}|^2} \|_{\mathbb{L}^1}\|g\|^2_{\mathbb{L}^\infty} |\Delta_j W|^2 
\\
& \leq \frac{1}{16\eps}\|{\mathfrak f}(X^j) - {\mathfrak f}(X^{j-1})\|^2 + C\varepsilon^{2\gamma-1}\Big(\|{\mathfrak f}(X^j) - {\mathfrak f}(X^{j-1})\|_{\mathbb{L}^1} + \|X^{j-1} \|^2\Big) |\Delta_j W|^2
\\
& \leq
\frac{1}{8\eps}\|{\mathfrak f}(X^j) - {\mathfrak f}(X^{j-1})\|^2 + C\varepsilon^{4\gamma-1}|\Delta_j W|^4 + C\varepsilon^{2\gamma-1}\Big(\|{\mathfrak f}(X^{j-1})\|^2+1\Big)|\Delta_j W|^2\,.
\end{split}
\end{equation}
The leading term may now be controlled by the numerical dissipation term
in (\ref{fest0}). Finally, by \revised{the} Poincar\'e's inequality, we estimate
\begin{equation}\nonumber
\begin{split}
A_{3,2} & \leq \|{\mathfrak f}(X^{j-1})\|^2\|g\|^2_{\mathbb{L}^\infty}|\Delta_j W|^2
+ \varepsilon^{2\gamma-2}\|X^{j}-X^{j-1}\|^2
\\
& \leq C \|{\mathfrak f}(X^{j-1})\|^2|\Delta_j W|^2
+ C_{\mathcal{D}}\varepsilon^{2\gamma-2}\|\nabla[X^{j}-X^{j-1}]\|^2\, .
\end{split}
\end{equation}
By combining the above estimates for $A_{3,1}$, $A_{3,2}$ we obtain an estimate for (\ref{est_a3}).

Next, we insert the estimates  (\ref{lc3}), (\ref{est_a1}), and (\ref{est_a3}) into (\ref{est_step1}), account for $2\gamma - 2 > 1$, sum the resulting inequality over $j$  and
take expectations,
\begin{equation}\label{est_step2}
\begin{split}
& \mathbb{E}\bigl[\frac{\eps}{2} \|\nabla X^j\|^2 
 +\frac{1}{4\eps} \|{\mathfrak f}(X^j)\|^2 \bigr] + \frac{1}{8\eps}\sum_{i=1}^j\mathbb{E}\bigl[\|{\mathfrak f}(X^{i})-{\mathfrak f}(X^{i-1})\|^2\bigr]
\\
&\qquad  + \Big(\frac{\eps}{8} - C_{\mathcal{D}}\eps^{2\gamma-2}\Big)\sum_{i=1}^j\mathbb{E}\bigl[\|\nabla[X^i - X^{i-1}]\|^2\bigr] + \big(k-\frac{k^2}{2 \eps^3}\Big)\sum_{i=1}^j\mathbb{E}\bigl[\|\nabla w^i\|^2\bigr]
\\
&\quad  \leq
\mathbb{E}\bigl[\frac{\eps}{2}\|\nabla X^0\|^2 
 +\frac{1}{4\eps} \|{\mathfrak f}(X^0)\|^2\bigr] +
CT\bigl(\varepsilon^{4\gamma-1}k + \rev{\varepsilon^{2\gamma+1}} + \revdim{\varepsilon^{2\gamma-1}} + \varepsilon^{2\gamma-3}\bigr) \\
&\qquad + C(1+\eps^{2\gamma-1}){k}\sum_{i=0}^{j-1} \mathbb{E}\big[\|{\mathfrak f}(X^{i})\|^2\big]\, .
\end{split}
\end{equation}
On noting that $\|F(u)\|_{{\mathbb L}^1} = \frac{1}{4}\|{\mathfrak f}(u)\|^2$, 
assertion i) now follows with the help
of the discrete Gronwall  lemma. \\

{\bf ii)}
The second estimate can be shown along the lines of the first part of the proof
by applying $\max_{j}$ before taking the expectation in (\ref{est_step2}).
The additional term that arises from the terms $A_2$, $A_4$ in (\ref{lc3}) can be rewritten by
using the second equation in Scheme \ref{scheme_time},
\begin{equation}\label{term_bdg}
\begin{array}{lll}
\D \mathbb{E}\Big[\max_{1\leq i\leq j}\Big| \sum_{\ell=1}^i\Big\{
\varepsilon^{\gamma-1} \bigl( f(X^{\ell-1}), g\bigr)
+
\varepsilon^{\gamma+1} \bigl( \nabla X^{\ell-1}, \nabla g\bigr)\Big\} \Delta_\ell W \Big|\Big]
\\
\D \qquad =
\mathbb{E}\Big[\max_{1\leq i\leq j}\Big| \sum_{\ell=1}^i
\varepsilon^{\gamma} \bigl(w^{\ell-1}, g\bigr) \Delta_\ell W \Big|\Big]
=
\mathbb{E}\Big[\max_{1\leq i\leq j}\Big| \sum_{\ell=1}^i
\varepsilon^{\gamma} \bigl(\overline{w}^{\ell-1}, g\bigr) \Delta_\ell W \Big|\Big]
\\
\D \qquad \leq \mathbb{E}\Big[\max_{1\leq i\leq j}\Big| \sum_{\ell=1}^i
\varepsilon^{\gamma} \bigl(\overline{w}^{\ell-1}, g\bigr) \Delta_\ell W \Big|^2\Big]^{1/2}
\, ,
\end{array}
\end{equation}
where the equality in the second line follows from the zero mean property of the noise.

The last sum in (\ref{term_bdg}) is a discrete square-integrable martingale, and by the independence properties of the summands,
the Poincar\'e inequality and the energy estimate i)  we have
\begin{eqnarray*}
\D \mathbb{E}\Big[\Big( \sum_{\ell=1}^i
\varepsilon^{\gamma} \bigl(\overline{w}^{\ell-1}, g\bigr) \Delta_\ell W \Big)^2\Big]
&=& \varepsilon^{2\gamma}\mathbb{E}\Big[k\sum_{\ell=1}^i
 \bigl(\overline{w}^{\ell-1}, g\bigr)^2 \Big]
\\
&\leq& C_{\mathcal{D}} \varepsilon^{2\gamma} \mathbb{E}\Big[k\sum_{\ell=1}^i
 \bigl\|\nabla w^{\ell-1}\|^2 \|g\|^2_{\mathbb{L}^\infty} \Big] \leq C \varepsilon^{2\gamma}.
\end{eqnarray*}
Therefore, (\ref{term_bdg}) can be estimated using the discrete BDG-inequality (see Lemma~\ref{lembdg}) and part~{i)} by
\begin{eqnarray*}
\leq
\D C\varepsilon^{\gamma}\|g\|_{\mathbb{L}^\infty} \mathbb{E}\Big[k\sum_{\ell=1}^J
 \bigl\|\overline{w}^{\ell-1}\|^2\Big]^{1/2}
\leq \D C \varepsilon^{\gamma} \mathbb{E}\Big[\sum_{\ell=1}^J  k\bigl\|\nabla w^{\ell-1}\|^2 \Big]^{1/2} \leq C\eps^{\gamma}\, .
\end{eqnarray*}

{\bf iii)} We show assertion iii) for $p=2^1$. By collecting \rev{the estimates of the terms in (\ref{lc3}) in part~i)
(cf. (\ref{est_a1}), \ref{est_a3}))} we deduce from {(\ref{est_step1})} that
\begin{eqnarray}\nonumber
&&\mathcal{E}(X^j) - \mathcal{E}(X^{j-1}) + \frac{\varepsilon}{4}
\Vert \nabla [X^{j} - X^{j-1}]\Vert^2 + \frac{1}{4\varepsilon}
\Vert \mathfrak{f}(X^j) -  \mathfrak{f}(X^{j-1})\Vert^2 + \frac{k}{2} \Vert \nabla w^j\Vert^2 \\ \label{ineq1}
&&\quad \leq C \Bigl( \varepsilon {\mathcal E}(X^{j-1}) + 1\Bigr) \vert \Delta_j W\vert^2 +  
C \varepsilon^{4\gamma-1}
\vert \Delta_j W\vert^4 + \rev{C (\varepsilon^{2\gamma+1} + \varepsilon^{2\gamma-3})\vert \Delta_j W\vert^2}  \\ \nonumber
&& \qquad + \varepsilon^{\gamma+1} (\nabla X^{j-1}, \nabla g) \Delta_j W +
\varepsilon^{\gamma-1} \bigl( f(X^{j-1}), g\bigr)\Delta_j W\, .
\end{eqnarray}
Multiply this inequality with $\mathcal{E}(X^j)$  and use the identity $(a-b)a = \frac{1}{2} [ a^2  - b^2 + (a-b)^2]$, \rev{the estimate $\varepsilon^{2\gamma+1} \leq \eps_0^{4}\varepsilon^{2\gamma-3}$}, 
Young's inequality, and the generalized H\"older's inequality to conclude\rev{
\begin{eqnarray}\nonumber
&&\frac{1}{2} \Bigl[  \vert {\mathcal E}(X^j)\vert^2 - \vert {\mathcal E}(X^{j-1})\vert^2
+ \vert {\mathcal E}(X^j) - {\mathcal E}(X^{j-1})\vert^2\Bigr] 
{+ \frac{\varepsilon}{4}
\Vert \nabla [X^{j} - X^{j-1}]\Vert^2 {\mathcal E}(X^j)}
\\ \nonumber
&&\quad \leq C \Bigl( \varepsilon \vert{\mathcal E}(X^{j-1})\vert^2 + {\mathcal E}(X^{j-1})\Bigr) \vert \Delta_j W\vert^2 
+  C \varepsilon^{2\gamma-3} {\mathcal E}(X^{j-1}) \vert \Delta_j W\vert^2
\\ \label{hmom1}
&&\qquad  
+  C\Bigl(\varepsilon^2\vert{\mathcal E}(X^{j-1})\vert^2 + 1 + \varepsilon^{4\gamma-1}{\mathcal E}(X^{j-1}) + \varepsilon^{2(2\gamma-3)} \Bigr) \vert \Delta_j W\vert^4    
+ C \varepsilon^{2(4\gamma-1)} \vert \Delta_j W\vert^8 
\\
\nonumber && \qquad
+ \frac{1}{4} \bigl\vert {\mathcal E}(X^j) -  {\mathcal E}(X^{j-1})\bigr\vert^2\\ \nonumber
&&\qquad + \Bigl[\varepsilon^{\gamma+1} (\nabla X^{j-1}, \nabla g) \Delta_j W +
\varepsilon^{\gamma-1} \bigl( f(X^{j-1}), g\bigr)\Delta_j W\Bigr] {\mathcal E}(X^{j-1}) 
\\
\nonumber
&&\qquad + C \max \bigl\{ \Vert \nabla g\Vert^2, \Vert g\Vert^2_{{\mathbb L}^{\infty}}
\bigr\}\Bigl[\varepsilon^{2(\gamma+1)} \Vert \nabla X^{j-1}\Vert^2  + \varepsilon^{2(\gamma-1)} \Vert \mathfrak{f}(X^{j-1})\Vert^2
\Vert X^{j-1}\Vert^2 \Bigr] \vert \Delta_j W\vert^2\, .
\end{eqnarray}
We note that to get the above estimate we employed 
the reformulation ${\mathcal E}(X^{j}) = {\mathcal E}(X^{j-1}) + ({\mathcal E}(X^{j})-{\mathcal E}(X^{j-1}))$ on the right-hand side}.

By Poincar\'e's inequality, the last term in (\ref{hmom1}) may be bounded as
\begin{equation*}\label{hmom1a}
\varepsilon^{2(\gamma-1)} \Bigl[ \eps^4 \Vert \nabla X^{j-1}\Vert^2  + \Vert \mathfrak{f}(X^{j-1})\Vert^2
\Vert X^{j-1}\Vert^2 \Bigr] \vert \Delta_j W\vert^2 
\leq C \varepsilon^{2(\gamma-1)} \Bigl[\rev{\varepsilon^{3}} {\mathcal E}\bigl(X^{j-1}\bigr)  +  \bigl\vert{\mathcal E}(X^{j-1})\bigr\vert^2 \Bigr] \vert \Delta_j W\vert^2\, .
\end{equation*}

After summing-up in (\ref{hmom1}) and  taking expectations
we get for any $j\leq J$ that
\begin{equation}\label{hmom3}
\begin{array}{llll}
\D \frac{1}{2}\mathbb{E}\big[\mathcal{E}(X^j)^2\big] 
+ \frac{1}{4}\sum_{i=1}^j\mathbb{E}\big[\big|\mathcal{E}(X^i) - \mathcal{E}(X^{i-1})\big|^2\big]
\\
\D\quad \leq \D 
\frac{1}{2}\mathbb{E}\big[\mathcal{E}(X^{0})^2\big] + C t_j + C \rev{(\eps^{2\gamma-3}+1+\eps^{4\gamma-1}k)} k\sum_{i=0}^{j-1} \mathbb{E}\big[\mathcal{E}(X^{i})] 
\\
\quad\quad
\D + C \rev{(\eps^{2(\gamma-1)}+\eps+ \eps^2k)}  k\sum_{i=0}^{j-1}\mathbb{E}\big[\mathcal{E}(X^{i})^2\big],
\end{array}
\end{equation}
where the third term is bounded via (\ref{est_step2}) in part {\bf ii)}, and
the statement then follows from the discrete Gronwall inequality.



For $p = 2^r$, $r=2$, we may now argue correspondingly: we start with
(\ref{hmom1}), which we now multiply with $\vert{\mathcal E}(X^j)\vert^2$.
Assertion iii) now follows via induction with respect to~$r$.\\

{\bf iv)} 
The last estimate follows \revised{analogously} to ii) from the BDG-inequality and iii).
\qed
\end{proof}

\medskip
{The error analysis of the implicit Scheme \ref{scheme_time} in the subsequent
Section \ref{sec_err} involves the use of a stopping index $J_\eps$, and
an associated random variable $\mathbbm{1}_{\{j \leq J_\eps\}}$ that is
{ measurable} w.r.t.~the $\sigma$-algebra $\mathcal{F}_{t_j}$, but not w.r.t.~$\mathcal{F}_{t_{j-1}}$.
This issue prohibits the use of the standard BDG-inequality since  
{$\mathbbm{1}_{\{j \leq J_\eps\}}$} is not independent of the Wiener increment $\Delta_j W$. The following lemma contains
a discrete BDG-inequality which will be used in Section \ref{sec_err}.
We take $\{ {\mathcal F}_{t_j}\}_{j=0}^J$ to be a discrete filtration associated with the
time mesh $\{ t_j\}_{j=0}^J \subset [0,T]$ on $(\Omega, {\mathcal F}, {\mathbb P})$.}
\begin{lem}\label{lembdg}
For every $j=1,\dots, J$, 
let $F_{j}$ be an $\mathcal{F}_{t_j}$-measurable random variable, and {$\Delta_jW$ be independent of $F_{j-1}$}.
Assume that the $\{{\mathcal F}_{t_j}\}_j$-martingale
$G_{\mfrakj } := \sum_{j=1}^{{\mfrakj }}F_{j-1}\Delta_j W$ ($1 \leq {\mfrakj } \leq J$), with $G_0=0$
be square-integrable.
Then for any stopping index $\tau: \Omega \rightarrow {\mathbb N}_0$ such that $\mathbbm{1}_{\{j\leq \tau\}}$
is $\mathcal{F}_{t_j}$-measurable,
it holds that
\begin{equation*}\label{bdgdin*}
\mathbb{E}\Big[\max_{{\mfrakj }=1,\dots, {\tau \wedge J}}\big\vert \sum_{j=1}^{{\mfrakj }}F_{j-1}\Delta_j W\big\vert^2\Big]
\leq 4\mathbb{E}\Big[\sum_{j=1}^{({\tau }+1)\wedge J}kF_{j-1}^2\Big]\, ,
\end{equation*}
\revised{where $\tau \wedge J = \min\{\tau,J\}$.}
\end{lem}
{\begin{proof}
We start by noting that
\begin{equation*}\label{jeps1}
\sum_{j=1}^{(\tau+1)\wedge \mfrakj } F_{j-1} \Delta_j W = \sum_{j=1}^{\mfrakj } \mathbbm{1}_{\{j-1\leq \tau\}} F_{j-1} \Delta_j W \qquad (1 \leq \mfrakj \leq J)\,.
\end{equation*}
With this identity, we obtain
\begin{align}\label{bdgest1}
\mathbb{E}\Big[\max_{\mfrakj =1,\dots,\tau\wedge J}\big\vert\sum_{j=1}^{\mfrakj } F_{j-1} \Delta_j W\big\vert^2\Big]
& \leq
\mathbb{E}\Big[\max_{\mfrakj =1,\dots, (\tau+1)\wedge J}\big\vert\sum_{j=1}^{\mfrakj } F_{j-1} \Delta_j W\big\vert^2\Big]
\\ \nonumber
 & = 
\mathbb{E}\Big[\max_{\mfrakj =1,\dots, J}\big\vert\sum_{j=1}^{\mfrakj } \mathbbm{1}_{\{j-1\leq \tau\}} F_{j-1} \Delta_j W\big\vert^2\Big]\, .
\end{align}
The random variable $\mathbbm{1}_{\{j-1\leq \tau\}}$ is {$\mathcal{F}_{t_{j-1}}$-measurable},
therefore, 
${G}_\mfrakj  := \sum_{j=1}^{\mfrakj } \mathbbm{1}_{\{j-1\leq \tau\}} F_{j-1} \Delta_jW$ is also a discrete square-integrable martingale.
Hence, by the $L^2$-maximum martingale inequality,
using the independence of $\mathbbm{1}_{\{j\leq \tau\}} F_{j}$ and $\Delta_\mfrakj W$ for $j < \mfrakj $
it follows that
\begin{eqnarray}
&& \mathbb{E}\Big[\max_{\mfrakj =1,\dots,J}\big\vert\sum_{j=1}^{\mfrakj } \mathbbm{1}_{\{j-1\leq \tau\}} F_{j-1} \Delta_j W\big\vert^2\Big]
\leq 4 \mathbb{E}\Big[ \big\vert\sum_{j=1}^{J}\mathbbm{1}_{\{j-1\leq \tau\}} F_{j-1} \Delta_j W\big \vert^2\Big]
\nonumber
\\
\nonumber
&&\leq 4\mathbb{E} \Big[\sum_{j=1}^{J} (\mathbbm{1}_{\{j-1\leq \tau\}} F_{j-1})^2 |\Delta_j W|^2\Big]  + 
 8\sum_{i,j=1;i<j}^{J}  \mathbb{E} \big[\mathbbm{1}_{\{i-1\leq \tau\}} F_{i-1} \mathbbm{1}_{\{j-1\leq \tau\}} F_{j-1} \Delta_i W\big] {\mathbb E}\big[ \Delta_j W\big]
\\ \label{bdgest3}
&&\qquad  = 4\sum_{j=1}^{J} \mathbb{E} \Big[(\mathbbm{1}_{\{j-1\leq \tau\}} F_{j-1})^2\Big] \mathbb{E} \Big[ |\Delta_j W|^2\Big]  
= 4\mathbb{E}\Big[\sum_{j=1}^{(\tau+1)\wedge J} F_{j-1}^2 k\Big]\, .
\end{eqnarray}
The assertion of the lemma
then follows from (\ref{bdgest1}) and (\ref{bdgest3}).
\qed
\end{proof}
}

\subsection{Error analysis}\label{sec_err}


Denote $Z^j:=X^j-\XCH^j$,  use Scheme \ref{scheme_time} for a fixed $\omega \in \Omega$,   and choose $\varphi = \deltainv Z^{j}(\omega)$, $\psi = Z^j(\omega)$.
We obtain $\mathbb{P}$-a.s.
\begin{equation}\label{dim13}
\begin{split}
\frac{1}{2}&\Bigl(\|\Delta^{-1/2}Z^j\|^2- \|\Delta^{-1/2}Z^{j-1}\|^2+
\|\Delta^{-1/2}[Z^j-Z^{j-1}]\|^2\Bigr) + k\eps\|\nabla Z^j\|^2\\
&\quad  +\frac{k}{\eps}\bigl(f(X^j)-f(\XCH^j),Z^j\bigr)
=\eps^\gamma(\Delta^{-1/2}g,\Delta^{-1/2}Z^j)\Delta_jW\, .
\end{split}
\end{equation}


{We use Lemma~\ref{lem_xch}, v) to obtain a first error bound.}
\begin{lem}\label{thma2}
Assume $\gamma > \frac{3}{2}$, \revised{$\Vert u^\varepsilon_0\Vert_{{\mathbb H}^3} \leq C \varepsilon^{-{\mathfrak p}_{\tt CH}}$ for $\varepsilon \in (0,\varepsilon_0)$, 
and let $k \leq C \varepsilon^{{\mathfrak l}_{\tt CH}}$} with ${\mathfrak l}_{\tt CH} \geq 3$ from Lemma \ref{lem_xch} be sufficiently small.
%
%
There exists $C>0$,
such that
$\mathbb{P}$-a.s.~and for all  $1\leq {\mfrakj } \leq J$,
\begin{eqnarray}\nonumber
&&\max_{1\leq j\leq {\mfrakj }}\|\Delta^{-1/2}Z^j\|^2+{\eps^4 k}\sum_{j=1}^{{\mfrakj }}\|\nabla Z^j\|^2\\ \label{error_eqn} 
&&\quad \leq 
\frac{Ck}{\eps}\sum_{j=1}^{{\mfrakj }}\|Z^j\|_{\mathbb{L}^3}^3+
C\eps^\gamma\max_{1\leq j\leq {\mfrakj }}|\sum_{i=1}^{j}(\deltainv g,Z^{i-1})\Delta_i W|
+C\eps^{2\gamma} \sum_{j=1}^{{\mfrakj }}|\Delta_j W|^2\, .
\end{eqnarray}
\end{lem}
\begin{proof}
%
{\bf 1.} Consider the last term on the left-hand side of (\ref{dim13}).
On recalling $Z^j=X^j-\XCH^j$,
by a property of $f$, see \rev{\cite[eq. (2.6)]{fp_ifb05}}, and Lemma~\ref{lem_xch},~iii), we get for some $C>0$
\begin{equation}\label{didi1}
\begin{split}
&\bigl(f(X^j)-f(\XCH^j),Z^j\bigr)=\bigl(f(\XCH^j)-f(X^j),\XCH^j-X^j\bigr)\\
&\qquad 
\geq 
\bigl(f'(\XCH^j)[\XCH^j-X^j],\XCH^j-X^j \bigr)-3\bigl(\XCH^j \vert \XCH^j-X^j \vert^2,\XCH^j-X^j\bigr)\\
&\qquad \geq (1-\eps^3)\bigl(f'(\XCH^j)Z^j,Z^j\bigr)-C\|Z^j\|_{\mathbb{L}^3}^3 + \eps^3\bigl(f'(\XCH^j)Z^j,Z^j\bigr)\, .
\end{split}
\end{equation}

{\bf 2.} In order to later keep a portion of $\|\nabla Z^j\|^2$ on the left-hand side of (\ref{dim13}) we use the identity
\begin{equation}\label{didi1a}
\begin{array}{lll}
\D \eps \|\nabla Z^j\|^2 + \frac{(1-\eps^3)}{\eps}\bigl(f'(\XCH^j)Z^j,Z^j\bigr)
\\
\qquad =\D
(1-\eps^3) \left(\eps \|\nabla Z^j\|^2 + \frac{(1-\eps^3)}{\eps}\bigl(f'(\XCH^j)Z^j,Z^j\bigr)\right)
 \\
\quad\qquad \D + \eps^3 \left(\eps \|\nabla Z^j\|^2 + \frac{(1-\eps^3)}{\eps}\bigl(f'(\XCH^j)Z^j,Z^j\bigr)\right)\,.
\end{array}
\end{equation}
We apply Lemma~\ref{lem_xch}, v) to get a lower bound for the first term on the 
right-hand side,
$$
\geq -(C_0+1)\|\Delta^{-1/2}Z^j\|^2_{\mathbb{L}^2}\,.
$$
On noting $\eps<1$, we estimate the remaining nonlinearities in (\ref{didi1a}) using Lemma~\ref{lem_xch}, iii),
$$
\eps^2 \bigl(f'(\XCH^j)Z^j,Z^j\bigr) \leq {C\varepsilon^2} \|\nabla Z^j\| \|\Delta^{-1/2}Z^j\|
\leq \frac{\eps^4}{4}\|\nabla Z^j\|^2 + C\|\Delta^{-1/2}Z^j\|^2.
$$

{\bf 3.} We insert the estimates from the steps {\bf 1.} and {\bf 2.} into \eqref{dim13}, and use the bound {
\begin{equation}\label{bdg-arg}
\eps^{\gamma} (\deltainv g, Z^j - Z^{j-1})\Delta_j W \leq
\frac{1}{4} \Vert \Delta^{-1/2} [Z^j - Z^{j-1}]\Vert^2 + \eps^{2\gamma} \vert \Delta_j W\vert^2 \Vert \Delta^{-1/2} g\Vert^2
\end{equation}} to validate
\begin{eqnarray*}\label{didi3}
&&\frac{1}{2}\Bigl(\|\Delta^{-1/2}Z^j\|^2-\|\Delta^{-1/2}Z^{j-1}\|^2+
 \frac{1}{2}\|\Delta^{-1/2}[Z^j-Z^{j-1}]\|^2 + \frac{\eps^4}{4} k \|\nabla Z^j\|^2\Bigr) \\
&&\quad \leq Ck\|\Delta^{-1/2}Z^j\|^2  +\frac{Ck}{\eps}\|Z^j\|_{\mathbb{L}^3}^3 
+\eps^\gamma(\Delta^{-1/2}g, \Delta^{-1/2} Z^{j-1})\Delta_jW
+ C \eps^{2\gamma}  \vert \Delta_j W\vert^2 \, .
\end{eqnarray*}
{\bf 4.} {We sum  the last inequality from
$j=1$ up to $j={\mfrakj }$, and consider $\max_{j\leq {\mfrakj }}$.
 On noting $Z^0=0$, we obtain ${\mathbb P}$-a.s.
 \begin{equation*}
A_{\mfrakj } \leq {C}{\mathcal R}_{\mfrakj } + Ck\sum_{i=1}^{\mfrakj } A_i\qquad 
(1 \leq  {\mfrakj }\leq J)\, ,
\end{equation*}
where
\begin{equation}\label{numm1}
\begin{split}
&A_{\mfrakj }  = \frac{1}{2} \max_{1\leq j\leq {\mfrakj }}\|\Delta^{-1/2}Z^j\|^2
{+ \frac{1}{2} \sum_{i=1}^{\mfrakj }\Vert \Delta^{-1/2}[Z^j - Z^{j-1}]\Vert^2} +\eps^4 k\sum_{i=1}^{{\mfrakj }}\|\nabla Z^i\|^2\, ,
\\
&{\mathcal R}_{\mfrakj }  = \frac{k}{\eps}\sum_{j=1}^{{\mfrakj }}\|Z^j\|_{\mathbb{L}^3}^3+  \eps^\gamma\max_{1\leq j\leq {\mfrakj }}|\sum_{i=1}^{j}(\deltainv g,Z^{i-1})\Delta_i W|
      + \eps^{2\gamma}\sum_{i=1}^{{\mfrakj }}|\Delta_iW|^2\, \revised{.}
\end{split}
\end{equation}
Hence, the implicit version of the discrete Gronwall  lemma implies for \rev{sufficiently small} $k \leq k_0({\mathcal D})$  that ${\mathbb P}$-a.s.
\begin{equation}\label{esti1}
A_{\mfrakj } \leq C \mathcal{R}_{\mfrakj }  \qquad \forall \, {\mfrakj } \leq J\, ,
\end{equation}
which concludes the proof.}
\qed
\end{proof} 

\medskip

In the deterministic setting (\revised{$g\equiv0$}),
an induction argument, along with an interpolation estimate for the
$\mathbb{L}^3$-norm is used to estimate the cubic error term on the right-hand side
of (\ref{error_eqn}); cf.~\cite{CHFP}.
In the stochastic setting, this induction argument {is not applicable} any more,
{which is why we separately bound errors in (\ref{error_eqn}) on two subsets $\Omega_2$ and $\Omega \setminus \Omega_2$. In the first step, \revised{we} study accumulated errors on $\Omega_2$ {\em locally} in time, and therefore mimic a related (time-continuous)
argument in \cite{abk18}. We introduce the 
%
stopping index  $1 \leq J_\eps\leq J$}
\begin{equation*}\label{exitind}
J_\eps:=\inf \bigl\{{1\leq j\leq J}:\,\, \frac{k}{\eps} \sum_{i=1}^{j} \|Z^{i}\|_{\mathbb{L}^3}^3 > \eps^{\sigma_0} \bigr\}\, ,
\end{equation*}
where the constant $\sigma_0>0$ will be specified later. {The purpose of the stopping index is to identify those $\omega \in \Omega$ where the cubic error term is small enough.}
In the sequel, we estimate the terms on the right-hand side of (\ref{error_eqn}),
putting ${\mfrakj } = J_{\varepsilon}$. Clearly, the  
part $\frac{k}{\varepsilon} \sum_{i=1}^{J_{\varepsilon}-1} \Vert Z^i \Vert_{{\mathbb L}^3}^3$
of ${\mathcal R}_{J_{\varepsilon}}$ in (\ref{numm1}) is bounded by $\varepsilon^{\sigma_0}$;  the {remaining part will be denoted by
$\widetilde{\mathcal R}_{J_{\varepsilon}} := {\mathcal R}_{J_{\varepsilon}} -
\frac{k}{\varepsilon} \sum_{i=1}^{J_{\varepsilon}-1} \Vert Z^i \Vert_{{\mathbb L}^3}^3$,} i.e.,
%
\begin{equation*}\label{defreps}
\widetilde{\mathcal{R}}_{J_\eps} = \eps^\gamma\max_{1\leq j\leq J_\eps}\bigl |\sum_{i=1}^{j}\bigl(\deltainv g,Z^{i-1}\bigr)\Delta_iW \bigr |+\eps^{2\gamma}\sum_{j=1}^{J_\eps}|\Delta W_j|^2
+ \frac{k}{\eps}\|Z^{J_\eps}\|_{\mathbb{L}^3}^3\,. 
\end{equation*}
{For $0< \kappa_0 < \sigma_0$, we gather those $\omega \in \Omega$ in the subset 
\begin{equation*}\label{omega2def}
\Omega_2 := \big{\{}\omega\in\Omega: \, \widetilde{\mathcal{R}}_{J_\eps}(\omega)  \leq {\eps}^{\kappa_0}
 \bigr\}
\end{equation*}
where the error terms in Lemma~\ref{thma2}
which cannot be controlled by the stopping index $J_\eps$
do not exceed the larger error threshold $\varepsilon^{\kappa_0}$.
The following lemma quantifies the possible error accumulation in time on $\Omega_2$ up to \revised{the} stopping index $J_\varepsilon$ in terms of $\sigma_0, \kappa_0 >0$, and illustrates the role of $k$ in this matter; it further
provides a lower bound for the measure of $\Omega_2$ correspondingly.}
%
%


\begin{lem}\label{lem_patherr}
Assume $\gamma > \frac{3}{2}$, $0< \kappa_0 < \sigma_0$, \revised{$\Vert u^\varepsilon_0\Vert_{{\mathbb H}^3} \leq C \varepsilon^{-{\mathfrak p}_{\tt CH}}$ for $\varepsilon \in (0,\varepsilon_0)$, 
and let} $k \leq C \varepsilon^{{\mathfrak l}_{\tt CH}}$  with ${\mathfrak l}_{\tt CH} \geq 3$ from Lemma \ref{lem_xch} be sufficiently small.
%
Then, there exists $C >0$ such that
\begin{eqnarray*}
{\rm i)}&& \max_{1\leq i\leq J_\eps}\|\Delta^{-1/2}Z^i\|^2+ {\eps^4} k\sum_{i=1}^{J_\eps}\|\nabla Z^i\|^2 \leq C \varepsilon^{\kappa_0} \qquad \mbox{on } \Omega_2\, , \\
{\rm ii)}&& {\mathbb E}\Bigl[ \mathbbm{1}_{\Omega_2} \Bigl(\max_{1\leq i\leq J_\eps}\|\Delta^{-1/2}Z^i\|^2+\frac{\eps^4}{2} k\sum_{i=1}^{J_\eps}\|\nabla Z^i\|^2 \Bigr)\Bigr]\leq
C\max \bigl\{ \frac{k^2}{\varepsilon^4},\varepsilon^{\gamma+\frac{\sigma_0+1}{3}},  \varepsilon^{\sigma_0}, \varepsilon^{2\gamma}\bigr\}\, .
\end{eqnarray*}
Moreover, $\mathbb{P}[\Omega_2] \geq 1-  \frac{C}{\varepsilon^{\kappa_0}} \max \bigl\{ \frac{k^2}{\varepsilon^4}, \varepsilon^{\gamma+\frac{\sigma_0+1}{3}}, \varepsilon^{\sigma_0},\varepsilon^{2\gamma}\bigr\}$.
\end{lem}
{The proof uses the discrete BDG-inequality (Lemma~\ref{lembdg}), which is suitable
for the implicit Scheme \ref{scheme_time};
{we use the
higher-moment estimates} from Lemma \ref{lem_energy}, iii)
to bound the last term in $\widetilde{\mathcal R}_{J_{\varepsilon}}$.
%
%
%

\begin{proof}
{\bf 1.} Estimate i) follows directly from Lemma \ref{thma2}, using the definitions of $J_{\varepsilon}$ and $\Omega_2$.


{\bf 2.} Let $\Omega_2^{c} := \Omega \setminus \Omega_2$.
%
%
We use \revised{Markov's} inequality to estimate $\mathbb{P}[\Omega_{2}^c]\leq
{\frac{1}{\eps^{\kappa_0}}}\mathbb{E}[\widetilde{\mathcal{R}}_{J_\eps}]$.
We first estimate the last term in $\widetilde{\mathcal R}_{J_{\varepsilon}}$: interpolation of ${\mathbb L}^3$ between
${\mathbb L}^2$ and ${\mathbb H}^1$, then of ${\mathbb L}^2$ between ${\mathbb H}^{-1}$ and ${\mathbb H}^1$ (${\mathcal D} \subset {\mathbb R^2}$)
and the Young's inequality yield
\begin{equation}\label{l3est}
\frac{k}{\varepsilon} \Vert Z^{J_{\varepsilon}}\Vert^3_{{\mathbb L}^3} \leq 
\frac{Ck}{\varepsilon} \Vert Z^{J_{\varepsilon}}\Vert_{{\mathbb H}^{-1}} \Vert \nabla Z^{J_{\varepsilon}}\Vert^2_{{\mathbb L}^{2}} \leq \frac{1}{8} \Vert \Delta^{-1/2} Z^{J_{\varepsilon}}\Vert^2_{{\mathbb L}^{2}}
+ \frac{C k^2}{\varepsilon^2} \Vert \nabla Z^{J_{\varepsilon}}\Vert^4_{{\mathbb L}^2}\,.
\end{equation}
The leading term on the right-hand side is absorbed on the left-hand 
side of the inequality in Lemma \ref{thma2}, which is considered on the whole of $\Omega$;
the expectation of the last term (on the whole of $\Omega$) is bounded via Lemma \ref{lem_energy}, iv) by $\frac{Ck^2}{\varepsilon^4} \bigl( \vert {\mathcal E}(u_0^\varepsilon)\vert^2 +1\bigr)$.

For the first term in $\widetilde{\mathcal R}_{J_{\varepsilon}}$ we use 
the discrete BDG-inequality (Lemma \ref{lembdg})  to bound its expectation by 
$$ C \eps^\gamma\mathbb{E}\Big[\sum_{i=1}^{J_\eps+1}k \bigl(\deltainv g,Z^{i-1} \bigr)^2\Big]^{\frac{1}{2}} \, .$$
In order to benefit from the definition of $J_{\varepsilon}$ for its estimate, we split the leading summand,
\begin{eqnarray*}\nonumber
&=&
C \eps^\gamma\mathbb{E}\Big[\sum_{i=1}^{J_\eps}k \vert \bigl(\deltainv g,Z^{i-1} \bigr) \vert^2\Big]^{\frac{1}{2}} + C\sqrt{k} \varepsilon^{\gamma} {\mathbb E} \bigl[ \vert(\deltainv g,
Z^{J_{\varepsilon}})\vert^2 \bigr]^{\frac{1}{2}}
\\ \nonumber 
&\leq&  C\eps^\gamma \mathbb{E}\Bigl[k \bigl(\sum_{i=1}^{J_\eps-1}\|Z^{i}\|_{\mathbb{L}^3}^{3}\bigr)^{\frac{2}{3}} \bigl(\sum_{i\leq J}1^3 \bigr)^{\frac{1}{3}}\Bigr]^{\frac{1}{2}} 
+ \rev{C\sqrt{k} \varepsilon^{\gamma} {\mathbb E} \bigl[ \vert(\nabla \deltainv g, \nabla \deltainv Z^{J_{\varepsilon}})\vert^2 \bigr]^{\frac{1}{2}}}
\\ \nonumber
& \leq & C \varepsilon^{\gamma+\frac{\sigma_0+1}{3}} +
C \sqrt{k} \varepsilon^{\gamma} {\mathbb E}\bigl[\Vert \Delta^{-1/2} Z^{J_{\varepsilon}}\Vert^2\bigr]^{\frac{1}{2}} \\ \nonumber
&\leq& C \varepsilon^{\gamma+\frac{\sigma_0+1}{3}} + C k \varepsilon^{2\gamma} +
\frac{1}{8} {\mathbb E}\bigl[\Vert \Delta^{-1/2} Z^{J_{\varepsilon}}\Vert^2\bigr]\, .
\end{eqnarray*}
Putting things together leads to ${\mathbb E}[\frac{1}{2}A_{J_{\varepsilon}}] \leq 
C \bigl(\varepsilon^{\sigma_0} + \varepsilon^{\gamma+\frac{\sigma_0+1}{3}} +  \varepsilon^{2\gamma} + \frac{k^2}{\varepsilon^4}\bigr)$.
Revisiting 
%
%
(\ref{l3est}) again then yields from Lemma \ref{thma2}
\begin{equation}\label{est_rtilde}
{\mathbb E}[\widetilde{\mathcal R}_{J_{\varepsilon}}] \leq   C\max \bigl\{ \frac{k^2}{\varepsilon^4}, \varepsilon^{\gamma+\frac{\sigma_0+1}{3}}, \varepsilon^{\sigma_0},
\varepsilon^{2\gamma}\bigr\}\, .
\end{equation}
{\bf 3.} Consider the inequality in Lemma \ref{thma2} on $\Omega_2$. 
\rev{The estimate ii) then follows after taking expectation, using (\ref{est_rtilde}) and recalling the definition of $J_\eps$.}

\qed
\end{proof}}

{The previous lemma establishes local error bounds for iterates of Scheme \ref{scheme_time} -- by using the 
stopping index $J_\varepsilon$, and the subset $\Omega_2 \subset \Omega$; the following lemma identifies values $(\gamma, \sigma_0, \kappa_0)$ such that
Lemma \ref{lem_patherr} remains valid globally in time on $\Omega_2$.}


{
\begin{lem}\label{lem_jeps}
 Let the assumptions in Lemma~\ref{lem_patherr} be valid. Assume
$$
\sigma_0 > 10\,, \qquad \rev{\sigma_0 >} \kappa_0 > \frac{2}{3}(\sigma_0 + 5)\, .
$$
There exists $\eps_0\equiv \eps_0(\sigma_0, \kappa_0)$, such that for every
$\eps \in (0,\eps_0)$
$$
J_\eps(\omega)=J \qquad \forall\, \omega \in \Omega_2\, .
$$
Moreover, $\lim_{\eps \downarrow 0}{\mathbb P}[\Omega_2] =1$ if
$$\gamma > \max\{{\frac{19}{3}}, \frac{\kappa_0}{2}\}\,, \qquad
k^2 \leq  C\varepsilon^{4+\kappa_0+\beta}\, ,
$$
where $\beta > 0$ may be arbitrarily small.
\end{lem}
\revised{Compared} to assumption {\bf (A)}, the less restrictive lower bound for $\gamma$ is due to the use of the 
discrete spectral estimate (see Lemma \ref{lem_xch}, v)),
which introduces a factor $\varepsilon^{-4}$ that 
is absorbed into $\eps^{\frac{3}{2}\kappa_0}$ in the proof below.
Consequently we {only need to} require $\gamma \geq \frac{19}{3}$
in order to ensure positive probability of $\Omega_2$.}

\begin{proof}
%
{\bf 1.} Assume that $J_\eps < J$ on $\Omega_2$; we want to verify that
\begin{equation*}\label{assum-1}
\frac{k}{\eps}\sum_{i=1}^{J_\eps}\|Z^{i}\|_{\mathbb{L}^3}^3 \leq \eps^{\sigma_0} \qquad \mbox{on}\,\, \Omega_2\, .
\end{equation*}
Use (\ref{l3est}), and the estimate Lemma \ref{lem_patherr}~i) to conclude
\begin{equation*}\label{dim27}
\frac{k}{\eps}\sum_{i=1}^{J_\eps}\|Z^{i}\|_{\mathbb{L}^3}^3
\leq \frac{C}{\eps} \max_{1\leq i\leq J_\eps}\|\Delta^{-1/2} Z^i\|_{\mathbb{L}^{2}}\Big(\sum_{i=1}^{J_\eps}k\|\nabla Z^i\|^2\Big)
\leq C \varepsilon^{-1 + \frac{\kappa_0}{2} +(\kappa_0-4)}\quad \rev{\mathrm{on}\,\,\Omega_2} \, .
\end{equation*}
The right-hand side above is below $\varepsilon^{\sigma_0}$ for 
$\frac{3\kappa_0}{2} > \sigma_0 + 5$ \rev{and $\eps< \eps_0$ with sufficiently small $\eps_0\equiv \eps_0(\sigma_0, \kappa_0)$}. The additional condition $\kappa_0 < \sigma_0$ \rev{(which will be required in step {\bf 2.} below)} imposes that $\sigma_0 > 10$.\\
%
%
{\bf 2.} \rev{Recall that the last part in Lemma \ref{lem_patherr} yields 
$\mathbb{P}[\Omega_2] \geq 1-  C\varepsilon^{-\kappa_0} \max \bigl\{ \frac{k^2}{\varepsilon^4}, \varepsilon^{\gamma+\frac{\sigma_0+1}{3}}, \varepsilon^{\sigma_0},\varepsilon^{2\gamma}\bigr\}$.}
\rev{Hence, to ensure  $\mathbb{P}[\Omega_2] > 0$
requires $\gamma + \frac{\sigma_0+1}{3} -\kappa_0 >0$, $\sigma_0>\kappa_0$, $\gamma > \frac{\kappa_0}{2}$ and $k^2 \leq C \eps^{4+\kappa_0+\beta}$, $\beta>0$.}
In addition, by step {\bf 1.},  $\kappa_0 > \frac{2}{3}(\sigma_0 + {5})$, \rev{$\sigma_0>10$, which along with  $\gamma + \frac{\sigma_0+1}{3} -\kappa_0 >0$, $\sigma_0>\kappa_0$} 
implies $\gamma \geq {\frac{19}{3}}$. 
\qed
\end{proof}}

Next, we bound $\max_{1\leq i\leq J}\|\Delta^{-1/2}Z^i\|^2+\frac{\eps^4}{2} k\sum_{i=1}^{J}\|\nabla Z^i\|^2$ on the whole sample set. We collect the requirements on the analytical and numerical parameters:
\begin{itemize}
\item[{\bf (B)}] {Let  $u^\varepsilon_0 \in {\mathbb H}^3$, \revdim{$\mathcal{E}(u_0^\eps)<C$}.} Assume that $(\sigma_0, \kappa_0, \gamma)$ satisfy 
$$\sigma_0 >10\,, \qquad { \sigma_0} > \kappa_0 > \frac{2}{3}(\sigma_0 + 5)\,, \qquad \gamma \geq \max\{ {\frac{19}{3}}, \frac{\kappa_0}{2}\}\, .$$
For sufficiently small $\varepsilon_0 \equiv (\sigma_0,\kappa_0) >0$ and ${\mathfrak l}_{\tt CH} \geq 3$ from Lemma \ref{lem_xch}, 
and arbitrary $0 < \beta < \frac{1}{2}$, the time-step satisfies
$$
k \leq C \min \bigl\{\varepsilon^{{\mathfrak l}_{\tt CH}},
\varepsilon^{2+\frac{\kappa_0}{2}+\beta}\bigr\} \qquad \forall\, \varepsilon \in (0,\varepsilon_0) \,.$$
\end{itemize}
We note that, except for the higher regularity of the initial condition, the assumption {\bf (B)} is less restrictive than the assumption {\bf (A)}
from Section~\ref{sec_cont_ch}.
\revdim{Furthermore, the condition $\mathcal{E}(u_0^\eps) < C$ can be weakened to $\mathcal{E}(u_0^\eps) < C\eps^{-\alpha}$, $\alpha>0$, cf. \cite[Assumption (GA$_2$)]{fp_ifb05}.}

{\begin{lem}\label{cor_err_xxa}
Suppose {\bf (B)}.
Then there exists $C>0$ such that
$$
\mathbb{E}\big[\max_{1\leq j\leq J}\|Z^j\|^2_{\mathbb{H}^{-1}}  + {\eps^4} k\sum_{i=1}^{J}\|\nabla Z^i\|^2
\bigr] \leq \Bigl(\frac{C}{\varepsilon^{\kappa_0}} \max \bigl\{ \frac{k^2}{\varepsilon^4}, \varepsilon^{\gamma+\frac{\sigma_0+1}{3}}, \varepsilon^{\sigma_0},\varepsilon^{2\gamma}\bigr\}\Bigr)^{\frac{1}{2}}\,.$$
\end{lem}
}
{\begin{proof}
Recall the notation from (\ref{numm1}), and 
\rev{split $\mathbb{E}[A_{J}] = \mathbb{E}[\mathbbm{1}_{\Omega_2} A_J] + \mathbb{E}[\mathbbm{1}_{\Omega^c_2} A_J]$}. 
\rev{Due to assumption {\bf (B)} it follows directly from Lemma~\ref{lem_patherr},~ii) and Lemma~\ref{lem_jeps} that
\begin{equation}\label{erra1}
{\mathbb E}[\mathbbm{1}_{\Omega_2} A_J] \leq C\max \bigl\{ \frac{k^2}{\varepsilon^4},\varepsilon^{\gamma+\frac{\sigma_0+1}{3}},  \varepsilon^{\sigma_0}, \varepsilon^{2\gamma}\bigr\}\,.
\end{equation}

In order to bound ${\mathbb E}[\mathbbm{1}_{\Omega^c_2} A_J]$,
we use the embedding $\mathbb{L}^4\subset \mathbb{H}^{-1}$ which along with the higher-moment estimate from Lemma~\ref{lem_energy}~iv)  implies that
$$
\mathbb{E}\big[A_J^2\big] \leq C \mathbb{E}\big[|\mathcal{E}(X^\revdim{J})|^2\big] \leq C(|\mathcal{E}(u_\eps^0)|^2+1)\,.
$$
Next, we note that by Lemma~\ref{lem_patherr} it follows that
$$
{\mathbb P}[\Omega^c_2]\leq 1 - {\mathbb P}[\Omega_2]\leq \frac{C}{\varepsilon^{\kappa_0}} \max \bigl\{ \frac{k^2}{\varepsilon^4}, \varepsilon^{\gamma+\frac{\sigma_0+1}{3}}, \varepsilon^{\sigma_0},\varepsilon^{2\gamma}\bigr\}\,.
$$
Hence, using the Cauchy-Schwarz inequality we get
\begin{equation}\label{erra2}
{\mathbb E}[\mathbbm{1}_{\Omega^c_2} A_J] \leq 
\bigl({\mathbb P}[\Omega_2^c]\bigr)^{1/2}
\bigl({\mathbb E}[A_J^2]\bigr)^{1/2} \leq 
\Bigl(\frac{C}{\varepsilon^{\kappa_0}} \max \bigl\{ \frac{k^2}{\varepsilon^4}, \varepsilon^{\gamma+\frac{\sigma_0+1}{3}}, \varepsilon^{\sigma_0},\varepsilon^{2\gamma}\bigr\}\Bigr)^{\frac{1}{2}} \bigl( {\mathcal E}(u^\varepsilon_0)+1\bigr)\,.
\end{equation}
After inspecting (\ref{erra1}), (\ref{erra2}) we note that the statement follows 
by assumption {\bf (B)}, since the latter contribution dominates the error.
}
\qed
\end{proof}}}

The dominating error contribution in Lemma~\ref{cor_err_xxa} comes from the term ${\mathbb E}[\mathbbm{1}_{\Omega^c_2} A_J]$.
{
This is in contrast to Section~\ref{sec_cont_ch} where the error contribution from the set $\Omega_1^c$ can be made arbitrarily small, 
due to the additional parameter $\ellch>0$ in Lemma~\ref{lem_rest} which can be chosen arbitrarily large independently of the other parameters.
}

We are now ready to prove the first main result of this paper.
\begin{thm}\label{thmfin}
Let  $u_0^\eps\in \mathbb{H}^3$, \revised{let} $u$ be the strong solution of \eqref{p1}, 
and \revised{let $\left\{ X^j,\ j=1,\dots, J\right\}$ solve} Scheme \ref{scheme_time}. Suppose {\bf (A)}. 
Then there exists a constant $C>0$ such that for all $0 { < } \beta < \frac{1}{2}$
\begin{eqnarray*}\nonumber
&&\mathbb{E}\big[\max_{1\leq j\leq J}\|u(t_j)-X^j\|_{\mathbb{H}^{-1}}^2 \bigr]
\\
&&\qquad 
\leq C \max\Bigl\{{\eps^{\frac{2}{3}\sigma_0}}, \Bigl(\varepsilon^{-\kappa_0} \max \bigl\{ \frac{k^2}{\varepsilon^4}, \varepsilon^{\gamma+\frac{\sigma_0+1}{3}}, \varepsilon^{\sigma_0},\varepsilon^{2\gamma}\bigr\}\Bigr)^{\frac{1}{2}},
\frac{k^{2-\beta}}{\varepsilon^{{\mathfrak m}_{\tt CH}}}\Bigr\}\, .
\end{eqnarray*}
\end{thm}
Due to condition {\bf (A)}$_2$ it holds that $\sigma_0-\kappa_0 < \frac{1}{3}\sigma_0$. Consequently
the contribution $\eps^{\frac{2}{3}\sigma_0}$ in the error estimate
is dominated by $\eps^{\frac{\sigma_0-\kappa_0}{2}}$; it is only stated explicitly
to highlight the error contribution from the difference $u-u_{\tt CH}$ from Section~\ref{sec_cont_ch}.
\\
\begin{proof}
We estimate the error via splitting \revised{it} into three contributions,
\begin{equation*}
\max_{1\leq j\leq J}\|u(t_j)-u_{\tt CH}(t_j)\|_{\mathbb{H}^{-1}}^2 +
\max_{1\leq j\leq J}\|u_{\tt CH}(t_j) - X^j_{\tt CH}\|_{\mathbb{H}^{-1}}^2 +
\max_{1\leq j\leq J}\|X^j_{\tt CH} - X^j\|_{\mathbb{H}^{-1}}^2 =: I + II + III\, .
\end{equation*}
Lemma \ref{RL2est} bounds ${\mathbb E}[I]$, Lemma \ref{lem_xch}, iv) yields ${\mathbb E}[II] \leq \frac{k^{2-\beta}}{\varepsilon^{{\mathfrak m}_{\tt CH}}}$, and
${\mathbb E}[III]$ is bounded in Lemma \ref{cor_err_xxa}.
\qed
\end{proof}

{
\begin{rem}\label{rem_est_discuss}
An alternative approach to Theorem~\ref{thmfin}
would be to follow the arguments in \cite{pm18} for a related problem, 
which exploit a weak monotonicity property of the drift operator in (\ref{p1}), and stability of the discretization to obtain a strong error estimate for Scheme \ref{scheme_time} of the form
\begin{equation}\label{wemayf1}
\mathbb{E}\Big[\max_{1\leq j\leq J}\|u(t_j)-X^j\|_{\mathbb{H}^{-1}}^2\Big] \leq C_{\beta} \exp\bigl(\frac{T}{\eps} \bigr)k^{1-\beta} \qquad (\beta > 0)\, .
\end{equation}
While the error tends to zero for $k \downarrow 0$ in (\ref{wemayf1}), this estimate 
is only of limited practical relevancy in the asymptotic regime where $\varepsilon$ is small, since only prohibitively small step sizes
 $k \ll \exp(-\frac{1}{\varepsilon})$
are required in (\ref{wemayf1}) to guarantee small approximation errors for
iterates from Scheme \ref{scheme_time}. Moreover, the error analysis that leads to (\ref{wemayf1})
does not provide any insight on how to numerically resolve diffuse interfaces via
proper balancing of discretization parameter $k$ and interface width $\varepsilon$ --- which is relevant in the asymptotic regime where $\varepsilon \ll 1$.
%

\end{rem}
}

\section{Space-time Discretization of (\ref{p1})}\label{sec_space_time}
{
We generalize the convergence results in Section \ref{asec_time} 
 for Scheme \ref{scheme_time} to its space-time discretization. For this purpose, we introduce some further notations: let ${\mathcal T} _h$ be a quasi-uniform triangulation of ${\mathcal D}$, and ${\mathbb V}_h \subset {\mathbb H}^1$ be the finite element space of piecewise affine, globally continuous functions,
$${\mathbb V}_h := \bigl\{ v_h \in C(\overline{D});\, v_h \bigl\vert_K \in P_1(K) \quad \forall\, K \in {\mathcal T}_h\bigr\}\, ,$$
and $\mathring{{\mathbb V}}_h := \bigl\{ v_h \in {\mathbb V}_h:\, (v_h,1) = 0\bigr\}$. We recall the ${\mathbb L}^2$-projection $P_{{\mathbb L}^2}: {\mathbb L}^2 \rightarrow {\mathbb V}_h$, via
$$\bigl( P_{{\mathbb L}^2} v -v, \eta_h\bigr) = 0 \qquad \forall\, \eta_h \in {\mathbb V}_h\, , $$
and the Riesz projection $P_{{\mathbb H}^1}: {\mathbb H}^1 \cap {\mathbb L}^2_0 \rightarrow \mathring{\mathbb V}_h$, via
$$\bigl( \nabla [P_{{\mathbb H}^1} v -v], \nabla \eta_h\bigr) = 0 \qquad \forall\, \eta_h \in {\mathbb V}_h\, .$$
In what follows, we allow meshes ${\mathcal T}_h$ for which $P_{{\mathbb L}^2}$ is ${\mathbb H}^1$-stable; see \cite{C1}. Also, we define the inverse discrete  Laplacian $\deltainvh: {\mathbb L}^2_0 \rightarrow \mathring{\mathbb V}_h$  via
$$\bigl(\nabla \deltainvh v, \nabla \eta_h\bigr) = (v,\eta_h) \qquad \forall\, \eta_h \in {\mathbb V}_h\, . $$
We are ready to present the space discretization of Scheme \ref{scheme_time}.
\begin{scheme}\label{scheme_space_time} For every $1 \leq j \leq J$, find 
{a $[{\mathbb V}_h]^2$-valued r.v.~$(X_h^j, w_h^j)$}
 such that ${\mathbb P}$-a.s.
\begin{equation*}
\begin{split}
&(X_h^j-X_h^{j-1},\varphi_h)+k(\nabla w_h^{j},\nabla
\varphi_h)=\varepsilon^{\gamma}\bigl(g,\varphi_h
\bigr)\Delta_j W
\;\;\;\; \, \, \quad \, \forall\, \varphi_h \in {\mathbb V}_h\,,\\
&\varepsilon(\nabla X_h^j,\nabla \psi_h)+\frac{1}{\varepsilon}
\bigl(f(X_h^j),\psi_h \bigr)=(w^j_h,\psi_h) \qquad \qquad \quad
\quad \ \, \forall\, \psi_h \in
{\mathbb V}_h\, ,\\
&X^0_h = {P_{{\mathbb L}^2}u_0^\eps} \in \revised{\mathring{\mathbb V}_h}\, .
\end{split}
\end{equation*}
\end{scheme}
For all $1 \leq j \leq J$, the  solution $\{(X_h^j, w_h^j)\}_{1 \leq j \leq J}$ satisfies $(X_h^j,1) = 0$ ${\mathbb P}$-a.s.\\

{\bf Claim 1.} $\{(X_h^j, w_h^j)\}_{1 \leq j \leq J}$ inherits all stability bounds in Lemma \ref{lem_energy}.

\begin{proof} {i')} In order to verify the corresponding version of i) for $\{{\mathcal E}(X^j_h)\}_{1 \leq j\leq J}$, we may choose
$\varphi_h = w_h^j(\omega)$ and $\psi_h = [X^j_h - X^{j-1}_h](\omega)$ 
in Scheme \ref{scheme_space_time}, as in
part {\bf i)} of the proof of
Lemma \ref{lem_energy}. We then obtain a corresponding version of (\ref{test1}),
and (\ref{fest0}). 

The next argument in the proof of Lemma~\ref{lem_energy} that leads to
(\ref{est_l2inc}) may again be reproduced for Scheme \ref{scheme_space_time}  by choosing
$\varphi_h = \deltainvh [X^j_h - X^{j-1}_h](\omega)$, and using the definition of $\deltainvh$, as well as $X^j_h, P_{{\mathbb L}^2} g \in {\mathbb L}^2_0$ ${\mathbb P}$-a.s., such that
$$\Vert \nabla \deltainvh [X^j_h - X^{j-1}_h]\Vert^2 \leq
\Bigl( k \Vert \nabla w^j_h\Vert + \varepsilon^\gamma \Vert \nabla \deltainvh  P_{{\mathbb L}^2} g\Vert \vert \Delta_j W\vert\Bigr)\Vert \nabla \deltainvh [X^j_h - X^{j-1}_h]\Vert\, ,$$
since $\Vert \nabla \deltainvh P_{{\mathbb L}^2} g\Vert \leq \Vert g \Vert \leq C$.

To obtain the first identity in (\ref{lc3}) for Scheme \ref{scheme_space_time}, we
use
$\varepsilon^\gamma (g, w^j_h)\Delta_jW = \varepsilon^\gamma \bigl(P_{{\mathbb L}^2}g, w^j_h \bigr)\Delta_jW$, such that the second equation in Scheme \ref{scheme_space_time} with
$\psi_h =P_{{\mathbb L}^2}g$ may be applied; as a consequence, $g$
has to be replaced by $P_{{\mathbb L}^2}g$ in the rest of 
equality (\ref{lc3}). This modification leads to the term
$\Vert \nabla P_{{\mathbb L}^2}g\Vert$ in (\ref{est_a1}), which is again bounded
by $\Vert \nabla g\Vert$; the bound $\Vert P_{{\mathbb L}^2}g\Vert_{{\mathbb L}^{\infty}} \leq C$, which is
required to bound the term $A_{3,1}$ from (\ref{est_a3}), follows by an approximation result; cf.~\cite[Chapter 7]{brennerscott}. 
The above steps then yield the estimate (\ref{est_step2}) for $\{ (X^j_h, w^j_h)\}_{1 \leq j \leq J}$. \\

{ii'), iii'), iv')} We can follow the argumentation in the proof of Lemma \ref{lem_energy} without change.
\end{proof}

\medskip

{\bf Claim 2.} Lemma \ref{thma2} holds for $\{(X_h^j, w_h^j)\}_{1 \leq j \leq J}$, i.e.:
$Z^j_h := X^j_h - X^j_{{\tt CH};h}$ satisfies 
$\mathbb{P}$-a.s. 
\begin{equation*}\label{error_eqn_h}
\begin{split}
&\max_{1\leq j\leq {\mfrakj }}\|\nabla \deltainvh Z_h^j\|^2+c{\eps k}\sum_{i=1}^{{\mfrakj }}\|\nabla Z_h^i\|^2\\
 &\quad \leq 
\frac{Ck}{\eps}\sum_{i=1}^{{\mfrakj }}\|Z_h^i\|_{\mathbb{L}^3}^3+
C\eps^\gamma\max_{1\leq j\leq {\mfrakj }}|\sum_{i=1}^{j}(\deltainvh 
P_{{\mathbb L}^2} g,Z_h^{i-1})\Delta_j W|
+C\eps^{2\gamma} \sum_{i=1}^{{\mfrakj }}|\Delta_i W|^2 \,,
\end{split}
\end{equation*}
{for all ${\mfrakj } \leq J$},
provided that additionally
\begin{equation}\label{assu-01}
k\leq C\min\{\eps^{{\mathfrak p}_{\tt CH}}, h^{\widetilde{\mathfrak{q}}_{\tt CH}} \}\,, \qquad 
{h \leq C \min\{ 1, \rev{k^{2\beta}}\} \varepsilon^{\widetilde{\mathfrak  p}_{\tt CH}}}
\end{equation}
for any $\beta > 0$, and ${\mathfrak p}_{\tt CH}, \widetilde{\mathfrak{q}}_{\tt CH}, \widetilde{\mathfrak  p}_{\tt CH} >0$.
\rev{The exponents ${\mathfrak p}_{\tt CH}, \widetilde{\mathfrak{q}}_{\tt CH}, \widetilde{\mathfrak  p}_{\tt CH} >0$
are chosen in order to satisfy the assumptions of \cite[Corollary 2]{CHFP} and \cite[Theorem 3.2]{fp_ifb05}.}
\rev{In particular (\ref{assu-01}) 
is required to obtain the fully discrete counterpart of Lemma~\ref{lem_xch},~iii)-iv).}

{Requirement (\ref{assu-01})$_2$ comes from \cite[Corollary 2, assumption 4)]{CHFP}; see also \cite[Theorem 3.1, assumption 3)]{fp_ifb05} accordingly.  Since $\beta >0$ may be chosen arbitrarily small, it does not severely restrict admissible $h>0$.}

\begin{proof}
Again, we here denote by $\{(X_{{\tt CH};h}^j, w_{{\tt CH};h}^j)\}_{1 \leq j \leq J} \subset [{\mathbb V}_h]^2$ the solution of Scheme \ref{scheme_space_time} for $g \equiv 0$, whose stability and convergence properties are studied in 
\cite{CHFP,fp_ifb05}. {Under the assumption (\ref{assu-01})},  \cite[Theorem~3.2, (iii)]{fp_ifb05} provides the bound
\begin{equation*}
\max_{0 \leq j \leq J} \Vert X^j_{{\tt CH};h}\Vert_{{\mathbb L}^{\infty}} \leq C\, .
\end{equation*}
We use this bound to adapt estimate (\ref{didi1}) to the present setting and get
\begin{equation*}
\begin{split}
\bigl(f(X_h^j)-f(X_{{\tt CH};h}^{j}),Z^j_h\bigr) 
&\geq  \bigl(f'(X_{{\tt CH};h}^{j})Z^j_h,Z^j_h \bigr) \rev{- C\|Z^j_h\|_{\mathbb{L}^3}^3}\\
&\geq [1-{\eps^3}] \bigl(f'(X_{{\tt CH};h}^{j})Z^j_h,Z^j_h\bigr)-C\|Z^j_h\|_{\mathbb{L}^3}^3 + {\eps^3} \bigl(f'(X_{{\tt CH};h}^{j})Z^j_h,Z^j_h\bigr) \, .
\end{split}
\end{equation*}
Step {\bf 2.} of the proof of Lemma \ref{thma2} involves the discrete spectral estimate
(see Lemma \ref{lem_xch}, iv)) for $\{X^j_{\tt CH}\}_{j}$ to handle the leading term on the right-hand side of (\ref{didi1})
 -- which we do not have for $\{X^j_{{\tt CH};h}\}_j$ in the present setting. Therefore, we perturb the leading term on the right-hand side of the last inequality, and use
the {$\mathbb{L}^\infty$-bounds for $X^j_{{\tt CH}}$, $X^j_{{\tt CH};h}$},
as well as
the mean-value theorem to conclude 
\begin{equation*}
\begin{split}
\bigl(f'(X_{{\tt CH};h}^{j})Z^j_h,Z^j_h \bigr) &= \bigl(f'(X_{{\tt CH}}^{j})Z^j_h,Z^j_h \bigr) + \Bigl( \bigl[f'(X_{{\tt CH};h}^{j}) - f'(X_{{\tt CH}}^{j})\bigr]Z^j_h,Z^j_h \bigr) \Bigr) \\
&\geq \bigl(f'(X_{{\tt CH}}^{j})Z^j_h,Z^j_h \bigr) - C \Vert Z^j_h\Vert^3_{{\mathbb L}^3}\, .
\end{split}
\end{equation*}
The remaining steps in the proof of Lemma \ref{thma2} now follow with only minor adjustments. 
\end{proof} 

\medskip

{\bf Claim 3.} {Additionally assume (\ref{assu-01}). Then Lemma \ref{lem_patherr} holds for $\{ Z^j_h\}_j$}, i.e.,
\begin{eqnarray*}
{\rm i)}&& \max_{1\leq i\leq J_\eps}\|\nabla \deltainvh  Z_h^i\|^2+ {\eps^4} k\sum_{i=1}^{J_\eps}\|\nabla Z_h^i\|^2 \leq C \varepsilon^{\kappa_0} \qquad \mbox{on } \Omega_{2;h}\, , \\
{\rm ii)}&& {\mathbb E}\Bigl[ \mathbbm{1}_{\Omega_{2;h}} \Bigl(\max_{1\leq i\leq J_\eps}\|\nabla \deltainvh Z_h^i\|^2+\frac{\eps^4}{2} k\sum_{i=1}^{J_{\eps,h}}\|\nabla Z_h^i\|^2 \Bigr)\Bigr]\leq
C\max \bigl\{ \frac{k^2}{\varepsilon^4},\varepsilon^{\gamma+\frac{\sigma_0+1}{3}},  \varepsilon^{\sigma_0}, \varepsilon^{2\gamma}\bigr\}\, .
\end{eqnarray*}
Moreover, $\mathbb{P}[\Omega_{2;h}] \geq 1-  \frac{C}{\varepsilon^{\kappa_0}} \max \bigl\{ \frac{k^2}{\varepsilon^4}, \varepsilon^{\gamma+\frac{\sigma_0+1}{3}}, \varepsilon^{\sigma_0},\varepsilon^{2\gamma}\bigr\}$,
where ${\Omega}_{2;h} := \bigl\{\omega \in \Omega;\ \widetilde{\mathcal R}_{J_{\varepsilon,h};h}(\omega) \leq \eps^{\kappa_0}\bigr\}$, 
for $J_{\eps,h}  :=\inf\bigl\{1 \leq j \leq J:\, \frac{k}{\eps} \sum_{i=1}^{j} \|Z_h^{i}\|_{\mathbb{L}^3}^3 > \eps^{\sigma_0} \bigr\}$,  and
\begin{equation*}
\widetilde{\mathcal{R}}_{J_{\eps,h};h}:=\eps^\gamma\max_{1\leq j\leq J_{\eps,h}}\bigl |\sum_{i=1}^{j}\bigl(\deltainvh  P_{{\mathbb L}^2} g,Z_h^{i-1}\bigr)\Delta_iW \bigr |+\eps^{2\gamma}\sum_{j=1}^{J_{\eps,h}}|\Delta W_j|^2
+ \frac{k}{\eps}\|Z_h^{J_{\eps,h}}\|_{\mathbb{L}^3}^3\, . 
\end{equation*}
\begin{proof}
The proof for Lemma \ref{lem_patherr} directly transfers to the present setting. 
\end{proof}

\medskip

{\bf Claim 4.} Lemma \ref{lem_jeps} remains valid for $\{Z^j_h\}_h$ accordingly, {provided that $h \leq C \varepsilon^{\widetilde{\mathfrak  p}_{\tt CH}}$  and
$k \leq C h^{\widetilde{\mathfrak q}_{\tt CH}}$}, i.e.:
$J_{\eps,h} = J$ for all $\omega \in \Omega_{2;h}$. \\

\begin{proof}
We only need to adapt the interpolation argument for ${\mathbb L}^3$ to the present setting, starting with the estimate $\Vert Z^i_h\Vert_{{\mathbb L}^3}^3 \leq C \Vert Z^i_h\Vert_{{\mathbb H}^{-1}} \Vert \nabla Z_h^i\Vert^2$. By the definition of the ${\mathbb H}^{-1}$-norm, the definition and ${\mathbb H}^1$-stability of the ${{\mathbb L}^2}$-projection, and again the fact that $(Z^i_h,1) = 0$, we deduce
\begin{equation*}
\begin{split}
\Vert Z^i_h\Vert_{{\mathbb H}^{-1}} &= \sup_{\psi \in {\mathbb H}^1}
\frac{(Z^i_h, P_{{\mathbb L}^2} \psi)}{\Vert \psi\Vert_{{\mathbb H}^1}}
\leq C \sup_{\psi \in {\mathbb H}^1}
\frac{(Z^i_h, P_{{\mathbb L}^2} \psi)}{\Vert \nabla P_{{\mathbb L}^2} \psi\Vert} 
= C \sup_{\psi \in {\mathbb H}^1}
\frac{(\nabla (\deltainvh Z^i_h), \nabla P_{{\mathbb L}^2} \psi)}{\Vert \nabla P_{{\mathbb L}^2} \psi\Vert} 
\\
&\leq C \Vert \nabla \deltainvh Z^i_h\Vert\, .
\end{split}
\end{equation*}
\end{proof}

\medskip

\rev{
Next, we formulate a counterpart of Lemma~\ref{cor_err_xxa} for the fully discrete numerical solution;
as a consequence of the Claims 1 to 4 above the corollary can be proven analogically to Lemma~\ref{cor_err_xxa}
with the assumption {\bf (B)} complemented by the additional restriction on the discretization parameters (\ref{assu-01}).
\begin{cor}\label{cor_err_xxah}
Suppose {\bf (B)} and (\ref{assu-01}).
Then there exists $C>0$ such that
$$
\mathbb{E}\big[\max_{1\leq j\leq J}\|Z^j_h\|^2_{\mathbb{H}^{-1}}  + {\eps^4} k\sum_{i=1}^{J}\|\nabla Z^i_h\|^2
\bigr] \leq \Bigl(\frac{C}{\varepsilon^{\kappa_0}} \max \bigl\{ \frac{k^2}{\varepsilon^4}, \varepsilon^{\gamma+\frac{\sigma_0+1}{3}}, \varepsilon^{\sigma_0},\varepsilon^{2\gamma}\bigr\}\Bigr)^{\frac{1}{2}}\,.$$
\end{cor}
}

We are now ready to extend Theorem \ref{thmfin} to Scheme \ref{scheme_space_time}.
\begin{thm}\label{thmfin_h}
Let $u$ be the strong solution of \eqref{p1}, and $\left\{ X^j_h;\, 1 \leq j \leq J\right\}$ the solution of Scheme \ref{scheme_space_time}.
{Assume \rev{{\bf (B)}}  and (\ref{assu-01})}.
Then there exists $C>0$ such that
\begin{equation}\nonumber
\begin{split}
& 
\mathbb{E}\Big[\max_{1\leq j\leq J}\|u(t_j)-X^j\|_{\mathbb{H}^{-1}}^2 
 \Big]
\\
 & \qquad 
\leq 
C \max\Bigl\{\Bigl(\varepsilon^{-\kappa_0} \max \bigl\{ \frac{k^2}{\varepsilon^4}, \varepsilon^{\gamma+\frac{\sigma_0+1}{3}}, \varepsilon^{\sigma_0},\varepsilon^{2\gamma}\bigr\}\Bigr)^{\frac{1}{2}},
\frac{k^{2-\beta}}{ \varepsilon^{{\mathfrak m}_{\tt CH}}} + \frac{h^4\rev{(1+k^{-\beta})}}{ \varepsilon^{\widetilde{\mathfrak m}_{\tt CH}}}
\Bigr\}\, ,
\end{split}
\end{equation}
where ${\mathfrak m}_{\tt CH}, \widetilde{\mathfrak m}_{\tt CH} > 0$. 
\end{thm}
\rev{We note that the exponents ${\mathfrak m}_{\tt CH}, \widetilde{\mathfrak m}_{\tt CH} > 0$ in the above estimate 
can be determined on closer inspection of \cite[Corollary 2]{CHFP} on assuming (\ref{assu-01}).
Furthermore, assumption (\ref{assu-01}), which is a simplified reformulation of assumption 4) in \cite[Corollary 2]{CHFP},  guarantees that
$\lim_{\eps \downarrow 0}\left(\frac{k^{2-\beta}}{ \varepsilon^{{\mathfrak m}_{\tt CH}}} + \frac{{h^4 \rev{(1+k^{-\beta})}}}{ \varepsilon^{\widetilde{\mathfrak m}_{\tt CH}}}\right) = 0$.}

\begin{proof}
We split the error into three contributions,
\begin{eqnarray*}\nonumber
\mathbb{E}\big[\max_{1\leq j\leq J}\|u(t_j)-X^j_h\|_{\mathbb{H}^{-1}}^2\big]&\leq&
3\mathbb{E}\big[\max_{1\leq j\leq J}\|u(t_j)-u_{\tt CH}(t_j)\|_{\mathbb{H}^{-1}}^2\big]
\\ 
&& +3\max_{1\leq j\leq J}\|u_{\tt CH}(t_j)-X^j_{{\tt CH};h}\|_{\mathbb{H}^{-1}}^2
  +3\mathbb{E}\big[\max_{1\leq j\leq J}\|X^j_h-X^j_{{\tt CH};h}\|_{\mathbb{H}^{-1}}^2\big].
\end{eqnarray*}
%
The first term is bounded by $C\eps^{\frac{2}{3}\sigma_0}$ as in Theorem~\ref{thmfin}.
The second term is bounded by $C \bigl(\frac{k^{2-\beta}}{ \varepsilon^{{\mathfrak m}_{\tt CH}}} + \frac{{h^4}\rev{(1+k^{-\beta})}}{ \varepsilon^{\widetilde{\mathfrak m}_{\tt CH}}}\bigr)$ thanks to
\cite[Corollary 2]{CHFP} \rev{(stated here in a simplified form)}, {provided assumption (\ref{assu-01}) holds}. 
The last term is bounded by
$\Bigl(\frac{C}{\varepsilon^{\kappa_0}} \max \bigl\{ \frac{k^2}{\varepsilon^4}, \varepsilon^{\gamma+\frac{\sigma_0+1}{3}}, \varepsilon^{\sigma_0},\varepsilon^{2\gamma}\bigr\}\Bigr)^{\frac{1}{2}}$, 
\rev{by Corollary~\ref{cor_err_xxah}}.
\qed
\end{proof}
}}

\section{Sharp-interface limit}\label{sec_sharp}
 In this section, we show the convergence of iterates $\{X^j\}_{j=1}^J$ of Scheme~\ref{scheme_time} to the solution of a sharp interface problem.
{Recall that in the absence of noise, the sharp} interface limit of
\eqref{p1} is given by the following deterministic Hele-Shaw/Mullins-Sekerka problem:
Find $v_{\tt MS} : [0,T] \times \mathcal{D}  \to \mathbb R$ 
and the interface $\big\{\Gamma^{{\tt MS}}_t;\, 0 \leq t \leq T\big\}$ such that
for all $t\in (0,T]$ the following conditions hold:
\begin{subequations}\label{eq:MS}
\begin{alignat}{2}
- \Delta v_{\tt MS} & = 0 && \qquad \mbox{in }\ 
\mathcal D \setminus \Gamma^{\tt MS}_t\,,\label{eq:MSa} \\
\left[\partial_\hsnormal v_{\tt MS}\right]_{\Gamma^{\tt MS}_t} & 
= - 2\,\mathcal{V}
&& \qquad \mbox{on }\ \Gamma^{\tt MS}_t\,, \label{eq:MSb} \\
v_{\tt MS} & = \alpha\,\varkappa && \qquad \mbox{on }\ \Gamma^{\tt MS}_t\,,
\label{eq:MSc} \\
\partial_\normal v_{\tt MS} & = 0 &&
\qquad \mbox{on } \partial \mathcal D\,,
\label{eq:MSd}\\
\Gamma^{\tt MS}_0 & = \Gamma_{00} \,, \label{eq:MSe} 
\end{alignat}
\end{subequations}
where $\varkappa$ is the \revised{curvature} of the evolving interface $\Gamma^{\tt MS}_t$, and $\mathcal{V}$ is the velocity in the direction of its normal ${\hsnormal}$, 
as well as $[\frac{\partial v_{\tt MS}}{\partial {\hsnormal}}]_{\Gamma^{\tt MS}_t}({z}) = 
(\frac{\partial v_{{\tt MS},+}}{\partial { \hsnormal}} - \frac{\partial v_{{\tt MS},-}}
{\partial {\hsnormal}})({z})$ 
for all ${z}\in\Gamma^{\tt MS}_t$. 
The constant in (\ref{eq:MSc}) is chosen as 
$\alpha = \tfrac12\,c_F$, where
$\cPsi = \int_{-1}^1 \sqrt{2\,F(s)}\;{\rm d}s = 
\tfrac13\,2^\frac32$,
and $F$ is the double-well potential; cf.~\cite{abc} for a further discussion of the model.

Below, we show that iterates $\{X^j\}_{j=1}^J$ of Scheme~\ref{scheme_time} converge to the limiting Mullins-Sekerka problem~(\ref{eq:MS}); 
see Theorem~\ref{cor_sharp} for a precise specification of the convergence result.
{For this purpose, we need sharper stability and convergence results than those available from Section \ref{asec_time}, which also requires to tighten the assumptions {\bf (B)}, and so to further restrict admissible choices of $\gamma>0$.}
We note that the stronger stability estimates below are derived formally using the (analytically) strong formulation of Scheme~\ref{scheme_time};
the derivation can be justified rigorously by the respective elliptic regularity of the Laplace and the bi-Laplace operators.

\begin{lem}\label{lem_linftybnd}
Assume {\bf (B)}.
For every \revp{$2<p<3$}, there exists $C \equiv C(p)>0$ such that the
solution $\{X^j\}_{j=1}^J$ of Scheme
\ref{scheme_time} satisfies
\begin{equation*}\label{est_linfty}
{\mathbb E}\bigl[ \max_{1 \leq j \leq J} \Vert X^j\Vert^p_{{\mathbb L}^{\infty}}\bigr] \leq
\revp{C\varepsilon^{1-p} k^{\frac{2-p}{2}}\,.} 
\end{equation*}
\end{lem}
\begin{proof}
{\bf 1.} {The second equation in} Scheme~\ref{scheme_time}$_2$ implies
$\sqrt{k}\|\Delta X^j(\omega)\| \leq 2\frac{\sqrt{k}}{\eps}\|w^j(\omega)\| + 2\frac{\sqrt{k}}{\eps^2}\|f(X^j(\omega))\|$,
for $\omega \in \Omega$.
Then Lemma~\ref{lem_energy}, ii), and Gagliardo-Nirenberg and Poincar\'e inequalities imply
\begin{eqnarray}\nonumber
\D \mathbb{E}\big[\max_{1\leq j \leq J} \sqrt{k}\|\Delta X^j\| \big] 
&\leq& 
\D \frac{C}{\eps}\mathbb{E}\Big[\Big(k\sum_{j=1}^{J}\|\nabla w^j\|^2\Big)^{1/2}\Bigr]
  + \frac{C\sqrt{k}}{\eps^2}\mathbb{E}\Big[\max_{1\leq j \leq J} \Big(\|X^j\|_{\mathbb{L}^6}^3 + \|X^j\|\Big)\Big]
\\ \label{h2_est}
&\leq& 
\D  \frac{C}{\eps} + \frac{C\sqrt{k}}{\eps^2}\mathbb{E}\Big[ \max_{1\leq j \leq J} \|X^j\|_{\mathbb{L}^4}^2\|\nabla X^j\|\Big]
\\ \nonumber
&\leq& 
\D  \frac{C}{\eps} + \frac{C\sqrt{k}}{\eps^2}\mathbb{E}\Big[\max_{1\leq j \leq J} \|X^j\|_{\mathbb{L}^4}^4\Big]^{1/2}\mathbb{E}\Big[\max_{1\leq j \leq J}\|\nabla X^j\|^2\Big]^{1/2}\, ,
\end{eqnarray}
which is bounded by $C \eps^{{-1}}$ for $k \leq \varepsilon^4$. 

{\bf 2.} Since ${\mathbb W}^{1,p} \hookrightarrow  {\mathbb L}^{\infty}$ ($p>2$),
by Gagliardo-Nirenberg \rev{inequality $\|\cdot\|_{\mathbb{L}^p}\leq C_p\|\cdot\|_{\mathbb{L}^2}^{\frac{2}{p}} \|\cdot\|_{\mathbb{H}^1}^{\frac{p-2}{p}}$ ($d=2$, $p>2$)},  H\"older inequality, 
Lemma~\ref{lem_energy}, iv), and step {\bf 1.}, we get for \revp{$2<p<3$}
\begin{eqnarray*}
\mathbb{E}\Big[{\max_{1\leq j\leq J}} \|X^j\|_{\mathbb{L}^\infty}^p\Big]
& \leq & \revp{ C \mathbb{E}\Big[\max_{1\leq j \leq J}\|\nabla X^j\|_{\mathbb{L}^p}^p\Big]
\leq C \mathbb{E}\Big[\max_{1\leq j \leq J}\|\nabla X^j\|^2\|\Delta X^j\|^{p-2}\Big]
}
\\
& \leq &
\revp{C\mathbb{E}\Big[\max_{1\leq j \leq J}\|\nabla X^j\|^{\frac{2}{3-p}}\Big]^{3-p}  \mathbb{E}\Big[\max_{1\leq j \leq J}\|\Delta X^j\|\Big]^{p-2}}
\\
& \leq &
\revp{C\eps^{-1}\mathbb{E}\Big[\eps^2\max_{1\leq j \leq J}\|\nabla X^j\|^{4}\Big]^{\frac{3-p}{2(3-p)}} k^{-\frac{p-2}{2}} \mathbb{E}\Big[\sqrt{k}\max_{1\leq j \leq J}\|\Delta X^j\|\Big]^{p-2}}
\\
& \leq & \revp{C\eps^{-1} k^{\frac{2-p}{2}}\eps^{-(p-2)} = C\eps^{1-p} k^{-\frac{p-2}{2}}}
\end{eqnarray*}
\qed
\end{proof}
The following lemma sharpens the statement of Lemma \ref{thma2} for iterates
$\{Z^j\}_{j=1}^J$, where $Z^j := X^j - X^j_{\tt CH}$. {It involves the
parameter ${\mathfrak n}_{\tt CH}>0$ from Lemma \ref{lem_xch}, ii).}

\begin{lem}\label{lem_err_l2l4}
{Suppose {\bf (B)}}.
There exists {$C >0$} such that
\begin{eqnarray*}\nonumber
&& {\mathbb E}\bigl[  \max_{1 \leq j \leq J}\Vert Z^j\Vert^2\bigr] + 
{\mathbb E}\Bigl[ \sum_{j=1}^J  \Vert Z^j - Z^{j-1}\Vert^2 
+ \varepsilon k \sum_{j=1}^J \Vert \Delta Z^j\Vert^2\Bigr] 
\\ 
\label{est_l2_err}
&& \quad + \frac{k}{\varepsilon} \sum_{j=1}^J {\mathbb E}\Bigl[  
 \Vert Z^j \nabla Z^j\Vert^2
 + \Vert X^j_{\tt CH} \nabla Z^j\Vert^2 \Bigr] 
 \leq {{\mathcal F}_1(k,\varepsilon; \sigma_0, \kappa_0, \gamma)} := \\
&&\qquad := C \max\Bigl\{ \Bigl(\frac{\max \bigl\{ \frac{k^2}{\varepsilon^4}, \varepsilon^{\gamma+\frac{\sigma_0+1}{3}}, \varepsilon^{\sigma_0},\varepsilon^{2\gamma}\bigr\}}{\varepsilon^{\kappa_0+10 + 4 {\mathfrak n}_{\tt CH}}} \Bigr)^{\frac{1}{2}}, \Bigl(\frac{\max \bigl\{ \frac{k^2}{\varepsilon^4}, \varepsilon^{\gamma+\frac{\sigma_0+1}{3}}, \varepsilon^{\sigma_0},\varepsilon^{2\gamma}\bigr\}}{\varepsilon^{\kappa_0+16}} \Bigr)^{\frac{1}{4}} \Bigr\} \, .
\end{eqnarray*}
\end{lem}
{In order to establish convergence to zero {(for $\varepsilon \downarrow 0$)} of the right-hand side in the inequality of the lemma, we need to impose a stronger assumptions than {\bf (B)}; for simplicity, we assume ${\mathfrak n}_{\tt CH} \geq \frac{3}{2}$ in Lemma \ref{lem_xch}:
\begin{itemize}
\item[{\bf (C$_1$)}] 
Assume {\bf (B)}, and that $(\sigma_0, \kappa_0, \gamma)$ also satisfies 
$$\sigma_0 >10 + \kappa_0 + 4 {\mathfrak n}_{\tt CH}\,, \qquad 
\gamma > \max \bigl\{\frac{2\kappa_0+19+8 {\mathfrak n}_{\tt CH}}{3} , \frac{\kappa_0 + 10 + 4{\mathfrak n}_{\tt CH}}{2} \bigr\}\, .$$
For sufficiently small $\varepsilon_0 \equiv (\sigma_0,\kappa_0) >0$ and ${\mathfrak l}_{\tt CH} \geq 3$ from Lemma \ref{lem_xch},
 and arbitrary $0 < \beta < \frac{1}{2}$ the time-step satisfies
$$
k \leq C \min \bigl\{\varepsilon^{{\mathfrak l}_{\tt CH}},
\varepsilon^{7+\frac{\kappa_0}{2}+2 {\mathfrak n}_{\tt CH} + \beta}\bigr\} \qquad \forall\, \varepsilon \in (0,\varepsilon_0) \,.
$$
\end{itemize}
\revised{Compared} to assumption {\bf (B)}, only larger values of $\sigma_0$, and consequently larger values of $\gamma$ are admitted, as well as smaller time-steps $k$.}

\begin{proof}
{\bf 1.} We subtract Scheme~\ref{scheme_time}  in strong form  for 
\revised{$g\not\equiv0$ and $g\equiv0$}, respectively, 
fix  $\omega\in \Omega$, and multiply the resulting error equations with $Z^j(\omega)$ and $-\Delta Z^j(\omega)$, respectively. We obtain
\begin{eqnarray}\nonumber
&&\frac{1}{2} \bigl( \Vert Z^j\Vert^2 - \Vert  Z^{j-1}\Vert^2 + \Vert Z^j - Z^{j-1} \Vert^2\bigr)
+ \varepsilon k \Vert \Delta Z^j\Vert^2 \\ \label{esti-three0}
&&\qquad 
 + \frac{k}{\varepsilon} \bigl( f(X^j)-f(X^j_{\tt CH}), -\Delta Z^j\bigr)  = \varepsilon^{\gamma}  \bigl( g,\rev{Z^j} \bigr) \Delta_j W\, .
\end{eqnarray}
We estimate the right-hand side above as
\rev{
\begin{eqnarray*}\nonumber
\varepsilon^{\gamma}  \bigl( g,\rev{Z^j} \bigr) \Delta_j W &=& \varepsilon^{\gamma}  \bigl( g, Z^j-Z^{j-1}\bigr) \Delta_j W   + \varepsilon^{\gamma} \bigl(g, Z^{j-1}\bigr) \Delta_j W 
\\ 
& \leq &\frac{1}{4} \Vert Z^j-Z^{j-1}\Vert^2 +
\varepsilon^{2\gamma}  \Vert g\Vert^2 \vert \Delta_j W\vert^2 +\varepsilon^{\gamma} \bigl( g, Z^{j-1}\bigr) \Delta_j W\, . 
\end{eqnarray*}
}
We restate the nonlinear term in (\ref{esti-three0}) as
\rev{
\begin{eqnarray*}\nonumber
\frac{k}{\varepsilon} \bigl( f(X^j)-f(X^j_{\tt CH}), -\Delta Z^j\bigr)
 &=& \frac{k}{\varepsilon} \Bigl( \vert X^j\vert^2 X^j - \rev{(\vert X^j_{\tt CH} \vert^2 X^j-\vert X^j_{\tt CH} \vert^2 X^j)} - \vert X_{\tt CH}^j\vert^2 X^j_{\tt CH}, -\Delta Z^j\Bigr) 
\\ & &- {\frac{k}{\varepsilon} (Z^j, -\Delta Z^j)}
\\ \label{esti-five1}
&=& \frac{k}{\varepsilon} \Bigl( Z^j [Z^j + 2 X^j_{\tt CH}] X^j { + \vert X^j_{\tt CH}\vert^2 Z^j}, -\Delta Z^j\Bigr) - {\frac{k}{\varepsilon} \|\nabla Z^j\|^2} \\ \nonumber
&=& \frac{k}{\varepsilon} \bigl( \vert Z^j\vert^2 Z^j, -\Delta Z^j \bigr)  
 -{\frac{k}{\varepsilon} \|\nabla Z^j\|^2}
\\ \nonumber
& & + \frac{3k}{\varepsilon} \bigl( \vert Z^j\vert^2 X^j_{\tt CH}  , -\Delta Z^j\bigr) + \frac{{3}k}{\varepsilon} \bigl(\vert X^j_{\tt CH}\vert^2 Z^j  , -\Delta Z^j\bigr)
\\ \nonumber
&=:& { \frac{3k}{\varepsilon} \Vert Z^j \nabla Z^j\Vert^2} -{\frac{k}{\varepsilon} \|\nabla Z^j\|^2} + {\tt I_1} + {\tt I_2}\, ,
\end{eqnarray*}
where in the last step we used integration by parts $\bigl( \vert Z^j\vert^2 Z^j, -\Delta Z^j \bigr) = 3\Vert Z^j \nabla Z^j\Vert^2$.
}

Next, we apply integration by parts to 
${\tt I_1}$, ${\tt I_2}$ to estimate
\begin{eqnarray*}
{\tt I}_{1} & := & \revdim{\frac{3k}{\varepsilon} \bigl( \vert Z^j\vert^2 X^j_{\tt CH}  , -\Delta Z^j\bigr)}
 = \frac{\revdim{3}k}{\varepsilon} \Bigl[2\bigl( Z^j \nabla Z^j X^j_{\tt CH}, \nabla Z^j\bigr) + \bigl( Z^j \nabla Z^j, Z^j \nabla X^j_{\tt CH}\bigr) \Bigr] \\
&\geq& - \frac{2k}{\varepsilon} \Bigl[C {\Vert X^j_{\tt CH}\Vert_{{\mathbb L}^{\infty}}}
 {\Vert \nabla Z^j\Vert}
+ {\Vert \nabla X^j_{\tt CH}\Vert_{{\mathbb L}^4}} {\Vert Z^j\Vert_{{\mathbb L}^4}}\Bigr] \Vert Z^j \nabla Z^j\Vert\,, \\
{\tt I}_{2} 
& := & \revdim{\frac{{3}k}{\varepsilon} \bigl(\vert X^j_{\tt CH}\vert^2 Z^j  , -\Delta Z^j\bigr)} 
\geq {\frac{3k}{\varepsilon} \Vert X^j_{\tt CH} \nabla Z^j\Vert^2} - \frac{6k}{\varepsilon} {\Vert Z^j}\Vert_{{\mathbb L}^{4}} 
{\Vert \nabla X^j_{\tt CH}\Vert_{{\mathbb L}^4}} \Vert X^j_{\tt CH}\nabla Z^j \Vert\, .
\end{eqnarray*}

Hence, using Poincar\'e, Sobolev and Young's inequalities, Lemma~\ref{lem_xch}, ii), and assumption {\bf (B)},  we deduce that
\begin{equation*}\label{esti-final}
\frac{k}{\varepsilon} \bigl( f(X^j)-f(X^j_{\tt CH}), -\Delta Z^j\bigr)
\geq {\frac{k}{2\varepsilon} \bigl[\Vert Z^j \nabla Z^j\Vert^2
 + \Vert X^j_{\tt CH} \nabla Z^j\Vert^2} \bigr] 
- \frac{Ck}{\varepsilon^{1+2{\mathfrak{n}}_{\tt CH}}} 
{\Vert \nabla Z^j\Vert^2}\, .
\end{equation*}
{\bf 2.} We insert these bounds into (\ref{esti-three0}), 
sum up over all time-steps, take $\max_{j\leq J}$ and expectations, 
\begin{eqnarray}\nonumber
&& \revdim{\frac{1}{2}}{\mathbb E}\bigl[ \max_{1 \leq j \leq J} \Vert Z^j\Vert^2\bigr] + 
{\mathbb E}\Bigl[ \sum_{j=1}^J \revdim{\frac{1}{4}} \Vert Z^j - Z^{j-1}\Vert^2 
+ \varepsilon k \sum_{j=1}^J \Vert \Delta Z^j\Vert^2\Bigr] 
\\ \label{est_xx1}
&& \quad + \frac{k}{2\varepsilon} \sum_{j=1}^J {\mathbb E}\Bigl[  { \Vert Z^j \nabla Z^j \Vert^2}
 + \Vert X^j_{\tt CH} \nabla Z^j\Vert^2 \Bigr] 
\\ \nonumber
&&\qquad  \leq 
 {\frac{C k}{\varepsilon^{1+{2\mathfrak{n}}_{\tt CH}}} 
 {\sum_{j=1}^J {\mathbb E}\bigl[ \Vert \nabla Z^j \Vert^2\bigr]}}
\revdim{+} \varepsilon^{\gamma}\mathbb{E}\Big[ \max_{1 \leq j \leq J} \sum_{i=1}^j\rev{\bigl( g, Z^{i-1}\bigr)} \Delta_i W\Big]
 + C \varepsilon^{2\gamma}\,.
\end{eqnarray}
{We use the discrete BDG-inequality (Lemma \ref{lembdg}) and the Poincar\'e inequality to estimate the last term as follows,
$$
\begin{array}{lll}
\rev{\D \varepsilon^{\gamma}\mathbb{E}\Big[\max_{1 \leq j \leq J} \sum_{i=1}^j  \Bigl( g, Z^{i-1}\bigr) \Delta_i W\Bigr)\Big]
\leq 
C\varepsilon^{\gamma} \| g\|_{\mathbb{L}^\infty} \mathbb{E}\Big[ k \sum_{j=1}^J \|\nabla Z^{j-1}\|^2\Big]^{\frac{1}{2}}}\,.
\end{array}
$$
We now use Lemma \ref{cor_err_xxa} to bound the right-hand side of  
(\ref{est_xx1}).}
\qed
\end{proof}

\medskip

{A crucial step in this section is to establish convergence of $\max_{1 \leq j \leq J}\Vert Z^j\Vert_{{\mathbb L}^{\infty}}$ for $\varepsilon \downarrow 0$; it turns out that this can only be validated on large subsets of $\Omega$, which motivates the introduction of the following (family of) subsets:}
For every \revp{$2< p < 3$}, we define
\begin{equation}\label{kap1} \kappa \equiv \kappa_p := \revp{\Bigl[ \varepsilon^{1-p} k^{\frac{2-p}{2}} \ln \bigl(\varepsilon^{1-p} \bigr)\Bigr]^{\frac{1}{p}}} \, ,
\end{equation}
and the sequence of sets $\{ {\Omega}_{\kappa, j}\}_{j=1}^J \subset \Omega$ via
\begin{equation}\label{cut1} 
{\Omega}_{\kappa, j} = \bigl\{ \omega \in \Omega: \, \max_{1 \leq \ell \leq j} \Vert X^\ell\Vert_{{\mathbb L}^{\infty}} \leq \kappa\bigr\} \qquad (\kappa >0)\, .
\end{equation}
Note that ${\Omega}_{\kappa, j} \subset {\Omega}_{\kappa, j-1}$.
Markov's inequality yields that
\begin{equation}\label{mark1}
{\mathbb P}\bigl[{\Omega}_{\kappa,j}\bigr]
\geq 1- \frac{{\mathbb E}[\max_{1 \leq \ell \leq j} \Vert X^\ell\Vert^p_{{\mathbb L}^{\infty}}]}{\kappa^p}\, .
\end{equation}
Clearly, $\D {\lim_{\revised{\eps \downarrow 0}}}\min_{1\leq j\leq J}\mathbb{P}[\Omega_{\kappa,j}] = 1$ by Lemma~\ref{lem_linftybnd}.

We use Lemma~\ref{lem_err_l2l4} to show a local error estimate.
\begin{lem}\label{lem_err_nablaz}
{Assume {\bf (B)}  and \revp{$2<p<3$}.}
Then there exists $C>0$ such that 
\begin{eqnarray*}
&& {\mathbb E}\bigl[ { \max_{0\leq j \leq J}} \mathbbm{1}_{{\Omega}_{\kappa,j}}\Vert \nabla Z^j\Vert^2\bigr]
\leq {\mathcal F}_2({k, \varepsilon}; \sigma_0, \kappa_0, \gamma) :=  \\
&&\quad := C\max \Bigl\{ \frac{(1+\kappa^2)}{\varepsilon^2} {{\mathcal F}_1 \bigl(k, \varepsilon; \sigma_0, \kappa_0, \gamma \bigr)} ,\frac{(1+ \kappa^2)}{\varepsilon^{7+ 2{\mathfrak n}_{\tt CH}}} \Bigl(\frac{1}{\varepsilon^{\kappa_0}} \max \bigl\{ \frac{k^2}{\varepsilon^4}, \varepsilon^{\gamma+\frac{\sigma_0+1}{3}}, \varepsilon^{\sigma_0},\varepsilon^{2\gamma}\bigr\}\Bigr)^{\frac{1}{4}}\Bigr\}\, .
\end{eqnarray*}
\end{lem}
{In order to establish convergence to zero {(for $\varepsilon \downarrow 0$)} of the right-hand side in the inequality of the lemma, we  impose again a stronger assumptions
than {\bf (C$_1$)}:
\begin{itemize}
\item[{\bf (C$_2$)}] Assume {\bf (C$_1$)}, and that $(\sigma_0, \kappa_0, \gamma)$, and {$k$} satisfy  
\begin{equation}\label{bass-C2}\lim_{\varepsilon \downarrow 0} {\mathcal F}_2({k, \varepsilon}; \sigma_0, \kappa_0, \gamma) =0\,.
\end{equation}
\end{itemize}
{\begin{rem}\label{assu-ok}
A strategy to identify admissible \revised{quadruples $(\sigma_0, \kappa_0, \gamma, k)$} which meet assumption {\bf (C$_2$)} is as follows:
\begin{itemize}
\item[(1)] assumption {\bf (C$_1$)} establishes $\lim_{\varepsilon \downarrow 0} {{\mathcal F}_1(k,\varepsilon; \sigma_0, \kappa_0, \gamma)} = 0$, which appears as \revised{a} factor in the first term on the right-hand side in Lemma \ref{lem_err_nablaz}.
\item[(2)] the leading factor in $\mathcal{F}_2$ is \revp{$\frac{\kappa^2}{\varepsilon^2} \equiv
 \frac{\kappa_p^2}{\varepsilon^2} \leq \varepsilon^{\frac{1-3p}{p}}  \bigl\vert \ln (\varepsilon^{1-p} )\bigr\vert^{\frac{2}{p}}
                                        k^{\frac{2-p}{p}}$, for $2<p<3$} via (\ref{kap1}).
To meet (\ref{bass-C2}) therefore additionally requires {\em for some $p>2$}
\begin{equation}
\label{tar1} \revp{k^{\frac{2-p}{p}}} {{\mathcal F}_1(k,\varepsilon; \sigma_0, \kappa_0, \gamma)}\varepsilon^{\revp{\frac{1-3p}{p}}}  
           \bigl\vert \ln (\varepsilon^{1-p} )\bigr \vert^{\frac{2}{p}} \rightarrow 0 \qquad (\varepsilon \downarrow 0)\, ,
\end{equation}
and hence
\begin{equation}\label{add-1a}
\Bigl[{{\mathcal F}_1(k,\varepsilon; \sigma_0, \kappa_0, \gamma)}\varepsilon^{ \revp{\frac{1-3p}{p}}} 
            \bigl\vert \ln (\varepsilon^{1-p} )\bigr\vert^{\frac{2}{p}}\Bigr]^{\revp{\frac{p}{p-2}}} = o(k)\, . 
\end{equation}
A proper scenario is $k = \varepsilon^{\alpha}$ for some 
$\alpha >0$ to meet assumption ${\bf (C}_1{\bf )}$. We then sharpen
this choice of the time-step to
$k = \varepsilon^{\widetilde{\alpha}}$ for some $\widetilde{\alpha} \geq \alpha >0$
to have 
$${{\mathcal F}_1(k,\varepsilon; \sigma_0, \kappa_0, \gamma)}\revp{\varepsilon^{\frac{1-3p}{p}}}  \ln^{\frac{2}{p}}\bigl(\varepsilon^{1-p} \bigr) \leq \varepsilon^\eta $$
for an arbitrary $\eta > 0$. We now choose $2<p$, s.t. $\revp{\frac{p}{p-2}} \gg 0$ is sufficiently large to meet (\ref{add-1a}).
%
\item[(3)] We may proceed analogously for the second term on the right-hand side in Lemma \ref{lem_err_nablaz}.
\end{itemize}
\end{rem}
}}

\begin{proof}
We subtract Scheme~\ref{scheme_time} for \revised{$g\not\equiv0$ and $g\equiv0$} for a fixed $\omega\in \Omega$, and
multiply the first error equation with $-\Delta Z^j(\omega)$, and the second with
$\Delta^2 Z^j(\omega)$. We \rev{integrate by parts in the nonlinear term} and obtain
\begin{eqnarray}\nonumber
&&\frac{1}{2} \bigl( \Vert \nabla Z^j\Vert^2 - \Vert \nabla Z^{j-1}\Vert^2 + \Vert \nabla [Z^j - Z^{j-1}]\Vert^2\bigr)
+ \varepsilon k \Vert \nabla \Delta Z^j\Vert^2 
\\ \label{esti-three}
&&\qquad = 
{\frac{k}{\varepsilon} \bigl( \nabla [f(X^j)-f(X^j_{\tt CH})], \nabla \Delta Z^j\bigr)} 
+ \varepsilon^{\gamma} {\bigl( g, \rev{-\Delta} Z^j\bigr) \Delta_j W}  =: {\tt I} + {\tt II}\, .
\end{eqnarray}
We proceed as in the proof of Lemma~\ref{lem_err_l2l4} and rewrite the nonlinearity on the right-hand side as
 \begin{eqnarray*}\nonumber
 {\tt I}  &=& \frac{k}{\varepsilon} \bigl( \nabla [\vert Z^j\vert^2 Z^j], \nabla \Delta Z^j \bigr)  
+ \frac{3k}{\varepsilon} \bigl( \nabla[\vert Z^j\vert^2 X^j_{\tt CH}]  , \nabla\Delta Z^j\bigr) 
\\
\label{erro-six1}
&&+ \frac{{3}k}{\varepsilon} \bigl(\nabla [\vert X^j_{\tt CH}\vert^2 Z^j], \nabla \Delta Z^j\bigr) + \rev{\frac{k}{\varepsilon}\|\Delta Z^j\|^2}
\\
& =: & {\tt I}_1 + {\tt I}_2 + {\tt I}_3 + \rev{\frac{k}{\varepsilon}\|\Delta Z^j\|^2} \, .
 \end{eqnarray*}
 We estimate
 \begin{eqnarray*}
 {\tt I}_1 &\leq& \frac{Ck}{\varepsilon^3} {\Vert Z^j\Vert_{{\mathbb L}^{\infty}}^2} {\Vert Z^j \nabla Z^j\Vert^2}
+ \frac{\varepsilon k}{8} \Vert \nabla \Delta Z^j\Vert^2\,, 
\\
 {\tt I}_2 &\leq& \frac{Ck}{\varepsilon^3} \bigl( {\Vert Z^j\Vert^2_{{\mathbb L}^{\infty}}} 
          {\Vert X^j_{\tt CH}\Vert^2_{{\mathbb L}^{\infty}}} \Vert \nabla Z^j\Vert^2 + {\Vert Z^j\Vert^2_{{\mathbb L}^\infty}} 
               \Vert Z^j\Vert^2_{{\mathbb L}^4} {\Vert \nabla X^j_{\tt CH}\Vert^2_{{\mathbb L}^4}}\bigr)
 + \frac{\varepsilon k}{8} \Vert \nabla \Delta Z^j\Vert^2\,,
\\
 {\tt I}_3 &\leq& 
 \frac{Ck}{\varepsilon^3} \bigl({\Vert X^j_{\tt CH}\Vert^4_{{\mathbb L}^{\infty}}}\Vert \nabla Z^j\Vert^2_{{\mathbb L}^{2}} + {\Vert X^j_{\tt CH}\Vert^2_{{\mathbb L}^{\infty}}
 \Vert \nabla X^j_{\tt CH}\Vert^2_{{\mathbb L}^{4}}}
 \Vert Z^j\Vert^2_{{\mathbb L}^4}\bigr)
 + \frac{\varepsilon k}{8} \Vert \nabla \Delta Z^j\Vert^2\, .
 \end{eqnarray*}
We estimate $\sum_{\ell=1}^3{\tt I}_{\ell}$ on $\Omega_{\kappa,j}$ via Lemma~\ref{lem_xch}, \rev{ii)-iii)} and the embedding ${\mathbb H}^1 \hookrightarrow  {\mathbb L}^{4}$ 
\rev{on recalling (\ref{cut1})}
\begin{equation}\label{estiii}
\D \mathbbm{1}_{\Omega_{\kappa,j}} \sum_{\ell=1}^3{\tt I}_{\ell}\leq  \D \mathbbm{1}_{\Omega_{\kappa,j}}\big\{ \frac{\varepsilon k}{2} \Vert \nabla \Delta Z^j\Vert^2 
+\frac{C(1+{\kappa^2})k}{\varepsilon^3} {\Vert Z^j \nabla Z^j\Vert^2}
\D  \D 
+ \frac{C(1+ {\kappa^2}) k}{\varepsilon^{3+2\mathfrak{n}_{\tt CH}}} \Vert \nabla Z^j\Vert^2\big\}\, .
\end{equation}
We multiply (\ref{esti-three}) by $\mathbbm{1}_{\Omega_{\kappa,j}}$, sum up for  $1 \leq i \leq j$, take $\max_{1\leq j \leq J}$ and expectation,
employ the identity \revdim{(recall, $\mathbbm{1}_{\Omega_{\kappa,j-1}} - \mathbbm{1}_{\Omega_{\kappa,j}}\geq 0$)}
 \begin{eqnarray*}\nonumber
&& \frac{1}{2}  {\mathbb E}\Bigl[ {\max_{0\leq j\leq J}\sum_{i=1}^j}\Big( \mathbbm{1}_{\Omega_{\kappa,i}} \bigl( \Vert \nabla Z^j\Vert^2 
- \Vert \nabla Z^{i-1}\Vert^2\bigr) - \rev{\big(\mathbbm{1}_{\Omega_{\kappa,i-1}}\Vert \nabla Z^{i-1}\Vert^2-\mathbbm{1}_{\Omega_{\kappa,i-1}}\Vert \nabla Z^{i-1}\Vert^2\big)}\Big) \Bigr] \\ \label{term_1}
&&\quad = \frac{1}{2} {\mathbb E}\Bigl[ \max_{0\leq j\leq J} \mathbbm{1}_{\Omega_{\kappa,j}}\Vert \nabla Z^j\Vert^2\Bigr]
+ {\frac{1}{2} \sum_{j=1}^J {\mathbb E}\Bigl[ {\bigl(\mathbbm{1}_{\Omega_{\kappa,j-1}} - \mathbbm{1}_{\Omega_{\kappa,j}}\bigr)}
\Vert \nabla Z^{j-1}\Vert^2\Bigr]}\,,
 \end{eqnarray*}
use  Lemmata~\ref{lem_err_l2l4} and \ref{cor_err_xxa} to estimate (\ref{estiii}) and obtain
 \begin{eqnarray}\nonumber
&& \frac{1}{2} {\mathbb E}\Bigl[ {\max_{0\leq j\leq J}} \mathbbm{1}_{\Omega_{\kappa,j}}\Vert \nabla Z^j\Vert^2\Bigr]
+ \frac{1}{2} \sum_{j=1}^J {\mathbb E}\Bigl[ {\bigl(\mathbbm{1}_{\Omega_{\kappa,j-1}} - \mathbbm{1}_{\Omega_{\kappa,j}}\bigr)}
\Vert \nabla Z^{j-1}\Vert^2\Bigr] \\ \label{mark2}
&&\qquad  
+ \frac{1}{2} \sum_{j=1}^J {\mathbb E}\Bigl[\mathbbm{1}_{\Omega_{\kappa,j}} 
\bigl(\Vert \nabla[Z^j - Z^{j-1}]\Vert^2 
+ \varepsilon k \Vert \nabla \Delta Z^j\Vert^2\bigr)\Bigr] \\ \nonumber
\nonumber
&&\quad \leq {\max \Bigl\{ \frac{C(1+\kappa^2)}{\varepsilon^2} {{\mathcal F}_1 \bigl(k, \varepsilon; \sigma_0, \kappa_0, \gamma \bigr)},\frac{C(1+ \kappa^2)}{\varepsilon^{7+ 2{\mathfrak n}_{\tt CH}}} \Bigl(\frac{C}{\varepsilon^{\kappa_0}} \max \bigl\{ \frac{k^2}{\varepsilon^4}, \varepsilon^{\gamma+\frac{\sigma_0+1}{3}}, \varepsilon^{\sigma_0},\varepsilon^{2\gamma}\bigr\}\Bigr)^{\frac{1}{4}}\Bigr\}} \\ \nonumber
&&\qquad +
\eps^\gamma \mathbb{E}\Big[\max_{0\leq j\leq J} \sum_{i=1}^j \mathbbm{1}_{\Omega_{\kappa,i}}\bigl(g,  \rev{-\Delta} Z^{i}\bigr) \Delta_i W\Big]\,.
\end{eqnarray}
To estimate the stochastic term we use {\rev{$\partial_\normal g = 0$} on $\partial {\mathcal D}$}
and proceed as follows,
\begin{eqnarray*}\nonumber 
&& \eps^\gamma \mathbb{E}\Bigl[ \max_{0\leq j \leq J} \sum_{i=1}^j \mathbbm{1}_{\Omega_{\kappa,i}}\bigl(\rev{-\Delta g},  Z^{i}\bigr) \Delta_i W\Bigr]
  = 
\varepsilon^{\gamma} \mathbb{E}\Bigl[ \max_{0\leq j\leq J} \sum_{i=1}^j\Big(  \mathbbm{1}_{\Omega_{\kappa,i}}  {\bigl(\rev{-\Delta g},  Z^i-Z^{i-1}\bigr) \Delta_i W}
\\ \nonumber
&& 
\qquad   + {\mathbbm{1}_{\Omega_{\kappa,i-1}}  \bigl(\nabla \rev{g},  \nabla Z^{i-1}\bigr) \Delta_i W} 
{+ 
{\bigl( \mathbbm{1}_{\Omega_{\kappa,i}} - \mathbbm{1}_{\Omega_{\kappa,i-1}}\bigr)
{\bigl(\nabla g,  \nabla Z^{i-1}\bigr)} \Delta_i W}}\Big)
 \Bigr]
\\ \nonumber
&& \leq 
\frac{\eps^\gamma}{2}\sum_{i=1}^J {\mathbb E}\Bigl[ \Vert Z^i-Z^{i-1}\Vert^2 + \rev{\|\Delta g\|^2}\vert \Delta_j W\vert^2  \Bigr]
+
\varepsilon^{\gamma}\mathbb{E}\Bigl[ \max_{0\leq j\leq J} \sum_{i=1}^j
{\mathbbm{1}_{\Omega_{\kappa,i-1}}  \bigl(\rev{\nabla g},  \nabla Z^{i-1}\bigr) \Delta_i W}  \Bigr]
\\ \label{term_3}
&& 
\qquad + {\frac{1}{4} \sum_{i=1}^J {\mathbb E}\bigl[  
\bigl( \mathbbm{1}_{\Omega_{\kappa,i}} - \mathbbm{1}_{\Omega_{\kappa,i-1}}\bigr)^2 \Vert \nabla Z^{i-1}\Vert^2\bigr]}
 + C\varepsilon^{\rev{2\gamma}} k \sum_{i=1}^J {\mathbb E}\bigl[ \Vert \rev{\nabla g}\Vert^2 \bigr]
\ .
\end{eqnarray*}
The first term on the right-hand side may be bounded by Lemma~\ref{lem_err_l2l4},
the third term is absorbed in the \rev{left-hand} side of (\ref{mark2}),
and for the second term we use the discrete BDG-inequality (Lemma \ref{lembdg})
and Lemma~\ref{cor_err_xxa} to estimate
\begin{eqnarray*}
&&\varepsilon^{\gamma}\mathbb{E}\Bigl[ \max_{0\leq j\leq J} \sum_{i=1}^j
{\mathbbm{1}_{\Omega_{\kappa,i-1}}  \bigl(\nabla \rev{g},  \nabla Z^{i-1}\bigr) \Delta_i W}  \Bigr] \\
&&\qquad \leq 
C \varepsilon^{\gamma}\|\rev{\nabla g}\|_{\mathbb{L}^\infty} \mathbb{E}\Bigl[ k\sum_{i=1}^J  \|\nabla Z^{i-1}\|^2 \Bigr]^{\frac{1}{2}} \leq {\frac{C \eps^{\gamma}}{\varepsilon^2} \Bigl(\frac{C}{\varepsilon^{\kappa_0}} \max \bigl\{ \frac{k^2}{\varepsilon^4}, \varepsilon^{\gamma+\frac{\sigma_0+1}{3}}, \varepsilon^{\sigma_0},\varepsilon^{2\gamma}\bigr\}\Bigr)^{\frac{1}{4}}}\, .
\end{eqnarray*}
\rev{Hence, the statement of the lemma follows from (\ref{mark2}) and the above estimates
 on noting that $(\mathbbm{1}_{\Omega_{\kappa,j}} - \mathbbm{1}_{\Omega_{\kappa,j-1}})^2 = \mathbbm{1}_{\Omega_{\kappa,j-1}} - \mathbbm{1}_{\Omega_{\kappa,j}} \geq 0$.}
\qed
\end{proof}

The $\mathbb{L}^\infty$-estimate in the next theorem is a crucial ingredient to show convergence
to the sharp-interface limit.
\begin{thm}\label{thm_err_linfty}
 Assume {\bf (C$_2$)}. For any \revp{$2<p<3$}, there exists $C\equiv C(p) >0$ such that
\begin{eqnarray*}\label{err_linfty}
{\mathbb E}\Bigl[ {\max_{0\leq j \leq J}} \mathbbm{1}_{{\Omega}_{\kappa,j}}\Vert  Z^j\Vert_{{\mathbb L}^\infty}^p\Bigr]
\leq 
\revp{C\eps^{-\frac{p}{2}} k^{\frac{2-p}{2}} \big(\mathcal{F}_2 (k, \varepsilon; \sigma_0, \kappa_0, \gamma)\big)^{3-p} \big(\mathcal{F}_1 (k, \varepsilon; \sigma_0, \kappa_0, \gamma)\big)^{\frac{p-2}{2}}\,.}
\end{eqnarray*}
\end{thm}
\begin{proof}
We proceed analogically as in step {\bf 2.}~in the proof of Lemma~\ref{lem_linftybnd}. We use the Sobolev and Gagliardo-Nirenberg inequalities, apply H\"older inequality twice;
then use Lemma~\ref{lem_err_nablaz}, Lemma~\ref{lem_err_l2l4}  \revdim{(i.e., $\mathbb{E}\big[\eps k\|\Delta Z^j\|^2\big] \leq C$)}
along with the triangle inequality in combination with Lemma~\ref{lem_xch}~i), Lemma~\ref{lem_energy}~iv) and get for \revp{$2<p<3$} that
\begin{eqnarray*}
\mathbb{E}\Big[{\max_{1\leq j\leq J}} \mathbbm{1}_{{\Omega}_{\kappa,j}} \|Z^j\|_{\mathbb{L}^\infty}^p\Big]
& \leq & \revp{ C \mathbb{E}\Big[\max_{1\leq j \leq J}\mathbbm{1}_{{\Omega}_{\kappa,j}}\|\nabla Z^j\|_{\mathbb{L}^p}^p\Big]
\leq C \mathbb{E}\Big[\max_{1\leq j \leq J}\mathbbm{1}_{{\Omega}_{\kappa,j}} \|\nabla Z^j\|^2\|\Delta Z^j\|^{p-2}\Big]
}
\\
& \leq &
\revp{C \mathbb{E}\Big[\max_{1\leq j \leq J}\mathbbm{1}_{{\Omega}_{\kappa,j}}\|\nabla Z^j\|^{\frac{2(3-p)}{3-p}}\Big]^{3-p}  \mathbb{E}\Big[\max_{1\leq j \leq J}\|\nabla Z^j\|^{\frac{2(p-2)}{p-2}} \|\Delta Z^j\|\Big]^{p-2}}
\\
& \leq &
\revp{C \big(\mathcal{F}_2 (k, \varepsilon; \sigma_0, \kappa_0, \gamma)\big)^{3-p} \mathbb{E}\Big[\max_{1\leq j \leq J}\|\nabla Z^j\|^{4}\Big]^{1/2} (\eps k)^{-\frac{p-2}{2}} \mathbb{E}\Big[\eps k\|\Delta Z^j\|^2\Big]^{\frac{p-2}{2}}}
\\
& \leq &
\revp{ C(\eps k)^{\frac{2-p}{2}} \big(\mathcal{F}_2 (k, \varepsilon; \sigma_0, \kappa_0, \gamma)\big)^{3-p} \big(\mathcal{F}_1 (k, \varepsilon; \sigma_0, \kappa_0, \gamma)\big)^{\frac{p-2}{2}}}
\\
& &
\revp{\eps^{-1} \Big(\max_{1\leq j \leq J}\eps\|\nabla \XCH^j\|^{2} + \mathbb{E}\Big[\eps^2\max_{1\leq j \leq J}\|\nabla X^j\|^{4}\Big]^{1/2}\Big) }
\\
& \leq & \revp{C\eps^{-\frac{p}{2}} k^{\frac{2-p}{2}} \big(\mathcal{F}_2 (k, \varepsilon; \sigma_0, \kappa_0, \gamma)\big)^{3-p} \big(\mathcal{F}_1 (k, \varepsilon; \sigma_0, \kappa_0, \gamma)\big)^{\frac{p-2}{2}}}\,.
\end{eqnarray*}
\qed
\end{proof}

In order to establish convergence to zero {(for $\varepsilon \downarrow 0$)} of the right-hand side in the inequality of the theorem, we  impose again a stronger assumption
than {\bf (C$_2$)}:
\begin{itemize}
\item[{\bf (C$_3$)}] Assume {\bf (C$_2$)}, and that $(\sigma_0, \kappa_0, \gamma)$, and {$k$} satisfy  
\begin{equation}\label{ass-C3}
\lim_{\varepsilon \downarrow 0} \revp{\Bigl[\eps^{-p} k^{2-p} \big(\mathcal{F}_2 (k, \varepsilon; \sigma_0, \kappa_0, \gamma)\big)^{6-2p} 
\big(\mathcal{F}_1 (k, \varepsilon; \sigma_0, \kappa_0, \gamma)\big)^{p-2}]^{\frac{1}{2}}} = 0\,.
\end{equation}
\end{itemize}

{\begin{rem}\label{assu-ok2}
We discuss a strategy to identify admissible {quadruples $(\sigma_0, \kappa_0, \gamma,k)$} which meet assumption {\bf (C$_3$)}: for this purpose, we limit {ourselves} to a discussion
of the leading term inside the maximum which defines
${\mathcal F}_2$ (see Lemma \ref{lem_err_nablaz}), and recall Remark \ref{assu-ok}.
\begin{itemize}
\item[(1)] To meet (\ref{ass-C3})  instead of (\ref{tar1}), we have to ensure that {\em for some \revp{$2<p<3$}}
$$
\eps^{-\frac{p}{2}} k^{\frac{2-p}{2}} \Big(\revp{k^{\frac{2-p}{p}}} \varepsilon^{\revp{\frac{1-3p}{p}}}  
           \bigl\vert \ln (\varepsilon^{1-p} )\bigr \vert^{\frac{2}{p}}\Big)^{3-p}\Big({{\mathcal F}_1(k,\varepsilon; \sigma_0, \kappa_0, \gamma)}\Big)^{\frac{4-p}{2}}
\rightarrow 0 \qquad (\varepsilon \downarrow 0)
$$
and hence
\begin{equation*}
\revp{\Bigl[
\Big({{\mathcal F}_1(k,\varepsilon; \sigma_0, \kappa_0, \gamma)}\Big)^{\frac{4-p}{2}}
\eps^{\revp{-\frac{p}{2}}}\varepsilon^{ \revp{\frac{(1-3p)(3-p)}{p}}}\bigl\vert \ln(\varepsilon^{1-p} ) \bigr\vert^{\frac{2(3-p)}{p}}
\Bigr]^{\revp{\frac{2p}{(2-p)(6-p)}}}} = o(k)\, . 
\end{equation*}
\item[(2)] We may now proceed as in (2) in Remark \ref{assu-ok} to identify proper choices $k = \varepsilon^{{\alpha}}$ ($\alpha >0$)
and \rev{$p = 2+\delta$, for sufficiently small $\delta>0$,} that guarantee (\ref{ass-C3}).
%
\end{itemize}
\end{rem}}

We are now ready \revised{to} formulate the second main result of this paper, which is
convergence in probability of the solution $\{X^j\}_{j=0}^J$ of Scheme~\ref{scheme_time} 
to the solution of the deterministic Hele-Shaw/Mullins-Sekerka problem (\ref{eq:MS})
for $\varepsilon \downarrow 0$, provided that assumption {\bf (C$_3$)} is valid, and
(\ref{eq:MS}) has a classical solution; cf.~Theorem~\ref{cor_sharp} below. The proof rests on 
\begin{enumerate}
\item[a)] the uniform bounds for $\{\mathbbm{1}_{{\Omega}_{\kappa,j}}\Vert Z^j\Vert^p_{{\mathbb L}^{\infty}}\}_{j=1}^J$ (see Theorem \ref{thm_err_linfty}), and the property that ${\lim_{\revised{\eps \downarrow 0}}}\max_{1\leq j\leq J}\mathbb{P}[\Omega_{\kappa,j}] = 1$ (in Lemma~\ref{lem_linftybnd}) for the sequence
$\{ \Omega_{\kappa,j}\}_{j=1}^J \subset \Omega$, and
\item[b)] a convergence result for $\{ X^j_{\tt CH}\}_{j=0}^J$ towards a smooth solution of the Hele-Shaw/Mullins-Sekerka problem in \cite[Section 4]{fp_ifb05}.
\end{enumerate}
For each $\eps\in (0,\eps_0)$ we consider below the piecewise affine interpolant in time of the iterates $\{ X^j\}_{j=0}^J$ of Scheme~\ref{scheme_time} via
\begin{equation}\label{interpol-1}
X^{\eps,k}(t) := \frac{t-t_{j-1}}{k}X^{j} + \frac{t_{j}-t}{k}X^{j-1}\qquad\mathrm{for}\quad t_{j-1}\leq t \leq t_{j} \, .
\end{equation}
Let $\Gamma_{00} \subset {\mathcal D}$ in (\ref{eq:MSe}) be a smooth closed curve, and $(v_{\tt MS}, \Gamma^{\tt MS})$ be a smooth solution of (\ref{eq:MS}) starting from $\Gamma_{00}$, where $\Gamma^{\tt MS} := \bigcup_{0 \leq t \leq T} \{t\} \times \Gamma^{\tt MS}_t$.
Let ${\rm d}(t,{x})$ denote the signed distance function to $\Gamma^{\tt MS}_t$ such that 
${\rm d}(t,{x}) < 0$ in ${\mathcal I}^{\tt MS}_t$, the inside of $\Gamma^{\tt MS}_t$, and ${\rm d}(t, { x})>0$ on ${\mathcal O}^{\tt MS}_t := {\mathcal D} \setminus (\Gamma^{\tt MS}_t \cap {\mathcal I}^{\tt MS}_t)$, the outside of $\Gamma^{\tt MS}_t$. We also define the inside ${\mathcal I}^{\tt MS}$ and the outside ${\mathcal O}^{\tt MS}$, 
$${\mathcal I}^{\tt MS} := \bigl\{ (t, {x}) \in \overline{{\mathcal D}_T}:\, {\rm d}(t,{x}) < 0\bigr\}\,, \qquad
{\mathcal O}^{\tt MS} := \bigl\{ (t, {x}) \in \overline{{\mathcal D}_T}:\, {\rm d}(t,{x}) >0\bigr\}\,. $$
For the numerical solution $X^{\eps,k} \equiv X^{\eps,k}(t,x)$, we denote the zero level set at time $t$ by $\Gamma^{\eps,k}_t$, that is,
$$\Gamma_t^{\eps,k} := \bigl\{ x \in {\mathcal D}:\, X^{\eps,k}(t,x) = 0\bigr\} \qquad (0 \leq t \leq T)\, . $$
We summarize the assumptions needed below concerning the Mullins-Sekerka problem (\ref{eq:MS}).
\begin{itemize}
\item[{\bf (D)}]
Let ${\mathcal D} \subset {\mathbb R}^2$ be a smooth  domain.
There exists a classical solution $(v_{\tt MS},\Gamma^{\tt MS})$ of (\ref{eq:MS}) evolving from $\Gamma_{00} \subset {\mathcal D}$, such that $\Gamma^{\tt MS}_t \subset {\mathcal D}$ for all $t \in [0,T]$.
\end{itemize}
By \cite[Theorem 5.1]{abc}, assumption {\bf (D)} establishes the existence of a family of smooth solutions $\{ u_0^\eps\}_{0 \leq \eps \leq 1}$ which are uniformly bounded in $\varepsilon$ and $(t, x)$, such that if $u^\varepsilon_{\tt CH}$ is the corresponding solution of (\ref{p1}) with $g \equiv 0$, then
\begin{eqnarray*}
{\rm i)}&& \lim_{\varepsilon \downarrow 0} u^{\varepsilon}_{\tt CH}(t,x) =
\left\{\begin{array}{l} +1  \quad \mbox{if } (t,x) \in{\mathcal O}^{\tt MS}\,, \\
-1  \quad \mbox{if } (t,x) \in{\mathcal I}^{\tt MS}\,,
\end{array}\right.
\qquad \mbox{uniformly on compact subsets of } {\mathcal D}_T\, , \\
{\rm ii)}&& \lim_{\varepsilon \downarrow 0} \bigl( \frac{1}{\varepsilon}f(u^\varepsilon_{\tt CH}) -
\varepsilon \Delta u^\varepsilon_{\tt CH}\bigr)(t,x) = v^{\tt MS}(t,x) \quad
\quad \ \mbox{uniformly on } {\mathcal D}_T\, .
\end{eqnarray*}
%
%
The following theorem establishes uniform convergence of
iterates $\{ X^j\}_{j=0}^J$ from Scheme \ref{scheme_time} in probability on
the sets ${\mathcal I}^{\tt MS}$, ${\mathcal O}^{\tt MS}$.
\begin{thm}\label{cor_sharp}
Assume {\bf (C$_3$)} and {\bf (D)}. Let $\{ X^\eps\}_{0 \leq \varepsilon \leq \varepsilon_0}$ in (\ref{interpol-1}) be obtained via Scheme \ref{scheme_time}.
Then
\begin{eqnarray*}
{\rm i)}&&   \lim_{\eps \downarrow 0}\, \mathbb{P}\left[\bigl\{
 \Vert X^{\eps,k} -   1\Vert_{C({\mathcal A})} > \alpha \quad \mathrm{for\,\,all\,\,}  {\mathcal A} \Subset \mathcal{O}^{\tt MS}
\bigr\} \right] = 0 \qquad \forall\, \alpha > 0\, , \\
{\rm ii)}&&   \lim_{\eps \downarrow 0}\, \mathbb{P}\left[\bigl\{
 \Vert X^{\eps,k} +   1\Vert_{C({\mathcal A})} > \alpha \quad \mathrm{for\,\,all\,\,}  {\mathcal A} \Subset \mathcal{I}^{\tt MS}
 \bigr\}\right] = 0 \qquad \, \forall\, \alpha > 0\, .
 \end{eqnarray*}
 \end{thm}
\begin{proof}
We decompose 
$\overline{{\mathcal D}_T} \setminus \Gamma = {\mathcal I}^{\tt MS} \cup {\mathcal O}^{\tt MS}$, and consider
related errors $X^{\varepsilon,k}_{\tt CH} + 1$, $X^{\varepsilon,k}_{\tt CH} - 1$ and $X^{\varepsilon,k} - X^{\varepsilon,k}_{\tt CH}$.

{\bf 1.} By \cite[Theorem 4.2]{fp_ifb05}\footnote{Note that the mesh requirement $k = {\mathcal O}(h^q)$ stated in \cite[Theorem 4.2]{fp_ifb05} does {\em not} apply for the semi-discretization in time of (\ref{p1}) with $g \equiv 0$. In fact, in \cite{fp_ifb05} -- where \revised{the} involved parameters $k,h,\varepsilon$ tend to zero simultaneously -- the given constraint goes back to requirement \cite[Theorem 3.1, 3)]{fp_ifb05} which uses \cite[(3.28)]{fp_ifb05}, where we formally send $h \downarrow 0$ first (with $\mu = \nu=\delta=1$, $N=2$) to address our case.}, the piecewise affine interpolant $X^{\eps,k}$ of   $\{X^j_{\tt CH}\}_{j=0}^J$ satisfies
\begin{eqnarray*}\label{xch_sharp}
{\rm i')}&& X^{\eps,k}_{\tt CH} \rightarrow  +1 \quad 
\mbox{uniformly on compact subsets of } {\mathcal O}^{\tt MS} \qquad (\eps \downarrow 0)\, , \\
{\rm ii')}&& X^{\eps,k}_{\tt CH} \rightarrow  -1 \quad 
\mbox{uniformly on compact subsets of } {\mathcal I}^{\tt MS} \qquad \, (\eps \downarrow 0)\, .
\end{eqnarray*}
{\bf 2.} 
%
Since $\Omega_{\kappa,J} \subset \Omega_{\kappa,j}$ for $1 \leq j\leq J$, 
Theorem~\ref{thm_err_linfty} \rev{and {\bf (C$_3$)}} imply ($2<p < 3$)
$$
{\mathbb E}\bigl[ {\max_{0\leq j \leq J}} \mathbbm{1}_{\Omega_{\kappa,J}}\Vert  Z^j\Vert_{{\mathbb L}^\infty}^p\bigr] \rightarrow 0 \qquad (\varepsilon \downarrow 0)\, .
$$
The discussion around (\ref{mark1}) shows $\lim_{\varepsilon \downarrow 0}
{\mathbb P}[\Omega \setminus \Omega_{\kappa,J}] = 0$.
%
Let $\alpha > 0$. By Markov's inequality
\begin{eqnarray*}\label{linftysharp1}
\D \mathbb{P}\big[\bigl\{\max_{0\leq j\leq J}\|Z^j \|_{\mathbb{L}^\infty}^p \geq \alpha \bigr\}\big]
& \leq &\D \mathbb{P}\big[\big\{\max_{0\leq j\leq J}\|Z^j \|_{\mathbb{L}^\infty}^p \geq \alpha\big\}\cap \Omega_{\kappa,J}  \big] 
+ \mathbb{P}\big[\Omega\setminus\Omega_{\kappa,J} \big]
\\
&\leq & \frac{1}{\alpha}
{\mathbb{E}\Bigl[\D \max_{0\leq j \leq J} \mathbbm{1}_{{\Omega}_{\kappa,J}} \Vert Z^j\Vert_{{\mathbb L}^\infty}^p\Bigr]} + \mathbb{P}\big[\Omega\setminus\Omega_{\kappa,J} \big] \rightarrow 0 \qquad (\varepsilon \downarrow 0)\, .
\end{eqnarray*}
%
%

\rev{The statement then follows by the triangle inequality and part~{\bf 1.}}
\qed
\end{proof}

A consequence of Theorem~\ref{cor_sharp}
is the {convergence in probability} of the zero level set $\{\Gamma^{\eps,k}_t;\, t \geq 0\}$  to the
interface $\Gamma_t^{\tt MS}$ of the Mullins-Sekerka/Hele-Shaw problem (\ref{eq:MS}).
\begin{cor}\label{cor_gamma}
Assume {\bf (C$_3$)} and {\bf (D)}. Let $\{ X^{\eps,k}\}_{0 \leq \varepsilon \leq \varepsilon_0}$ in (\ref{interpol-1}) be obtained via Scheme \ref{scheme_time}. Then
$$
\lim_{\eps\downarrow 0}\mathbb{P}\bigl[ \bigl\{\sup_{(t,x)\in[0,T]\times\Gamma^{\eps,k}_t} \mathrm{dist}(x,\Gamma_t^{\tt MS}) > \alpha \bigr\}\bigr]=0 \qquad \forall \, \alpha > 0\,.
$$
\end{cor}
\begin{proof}
{We adapt arguments from the proof of \cite[Theorem 4.3]{fp_ifb05}.

{\bf 1.} For any $\eta \in (0,1)$ we construct an open tubular neighborhood 
$$\mathcal{N}_\eta := \bigl\{ (t,x) \in {\overline{{\mathcal D}_T}}:\, {|{\rm d}(t,x)|} < \eta\bigr\}$$
of width $2\eta$ of the interface $\Gamma^{\tt MS}$
and define compact subsets $$\mathcal{A}_{\mathcal{I}} = \mathcal{I}^{\tt MS}\setminus \mathcal{N}_\eta\,, \qquad 
\mathcal{A}_{\mathcal{O}} = \mathcal{O}^{\tt MS}\setminus \mathcal{N}_\eta\, .$$
Thanks to Theorem~{\ref{cor_sharp}} there exists $\varepsilon_0 \equiv \varepsilon_0(\eta) >0$ such that for all $\varepsilon \in (0,\varepsilon_0)$
it holds that
\begin{equation}\label{proba1}
\begin{split}
&\mathbb{P}\big[ \{|X^{\eps, {k}}(t,x) - 1| \leq \eta\,\,\mathrm{for}\,\, (t,x)\in \mathcal{A}_{\mathcal{O}}\}\big]  \geq  1-\eta\,,\\
&\mathbb{P}\big[\{ |X^{\eps, {k}}(t,x) + 1| \leq \eta\,\,\mathrm{for}\,\, (t,x)\in \mathcal{A}_{\mathcal{I}}\} \big]  \geq  1-\eta\,.
\end{split}
\end{equation}
In addition, for any $t \in [0,T]$, and $x \in \Gamma^{\eps,k}_t$, since $X^\eps(t,x) = 0$, we have
\begin{equation}\label{xint}
\bigl\vert X^{\eps, {k}}(t,x) - 1\vert = \bigl\vert X^{\eps,{k}}(t,x) + 1\vert = 1\, . 
\end{equation}

{\bf 2.} 
We observe that for any $\eta \in (0,1)$
\begin{eqnarray}\nonumber
{\mathbb P}\bigl[ \{ (t, \Gamma^{\varepsilon,k}_t);\, t \in [0,T] \subset
{\mathcal N}_\eta\}\bigr] &=& {\mathbb P}\bigl[ \bigl\{\{ (t,x):\, t \in [0,T], \ X^{\varepsilon, {k}}(t,x) =0\} \subset {\mathcal N}_\eta\bigr\}\bigr] \\
\label{adden-1}
&=& 1- {\mathbb P}\bigl[ \bigl\{ \exists\, (t,x) \in \overline{{\mathcal D}_T} \setminus {\mathcal N}_\eta:\, X^{\varepsilon, {k}}(t,x) = 0\bigr\}\bigr]
\\\nonumber
&:=& 1- {\mathbb P}\bigl[ \widetilde{\Omega}_3\bigr]
\, .
\end{eqnarray}

\rev{On noting (\ref{xint}) we deduce that $\mathbb{P}[\widetilde{\Omega}_3] \leq \mathbb{P}[\Omega_3]$ where}
$$ \Omega_3 := \bigl\{ \exists\, (t,x) \in {\mathcal A}_{\mathcal O} :\, \vert X^{\eps, {k}}(t,x) - 1\bigr\vert > \eta \ \lor \ \exists\, (t,x) \in {\mathcal A}_{\mathcal I} :\, \vert X^{\eps, {k}}(t,x) + 1\bigr\vert > \eta\bigr\}\, .$$
By (\ref{proba1}), it holds for $\eps\in(0,\eps_0)$ that
\begin{eqnarray*} 
1-{\mathbb P}[\widetilde{\Omega}_3]\geq
{\mathbb P}[\Omega \setminus \Omega_3] &=& {\mathbb P}\bigl[\bigl\{ \forall\, (t,x) \in {\mathcal A}_{\mathcal O} :\, \vert X^{\eps, {k}}(t,x) - 1\bigr\vert \leq \eta \\
&& \quad \  \land \ \forall\, (t,x) \in {\mathcal A}_{\mathcal I} :\, \vert X^{\eps, {k}}(t,x) + 1\bigr\vert \leq \eta\bigr\}\bigr] \geq 1- {2}\eta\,.
\end{eqnarray*}

Inserting this estimate into (\ref{adden-1}) yields for all $\varepsilon \in (0,\varepsilon_0)$
\begin{eqnarray*}
\rev{\mathbb{P}\bigl[ \bigl\{\sup_{(t,x)\in[0,T]\times\Gamma^{\eps,k}_t} \mathrm{dist}(x,\Gamma_t^{\tt MS}) \leq \alpha \bigr\}\bigr]}
& \geq &  \mathbb{P}\big[ \{(t,\Gamma^{\eps,k}_t),\,\,t\in[0,T]\}\subset \mathcal{N}_{\eta} \big] 
\\
& \geq & 1-{2}\eta\,,
\end{eqnarray*}
\rev{which holds for any $\alpha \geq \eta $.}
\rev{
The desired result then follows 
on noting that $\eta$ can be chosen arbitrarily small
once we take $\lim_{\eps\downarrow0}$ in the above inequality.}
\qed
}
\end{proof}
}
\begin{rem}\label{rem_sharp1}
The numerical experiments in Section \ref{sec_numer} suggest that the conditions on $\gamma$ and $k$ which are required for
Theorem~\ref{cor_sharp} to hold are too pessimistic; in particular, 
they indicate convergence to the deterministic Mullins-Sekerka/Hele-Shaw problem already for $\gamma=1$, $k =\mathcal{O}(\eps)$.
\end{rem}


\section{{Computational} experiments}\label{sec_numer}

{
The {computational} experiments {are meant to support and complement the theoretical results in the earlier sections}:
\begin{itemize}
\item Convergence to the deterministic sharp-interface limit (\ref{eq:MS}) for the space-time white noise in Section~\ref{sec_sharp_white}.
{We study} pathwise convergence of the white noise-{driven} simulations to the deterministic sharp interface limit, { which is a scenario beyond the one
for regular trace-class noise where Theorem~\ref{cor_sharp} and Corollary~\ref{cor_gamma} establish convergence in probability}.
\item {Pathwise} convergence to the stochastic sharp interface limit (\ref{eq:HS}) (introduced in Section~\ref{sec_stochms} below) for spatially smooth noise in Section~\ref{hsms-1},
where we also examine the sensitivity of numerical simulations with respect to the mesh refinement.
\end{itemize}
}

\subsection{Implementation and adaptive mesh refinement}
For the computations below we employ a mass-lumped  variant of Scheme~\ref{scheme_space_time}
\begin{equation}\label{scheme_lumped}
\begin{split}
&(X^j_h-X^{j-1}_h,\varphi_h)_h+k(\nabla w^{j}_h,\nabla \varphi_h)=\varepsilon^{\gamma}\bigl(g\Delta_j W^h, \varphi_h \bigr)_h
\;\;\;\; \, \, \quad \, \forall\, \varphi_h \in \mathbb{V}_h\,,\\
&\varepsilon(\nabla X^j_h,\nabla \psi_h)+\frac{1}{\varepsilon} \bigl(f(X^j_h),\psi_h \bigr)_h=(w^j_h,\psi_h)_h \quad \qquad \qquad
\quad  \quad \forall\, \psi_h \in \mathbb{V}_h\, ,\\
&X^0_h=u_0^{\eps,h} \in \mathbb{V}_h\, ,
\end{split}
\end{equation}
where the standard $\mathbb{L}^2$-inner product in Scheme~\ref{scheme_space_time} is replaced by  
the discrete (mass-lumped) inner product $(v,w)_h = \int_{\mathcal{D}} \mathcal{I}^h (v(x)w(x))\mathrm{d}x$ for $v,w\in \mathbb{V}_h$,
where $\mathcal{I}^h:C(\overline{\mathcal{D}})\rightarrow \mathbb{V}_h$ 
is the standard interpolation operator.
In all experiments we take $\mathcal{D} = (0,1)^2\subset \mathbb{R}^2$ and $g$ is taken to be a constant.
\rev{We note that an implicit Euler finite element scheme similar to Scheme~\ref{scheme_lumped} has been used previously in \cite{phd_goudenege},
which also performs simulations to study long time behavior of the system for different strengths of the (space-time white) noise with fixed $\varepsilon$.}

{For a given initial interface $\Gamma_{00}$ 
we construct an $\eps$-dependent family of initial conditions $\{u^{\eps}_0\}_{\eps>0}$ as $u^{\eps}_0(\xx) =\tanh (\frac{\mathrm{d}_0(\xx)}{\sqrt{2}\eps})$,
where $\mathrm{d}_0$ is the signed distance function to $\Gamma_{00}$.
Consequently, $\{u^{\eps}_0\}_{\eps>0}$ have bounded energy and contain a diffuse layer of thickness proportional to $\eps$ along $\Gamma_{00}$,
and $u^\eps_0(\xx) \approx  -1$, $u^\eps_0(\xx) \approx  1$ in the interior, exterior of $\Gamma_{00}$, respectively.
The construction ensures that $\int_\mathcal{D} u^\eps_0\, \mathrm{d}x \rightarrow m_0$ for $\eps\rightarrow 0$, 
where $m_0$ is the difference between the respective areas of the exterior and interior of $\Gamma_{00}$ in $\mathcal{D}$.
For convenience we set $u_0^{\eps, h} = \mathcal{I}^h u_0^{\eps}$.
}

The discrete increments $\Delta_j W^h = W^h(t_j) - W^h(t_{j-1})$ in (\ref{scheme_lumped})  are
$\mathbb{V}_h$-valued random variables which approximate the 
\revised{increments}
of a ${\mathcal Q}$-Wiener process on a probability space $(\Omega, \mathcal{F},\mathbb{P})$ which is given by
$$
W(t,x) = \sum_{i=1}^{\infty} \lambda_i e_i(x)\beta_i(t)\, ,
$$
where $\{e_i\}_{i\in \mathbb{N}}$ is an orthonormal basis in $\mathbb{L}^2(\mathcal{D})$, $\{\beta_i\}_{i\in\mathbb{N}}$ 
are independent real-valued Brownian motion, and $\{\lambda_i\}_{i\in\mathbb{N}}$ are
real-valued coefficients such that $\mathcal{Q}e_i = \lambda_i^2e_i$, $i\in \mathbb{N}$. 
In order to preserve mass the noise is required to satisfy $\mathbb{P}$-a.s. $\int_\mathcal{D} W(t,x) \,\revised{{\rm d}x} = 0$, $t\in [0,T]$.

In the experiments below we consider two types of Wiener \revised{processes}:
a smooth (finite dimensional) noise and a $\mathbb{L}^2_0$-cylindrical Wiener process (space-time white noise).
The smooth noise is given by
$$
\displaystyle \Delta_j \widehat{W}(t,x) = \frac{1}{2}\sum_{k,\ell=1}^{64}\cos(2\pi kx_1)\cos(2\pi \ell x_2)\Delta_j \beta_{k\ell}\qquad x=(x_1,x_2)\in [0,1]^2\,,
$$
where $\Delta_j \beta_{k\ell} = \beta_{k\ell}(t_j)-\beta_{k\ell}(t_{j-1})$ are independent scalar-valued Brownian increments.
The discrete approximation of the smooth noise is then constructed as
\begin{equation}\label{noise_n64}
\displaystyle \Delta_j W^h(\xx) = \sum_{\ell=1}^{L} \Delta_j \widehat{W}(\xx_\ell)\phi_\ell(\xx),
\end{equation}
where $\phi_\ell (\xx_m) = \delta_{\ell m}$, $\ell=1,\dots,  L$ are the (standard) nodal basis function of $\mathbb{V}_h$, i.e., $\mathbb{V}_h = \mathrm{span}\{\phi_\ell, \, \ell=1, \dots, L\}$.
The space-time white noise ($\mathcal{Q} = I$) is approximated as (cf. \cite{book_bbnp14})
\begin{equation*}
\Delta_j \widetilde{W}^h(\xx)
   =  \sum_{\ell=1}^{L} \frac{\phi_\ell (\xx)}{\sqrt{\frac{1}{3}|\mathrm{supp} \, \phi_\ell |}} \Delta_{j} {\beta}_\ell \qquad \forall\, 
\xx \in \overline{\mathcal D}\subset \mathbb{R}^2\,.
\end{equation*}
In order to preserve the zero mean value property of the noise we normalize the increments as
\begin{equation}\label{whitenoise_mean}
\Delta_j W^h = \Delta_j \widetilde{W}^h - \frac{1}{|\mathcal{D}|}\int_\mathcal{D} \Delta_j \widetilde{W}^h\, \revised{{\rm d}x} .
\end{equation}

{The Wiener process is simulated using a standard Monte-Carlo technique,
i.e., for $\omega_m \in \Omega$, $m=1, \dots, M$, 
we approximate the Brownian increments in (\ref{noise_n64}),(\ref{whitenoise_mean}) as
$\Delta_j \beta_\ell (\omega_m) \approx \sqrt{k} \mathcal{N}_\ell^j(0,1)(\omega_m)$}, 
where $\mathcal{N}_\ell^j(0,1)(\omega_m)$ is a realization of the Gaussian random number generator at time level $t_j$.
The discrete nonlinear systems related to (realizations of) the scheme (\ref{scheme_lumped})
are solved using the Newton method with a multigrid linear solver.

To increase the efficiency of the computations we employ a pathwise mesh refinement algorithm.
For a realization $X_{h,m}^{j}:=X_h^{j}(\omega_m)$, $\omega_m\in \Omega$
of the $\mathbb{V}_h$-valued random variable $X_h^{j}$
we define $\eta_{grad}(x) = \max\{|\nabla X_{h,m}^{j}(x)|, |\nabla X_{h,m}^{j-1}(x)|\}$ 
and refine the finite element mesh in such a way
that $h(x) = h_{\mathrm{min}}$ if $\eps \eta_{grad}(x) \geq 10^{-2}$
and $h(x) \approx h_{\mathrm{max}}$ if $\eps \eta_{grad}(x) \leq 10^{-3}$; the mesh 
produced at time level $j$ is then used for the computation of $X_{h,m}^{j+1}$.
The adaptive algorithm produces meshes
with mesh size $h = h_{\mathrm{min}}$ along the interfacial area
and $h \approx h_{\mathrm{max}}$ in the bulk where $u \approx \pm 1$, see
Figure~\ref{fig_2circ_conv} for a typical adapted mesh.
In our computations we choose $h_{\mathrm{max}} = 2^{-6}$
and $h_{\mathrm{min}} = \frac{\pi}{4}\eps$, i.e. $h_{\mathrm{min}} = h_{\mathrm{max}}$
for $\eps \geq 1/(16\pi)$ and $h_{\mathrm{min}}$ scales linearly for smaller values of $\eps$.

{In the presented simulations, mesh refinement
did not appear to significantly influence the asymptotic behavior of the numerical solution. This
is supported by comparison with additional numerical simulation on uniform meshes. 
The observed robustness of numerical simulations with respect to the mesh refinement can be explained
by the fact that the asymptotics are determined by pathwise properties of the solution on a large probability set.
{This conjecture is supported by the convergence in probability in Theorem~\ref{cor_sharp} and Corollary~\ref{cor_gamma}}.
In the present setup the (possible) bias due to the pathwise adaptive-mesh refinement did not have significant impact on the results.
In general, the use of adaptive algorithms with rigorous control of weak errors may be a preferable approach, cf. \cite{ps19}}.
\subsection{Stochastic Mullins-Sekerka problem and its discretization}\label{sec_stochms}

We consider the following stochastic modification of \revised{the} Mullins-Sekerka problem (\ref{eq:MS})
\begin{subequations}\label{eq:HS}
\begin{alignat}{2}
- \Delta\, v\, \mathrm{d}t & = g\, \mathrm{d}W && \qquad \mbox{in }\ 
\mathcal D \setminus \Gamma_t\,,\label{eq:HSa} \\
\left[\partial_{\hsnormal} v \right]_{\Gamma_t} & 
= - 2\,\mathcal{V}
&& \qquad \mbox{on }\ \Gamma_t\,, \label{eq:HSb} \\
v & = \alpha\,\varkappa && \qquad \mbox{on }\ \Gamma_t\,,
\label{eq:HSc} \\
\partial_{\normal} v & = 0 &&
\qquad \mbox{on } \partial \mathcal D\,,
\label{eq:HSd}\\
\Gamma_0 & = \Gamma_{00} \,. \label{eq:HSe} 
\end{alignat}
\end{subequations}
We note that the only difference between (\ref{eq:MS})
and (\ref{eq:HS}) is in the equations (\ref{eq:MSa}), (\ref{eq:HSa}), respectively.
Alternatively equation (\ref{eq:HSa}) can be stated in an integral form as
$$
-\int_0^t \Delta v\, \revised{{\rm d}s} = g\int_0^t \mathrm{d}W \qquad \mathrm{in}\quad \mathcal{D}\setminus \Gamma_t.
$$

For the approximation of the stochastic Mullins-Sekerka problem (\ref{eq:HS}), we adapt the unfitted finite
element approximation for the deterministic problem (\ref{eq:MS}) from \cite{dendritic}. 
In particular, let $\Gamma^{j-1}$ be a polygonal approximation of the interface $\Gamma$ at time $t_{j-1}$,
parameterized by $\rnvec{Y}^{j-1}_h \in [\mathbb{V}_h(I)]^2$, 
where $I = \RZ$ is the periodic unit \revised{interval}, and where $\mathbb{V}_h(I)$ is the space
of continuous piecewise linear finite elements on $I$ with uniform mesh size
$h$.
Let $\pi^h:C(I) \to \Wh$ be the standard nodal interpolation operator,
and let $\langle\cdot,\cdot\rangle$ denote the $L^2$--inner product on $I$,
with $\langle\cdot,\cdot\rangle_h$ the corresponding mass-lumped inner product.
Then we find 
$v_h^{j} \in \mathbb{V}_h$, $\rnvec {Y}^{j}_h \in \Vh$ and 
$\kappa^{j}_h \in \Wh$ such that
\begin{subequations}\label{eq:HSnum}
\begin{align}
& k\,(\nabla\,v^{j}_h, \nabla\,\varphi_h) - 2\,
\left\langle \pi^h\left[
{\rnvec{Y}^{j}_h-\rnvec{Y}^{j-1}_h} \,.\,\rnvec\nu^{j-1}_h \right], 
\varphi_h \circ \rnvec{Y}^{j-1}_h\,|[\rnvec{Y}^{j-1}_h]_\rho| \right\rangle = \bigl(g\Delta_j W^h, \varphi_h \bigr)_h 
\nonumber \\ & \hspace{8cm}
\qquad \forall\ \varphi_h \in \mathbb{V}_h
\,, \label{eq:LBa}\\
& \langle v^{j}_h, \chi_h \,|[\rnvec{Y}^{j-1}_h]_\rho| \rangle - \alpha\,\langle \kappa^{j}_h, \chi_h \,|[\rnvec{Y}^{j-1}_h]_\rho| \rangle_h = 0
\qquad \forall\ \chi_h \in \Wh \,, \label{eq:LBb} \\
& \langle \kappa^{j}_h\,\rnvec\nu^{j-1}_h, \rnvec\eta_h 
\,|[\rnvec{Y}^{j-1}_h]_\rho| \rangle_h + 
\langle [\rnvec{Y}^{j}_h]_\rho, [\rnvec\eta_h]_\rho \,|[\rnvec{Y}^{j-1}_h]_\rho|^{-1} 
\rangle = 0 \qquad \forall\ \rnvec\eta_h \in \Vh
\,. \label{eq:LBc} 
\end{align}
\end{subequations}
In the above, $\rho$ denotes the parameterization variable, so that  $|[\rnvec{Y}^{j-1}]_\rho|$ is the length element on $\Gamma^{j-1}$
and $\rnvec\nu^{j-1}_h \in \Vh$ is a nodal discrete normal vector,
see \cite{dendritic} for the precise definitions. 

\subsection{Convergence to the deterministic sharp-interface limit}\label{sec_sharp_white}

\subsubsection{One circle} \label{sec:onecircle}
We set $\gamma=1$, $g = 8\pi$ and consider the discrete space-time white noise (\ref{whitenoise_mean}).
We note that the considered space-time white noise does not satisfy the smoothness assumptions required for the theoretical part of the paper
(i.e., $\gamma > 1$ and $\mathrm{tr}(\Delta\mathcal{Q}) < \infty$), however the numerical results indicate
that for \revised{$\eps \downarrow 0$} the computed evolutions still converge to the deterministic Mullins-Sekerka problem (\ref{eq:MS}).

The numerical studies below are performed using the scheme (\ref{scheme_lumped}) with adaptive mesh refinement.
The time-step size for $\eps = 2^{-i}/(64\pi)$, $i=0, \dots,4$ was
$k_i=2^{-i}10^{-5}$. The motivation of the different choice of the time-step is
to eliminate possible effects of numerical damping and to ensure the convergence
of the Newton solver for smaller values of $\eps$.


For each $\eps$ we use the initial condition $u^{\eps,h}_0$ that approximates a circle with radius $R=0.2$.
Since circles are stationary solutions of the deterministic Mullins-Sekerka problem, the
convergence of the numerical solution for the stochastic Cahn-Hilliard equation
to the solution of the Mullins-Sekerka problem
for \revised{$\eps\downarrow 0$} can be determined by measuring the deviations
of the zero level-set of the solution $X_h^j$, $j=1,\dots, J$ from the circle with radius $R=0.2$
for a sufficiently large computational time.
We note that the zero level-set of the initial condition $u^{\eps,h}_0$ above, exactly approximates the corresponding stationary solution of the Mullins-Sekerka problem,
but it is not a stationary solution of the corresponding (discrete) deterministic Cahn-Hilliard equation, i.e.,  of (\ref{scheme_lumped}) with 
\revised{$g\equiv0$}.
In order to obtain the optimal phasefield profile across the interfacial
region, we let $u^{\eps,h}_0$ relax towards the discrete stationary state by computing with (\ref{scheme_lumped}) for \revised{$g\equiv0$} for a short time 
and then use that discrete solution as the actual initial condition for the subsequent simulations.

The results in Figure~\ref{fig_epsdw} indicate that for decreasing
$\eps$ the evolution of the zero level set of the numerical solution approaches the solution of the deterministic Mullins-Sekerka model,
which is represented by the stationary circle with radius $0.2$. We observe that the deviations
of the interface from the circle are decreasing for smaller $\eps$.
\begin{figure}[!htp]
\center
\includegraphics[width=0.6\textwidth]{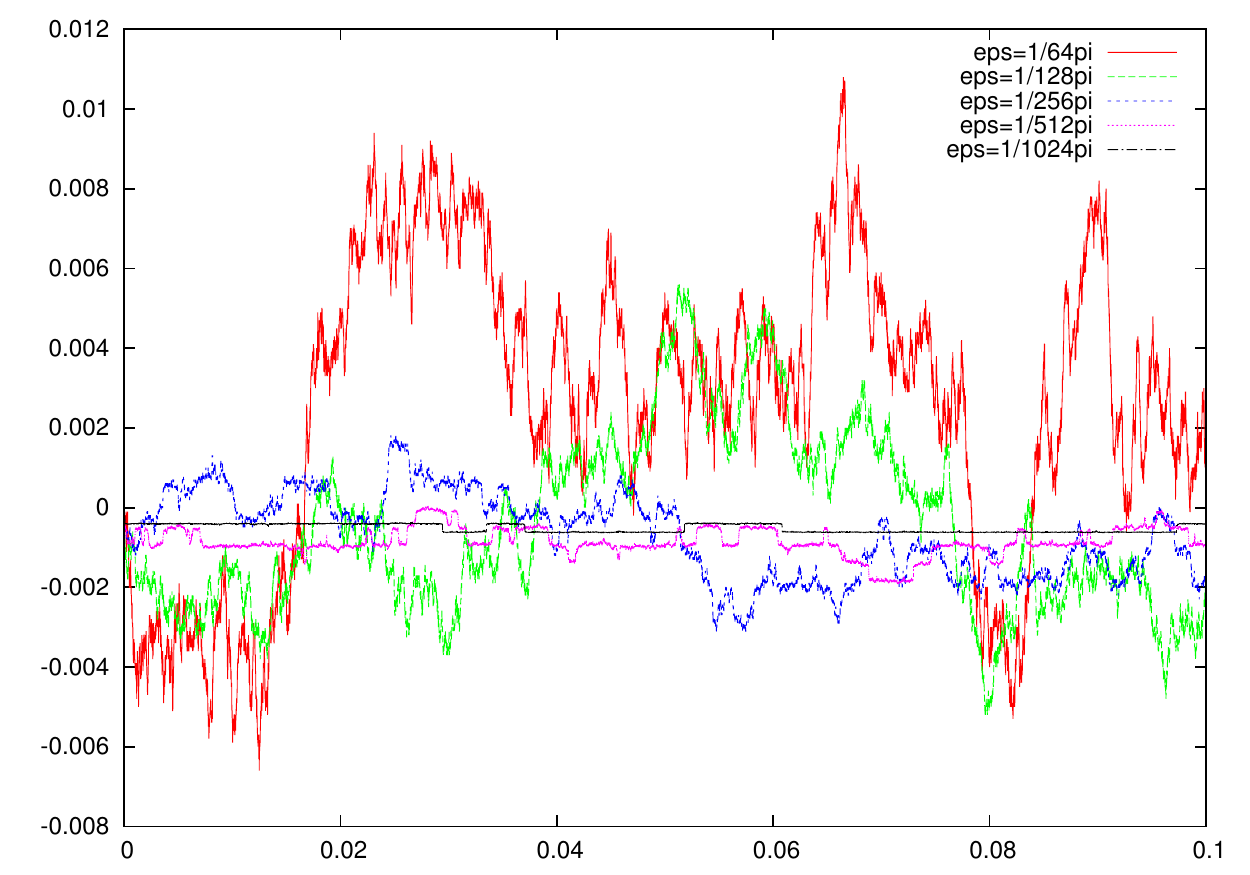}
\caption{Deviation of the interface along the $x$-axes from the circle for $\eps = 2^{-i}/(64\pi)$, $i=0, \dots,4$.}
\label{fig_epsdw}
\end{figure}

\subsubsection{Two circles}

In this experiment we consider the same setup as in the previous one
with an initial condition which consists of two circles with radii $R_1=0.15$ and $R_2=0.1$, respectively.
The evolution of the solution is more complex than in the previous experiment as the interface undergoes a topological change.
To minimize the Ginzburg-Landau energy,
the left (larger) circle grows, 
the right (smaller) circle shrinks and the resulting steady state is 
a single circle with mass equal to the mass of the two initial circles; 
see Figure~\ref{fig_2circ} for an example of a deterministic evolution with $\eps=1/(512\pi)$.
\begin{figure}[!htp]
\center
\includegraphics[width=0.25\textwidth]{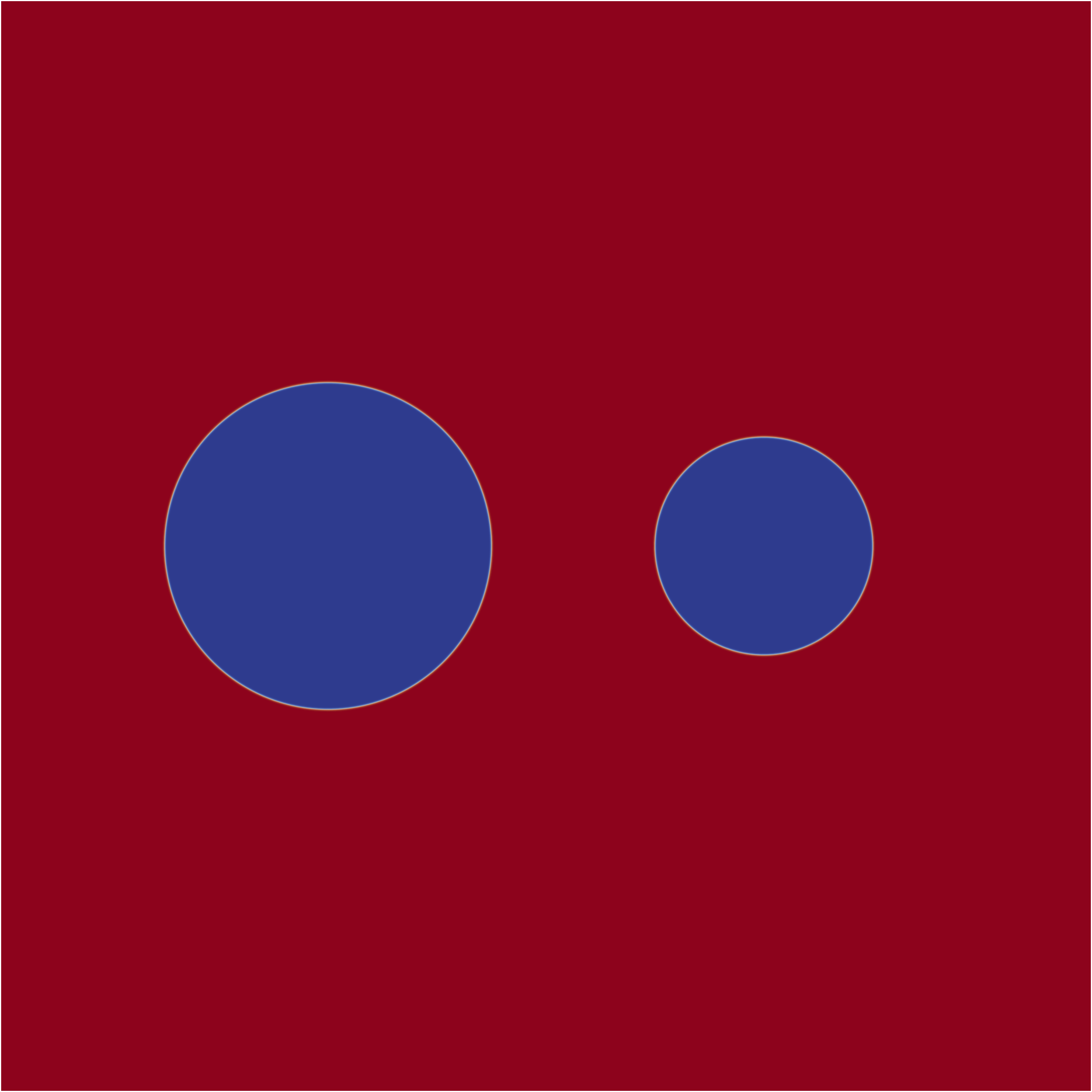}
\includegraphics[width=0.25\textwidth]{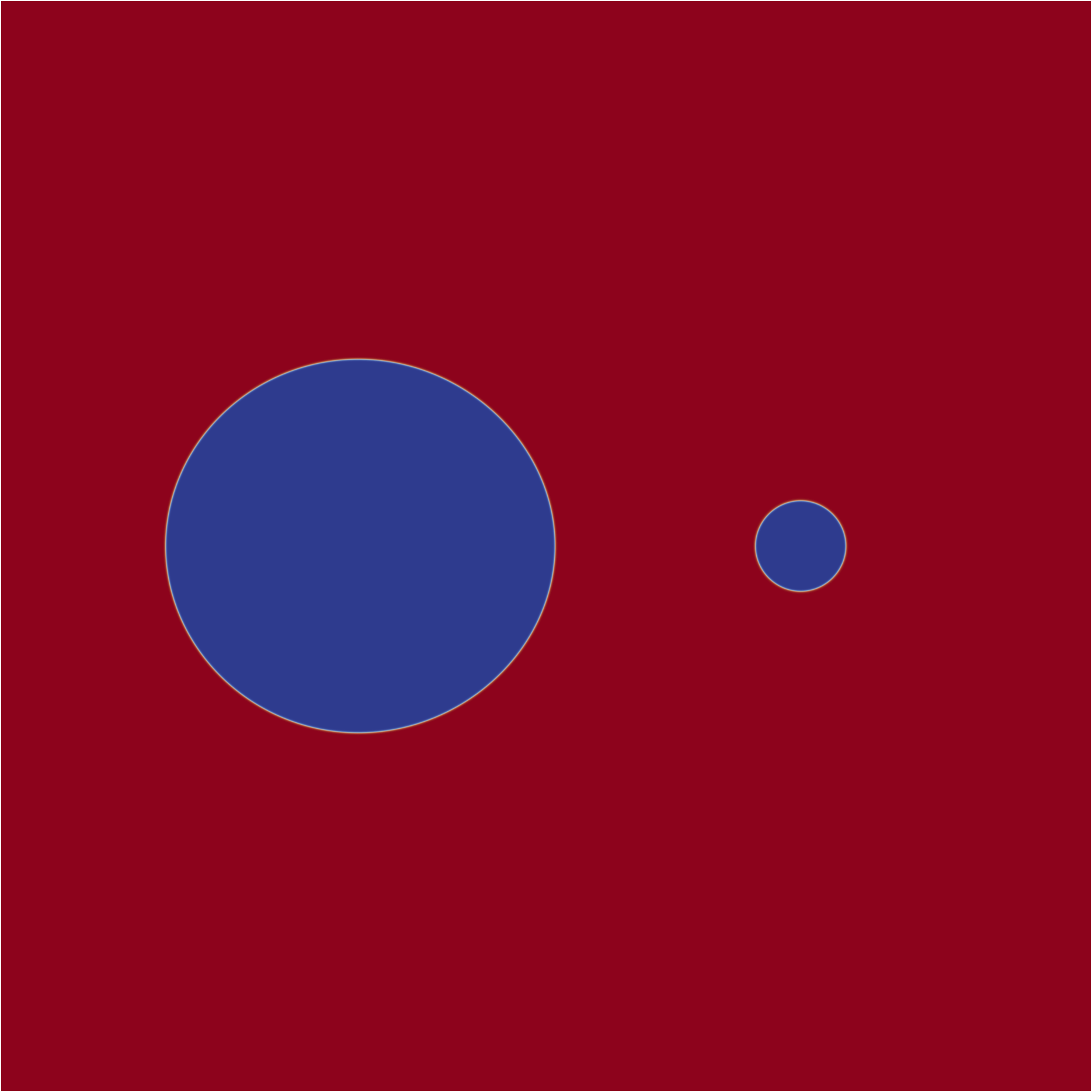}
\includegraphics[width=0.25\textwidth]{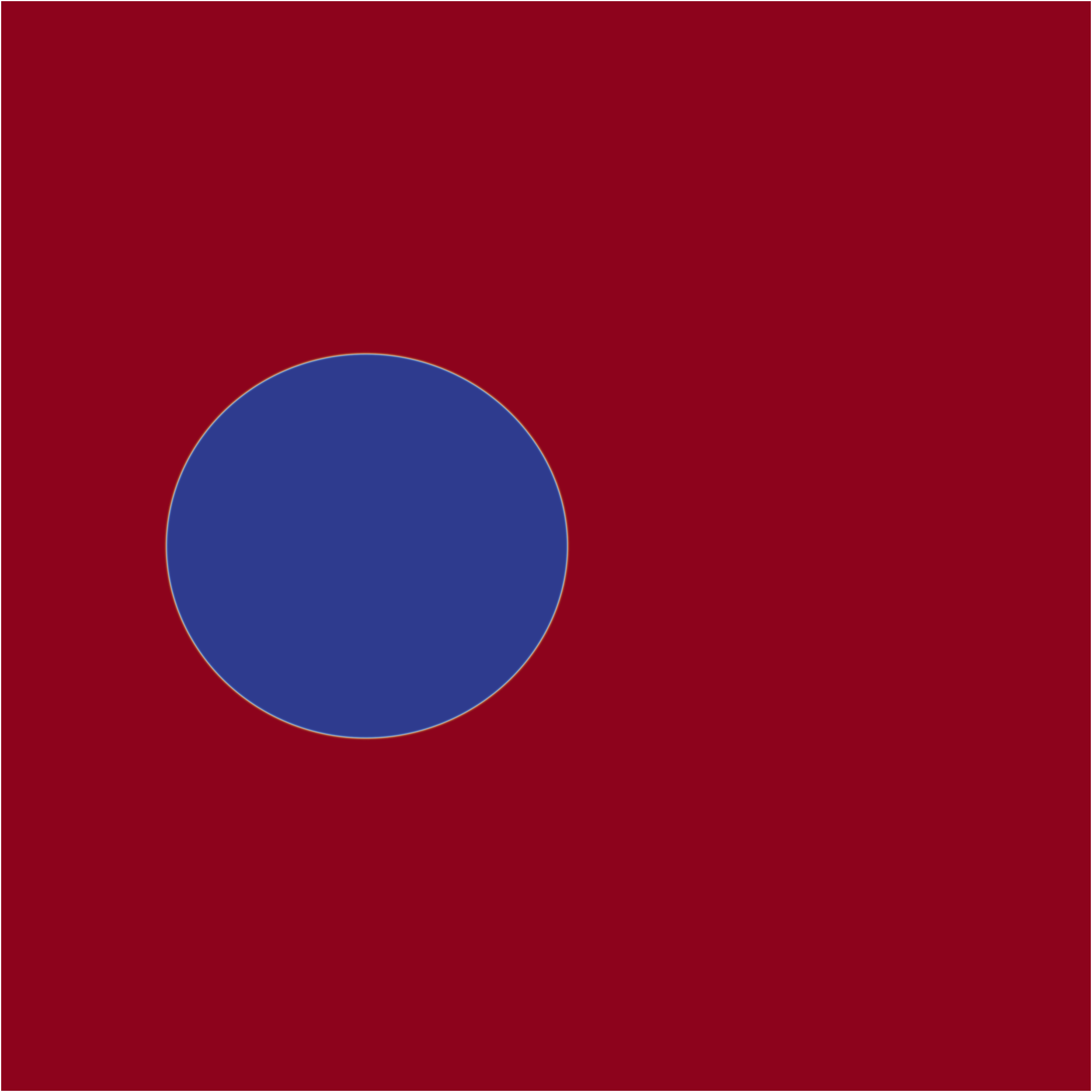}
\caption{Numerical solution for $\eps=1/(512\pi)$ at time $t=0,0.007,0.008$.}
\label{fig_2circ}
\end{figure}
In Figure~\ref{fig_2circ_conv} we display the graph of the evolution of the position of the $x$-coordinate of rightmost point
of the interface along the $x$-axis (i.e., we consider the rightmost point
on the right (smaller) circle and after the right circle disappears we track the rightmost point of the left circle)
for the deterministic Cahn-Hilliard equation as well as for typical realizations of the stochastic Cahn-Hilliard equation
for decreasing values of $\eps$, and of the deterministic Mullins-Sekerka problem.
Here the evolutions for the Mullins-Sekerka problem were computed with the scheme
\eqref{eq:HSnum} in the absence of noise.
We observe that the solution of the stochastic Cahn-Hilliard equation with the scaled space-time white noise (\ref{whitenoise_mean})
converges to the solution of the deterministic Mullins-Sekerka problem for decreasing values of the interfacial width parameter.
In addition, the differences between the the stochastic and the deterministic evolutions of the Cahn-Hilliard equation diminish for decreasing values
of $\eps$.
\begin{figure}[!htp]
\center
\includegraphics[width=0.5\textwidth]{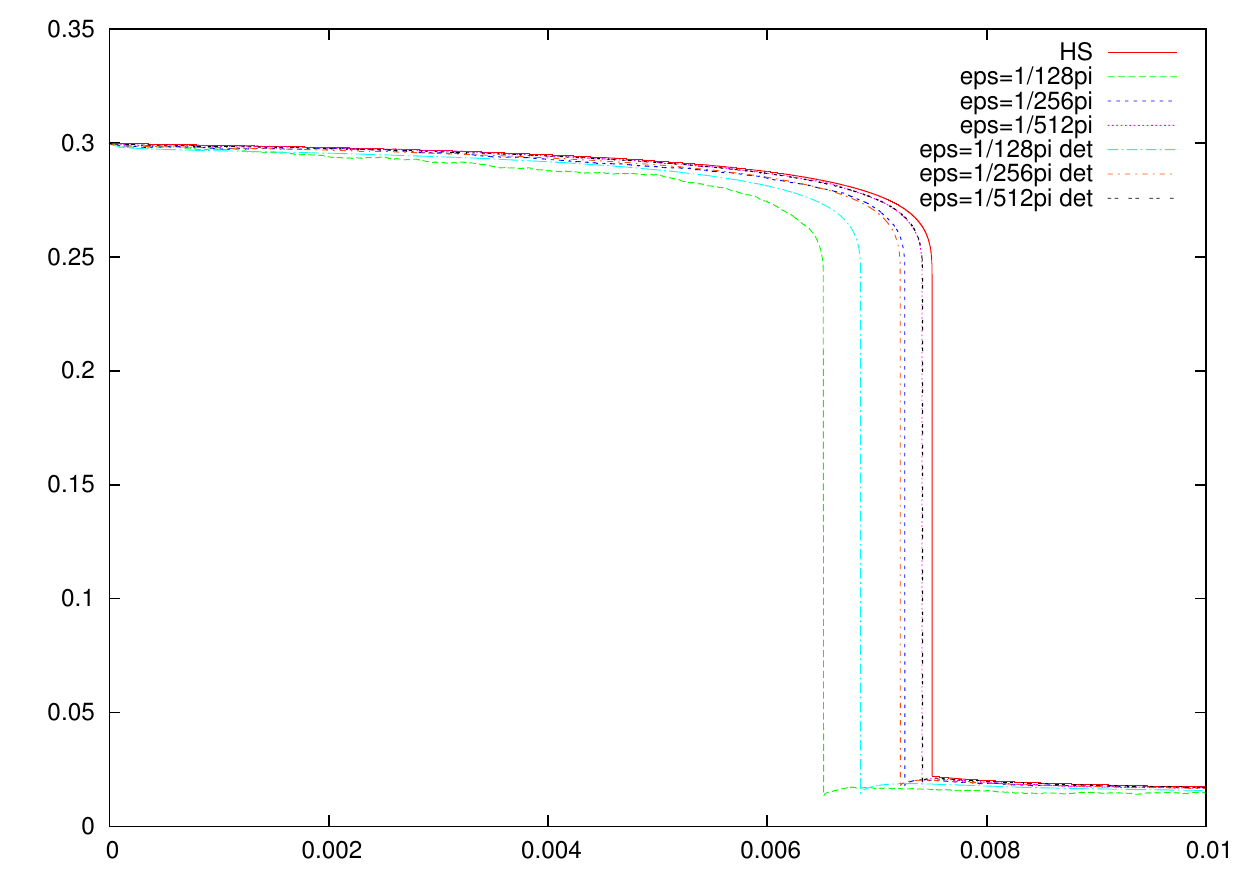}
\includegraphics[width=0.35\textwidth]{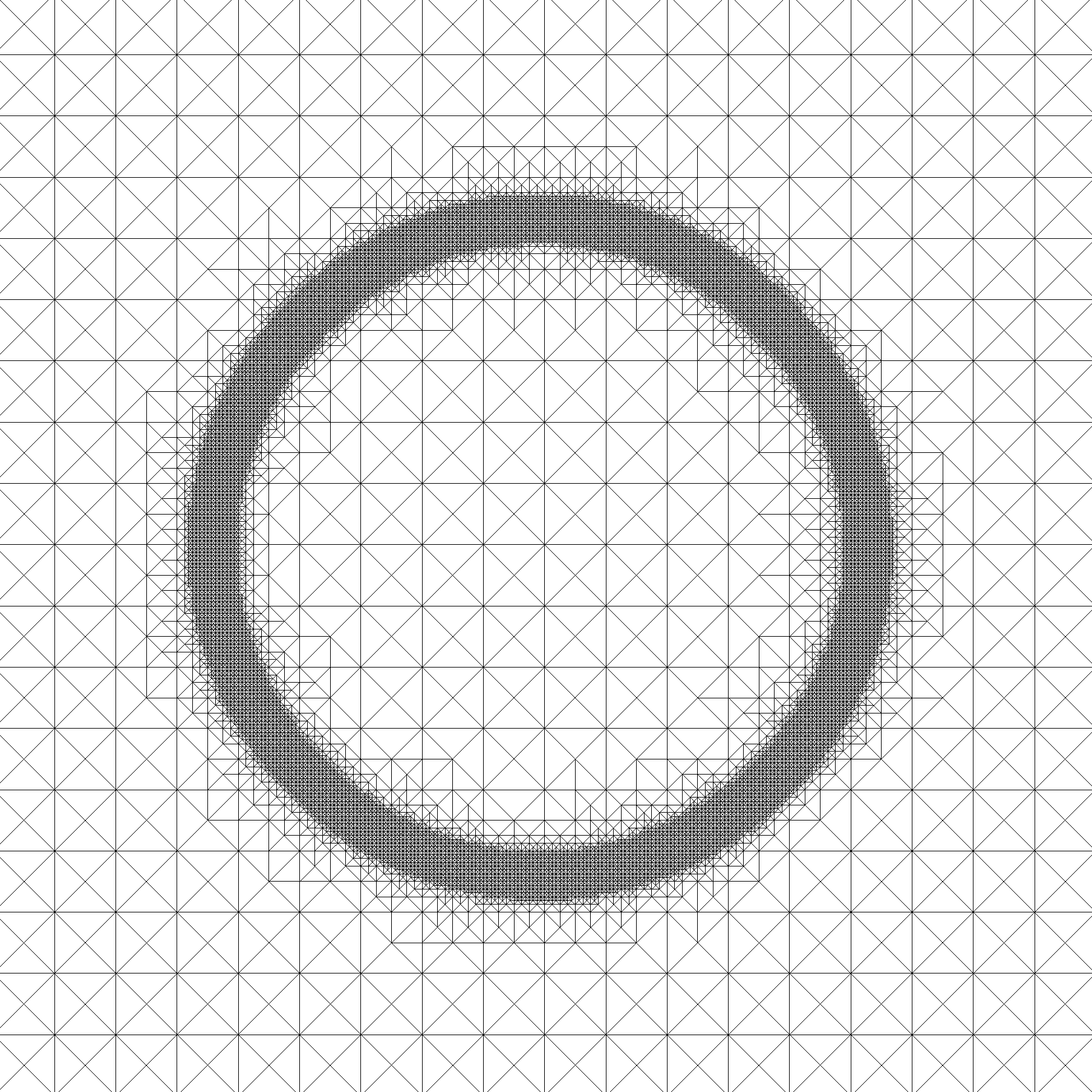}
\caption{(left) Position of the rightmost point of the interface for the stochastic
and the deterministic Cahn-Hilliard equations with $\eps = 2^{-i}/(64\pi)$, $i=0, \dots,4$, $\gamma=1$ 
and the deterministic Mullins-Sekerka problem; the values are shifted by $-0.5$. 
(right) Zoom on the adapted mesh around the smaller circle for $\eps=1/(512\pi)$ at $t=0.007$.}
\label{fig_2circ_conv}
\end{figure}

\subsection{Comparison with the stochastic Mullins-Sekerka model}\label{hsms-1}

We use the numerical scheme (\ref{scheme_lumped})
to study the case of non-vanishing noise, i.e., $\gamma=0$, with the discrete approximation of the smooth noise (\ref{noise_n64}).
The noise is symmetric across the center of the domain in order to facilitate an easier
comparison with the Mullins-Sekerka problem.
{The computations below are pathwise, i.e.,
in the graphs below we display results computed for a single realization of the Wiener process.}
If not mentioned otherwise we use the time-step size $k=10^{-5}$.

The initial condition is taken to be the $\eps$-dependent approximation of a circle with radius $R=0.2$ as in \S\ref{sec:onecircle}.
In the computations, as before, we first let the initial condition relax to a stationary state 
and then use the stabilized profile $X^{0}_{h} := X^{j_{s}}_{h}$ as an initial condition
for the computation.
The zero level-set of the stationary solution $X^{j_s}_{h}$ is a circle with perturbed radius $R=0.2+\mathcal{O}(\eps)$,
where in general the perturbation $\mathcal{O}(\eps)$ also depends on the finite element mesh.
To compensate for the effect of the perturbation in the initial condition for larger values of $\eps$ we represent the interface
by a level set $\Gamma_{u_{\Gamma}}^j := \{\xx \in \mathcal{D};\, X^j_h(x) = u_{\Gamma}\}$
(i.e., $\Gamma_{0}^j$ is the zero level set of the discrete solution at time level $t_j$)
where the values $u_{\Gamma} = X^{j_{s}}(0.2,0)$, i.e., it is the ''compensated''
level-set for which the stationary profile $\Gamma_{u_{\Gamma}}^{j_{s}}$ coincides with the circle with radius $R=0.2$.
The usual value for the ''compensated'' level-set was $u_{\Gamma} \approx 0.27$ in the computations below.

We observe that in order to properly resolve the spatial variations of the noise
it is necessary to use \revised{a} mesh size smaller or equal to $h_{\mathrm{max}} = 2^{-7}$ for the discretization
of the Cahn-Hilliard equation. The computations for the Mullins-Sekerka problem,
using the scheme \eqref{eq:HSnum}, 
were more sensitive to the mesh size, and an accurate resolution for the considered noise required
\revised{a}
mesh size $h_{\mathrm{max}} = 2^{-8}$, cf.~Figure~\ref{fig_cmp_n64_uniform} which includes the results
for $h_{\mathrm{max}} = 2^{-8}$ as well as $h_{\mathrm{max}} = 2^{-7}$.

In Figure~\ref{fig_cmp_n64_uniform} we compare the evolution  for the stochastic Cahn-Hilliard equation
for $\eps = 1/(32\pi)$, $\eps = 1/(64\pi)$
on a uniform mesh with $h = 2^{-7}$, 
$h = 2^{-8}$, respectively, with the evolution of the stochastic Mullins-Sekerka problem (\ref{eq:HS})
on uniform meshes with $h = 2^{-7}$, $h = 2^{-8}$, respectively, for a single realization of the noise.
We also include results for $\eps = 1/(128\pi)$, $\eps = 1/(512\pi)$,
where to make the computations feasible we employ the adaptive algorithm with $h_{\mathrm{max}}=2^{-8}$
and $h_{\mathrm{max}}=2^{-9}$, $h_{\mathrm{max}}=2^{-11}$, respectively.
Furthermore, in order to ensure convergence of the Newton solver
for $\eps = 1/(512\pi)$ we decrease the time-step size $k=10^{-6}$.
To be able to directly compare with the results for $\eps = 1/(512\pi)$, 
we take the values of the realization of the noise generated with step size $k=10^{-5}$,
which was used in the other simulations,
and to obtain values at the intermediate time levels we employ linear interpolation in time.
\begin{figure}[!htp]
\center
\includegraphics[width=0.48\textwidth]{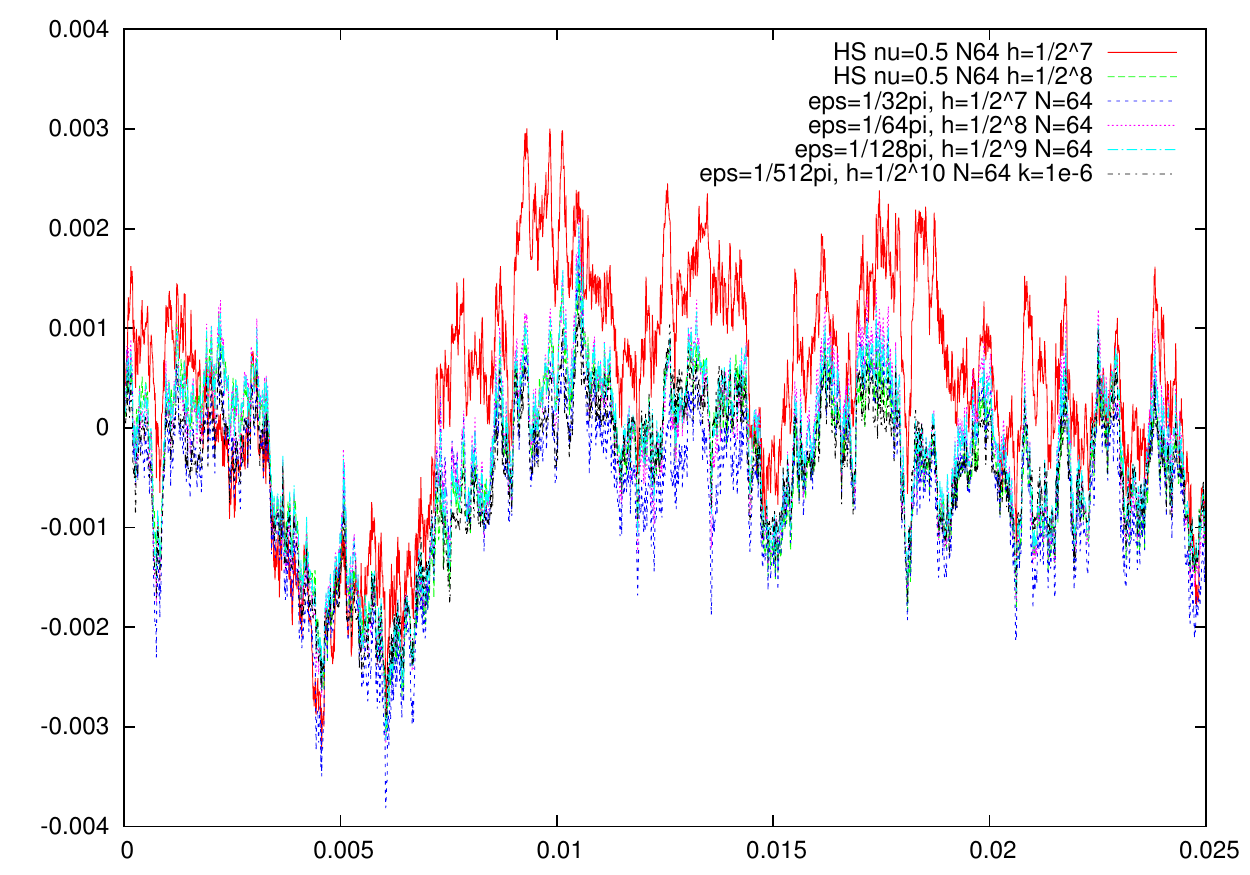}
\includegraphics[width=0.48\textwidth]{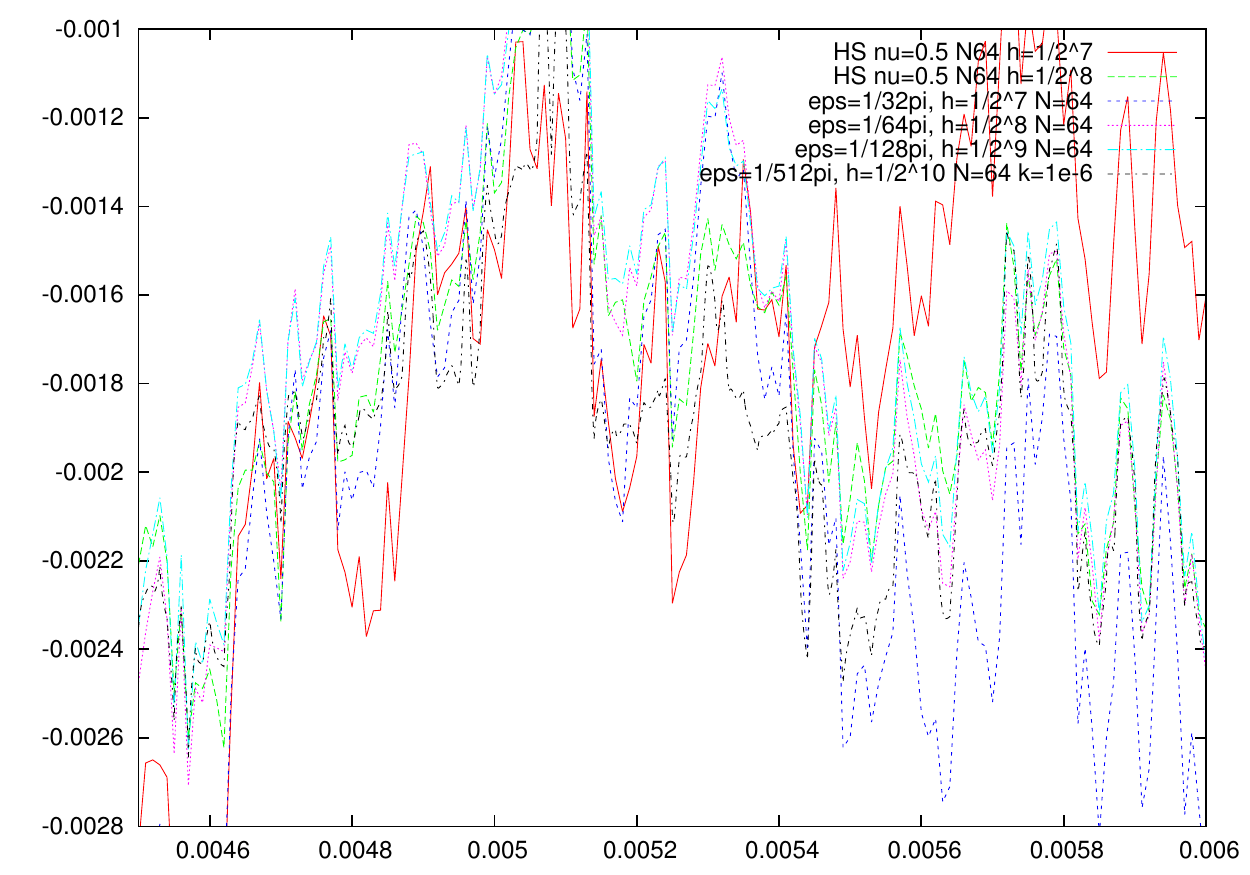}
\includegraphics[width=0.48\textwidth]{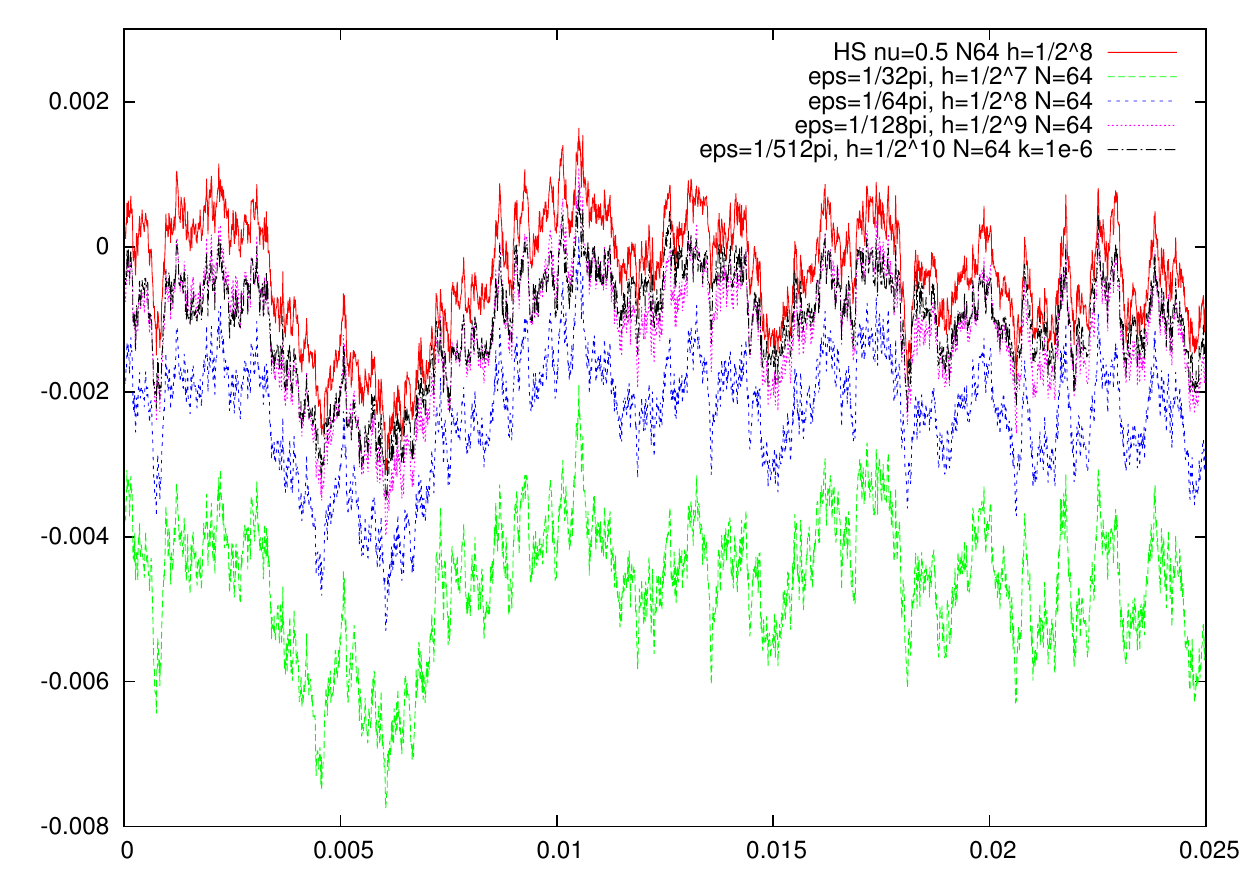}
\caption{Oscillations of the interface along the $x$-axis $(x,0)$ on uniform meshes 
for the stochastic Cahn-Hilliard equation with $\eps = 1/(32\pi)$, $h = 2^{-7}$,
$\eps = 1/(64\pi)$, $h = 2^{-8}$, 
$\eps = 1/(128\pi)$, $h_{\mathrm{min}} = 2^{-9}$, $\eps = 1/(512\pi)$, $h_{\mathrm{min}} = 2^{-11}$
and for the stochastic Mullins-Sekerka problem with $h = 2^{-7}$ and $h = 2^{-8}$
with the noise (\ref{noise_n64}) (top left); detail of the evolution (top right);
evolution of the zero level-set of the solution (bottom middle).}
\label{fig_cmp_n64_uniform}
\end{figure}
We observe that the results in Figure~\ref{fig_cmp_n64_uniform} for the stochastic Mullins-Sekerka model are
more sensitive to the mesh size, i.e., the graph \revised{for the} mesh with $h = 2^{-7}$ 
differs significantly from the remaining results.
For the mesh with $h_{\mathrm{min}} = 2^{-8}$ the results for the stochastic Mullins-Sekerka model are in good agreement
with the results for \revised{the} stochastic Cahn-Hilliard model.
We note that for values smaller than $\eps = 1/(128\pi)$  
we do not observe significant improvements of the approximation of the stochastic Mullins-Sekerka problem.
This is likely caused by the discretization errors in the numerical approximation of the stochastic Mullins-Sekerka model
which, for small values of $\eps$, are greater than the approximation error w.r.t. $\eps$ in the stochastic Cahn-Hilliard equation.

{From the above numerical results we conjecture that for \revised{$\eps \downarrow 0$}  
the solution of the stochastic Cahn-Hilliard equation with a non-vanishing noise term ($\gamma=0$)
converges to the solution of a stochastic Mullins-Sekerka problem (\ref{eq:HS}).
Formally, the stochastic Mullins-Sekerka problem (\ref{eq:HS}) can be obtained as a sharp-interface limit of a generalized Cahn-Hilliard equation
where the noise is treated as a deterministic function $G_1(t) = g\, \dot{W}(t)$,
cf.~(2.3) in \cite{abk18} and~(1.12) in \cite{AKK}.}

To examine the robustness of previous results with respect to adaptive mesh refinement
we recompute the previous problems with the noise (\ref{noise_n64}) using the adaptive mesh refinement 
algorithm \revised{with} $h_{\mathrm{max}}=2^{-6}$
and $h_{\mathrm{min}}=\frac{\pi}{4}\eps$.
The stochastic Mullins-Sekerka model is computed with $h_{\mathrm{max}}=2^{-6}$ and
the mesh is refined along the interface $\Gamma$ with mesh size $h_{\mathrm{min}}=2^{-8}$.

We note that with adaptive mesh refinement the results differ from those computed using uniform meshes,
since the noise (\ref{noise_n64}) is mesh dependent. \revised{For} instance,
in the regions with coarse mesh the noise (\ref{noise_n64}) is not properly resolved.
The computed results with the adaptive mesh refinement
can be interpreted as replacing the additive noise
(\ref{noise_n64}) with a multiplicative type noise that has lower intensity when $u\approx\pm 1$.
The presented computations contain an additional ''geometric'' factor in the numerical error
that is due to the fact that the mesh is adapted according to the position of the interface, 
as well as due to the fact that the adaptive mesh refinement algorithm for the Mullins-Sekerka problem is different.
Nevertheless, the results are still in good agreement with the stochastic Mullins-Sekerka problem, see Figure~\ref{fig_cmp_n64_adapt}.
In particular we observe that the convergence for smaller values of $\eps$ is more obvious for the zero level-set of the solution
than in the case of uniform meshes.
In Figure~\ref{fig_cmp_n64_adapt} we also \revised{include} a graph ('ftilde' in pink) 
which was computed using a modification of scheme (\ref{scheme_lumped}) with $\bigl(f(X^j_h),\psi_h \bigr)$ replaced by $\bigl(\tilde{f}(X^j_h, X^{j-1}_h),\psi_h \bigr)$ where
$\tilde{f}(X^j_h, X^{j-1}_h) = \frac{1}{2}(|X^j_h|^2-1)(X^j_h + X^{j-1}_h)$; for equal time-step size the modified scheme provides worse approximation
of the Mullins-Sekerka problem due to numerical damping.

\begin{figure}[!htp]
\center
\includegraphics[width=0.48\textwidth]{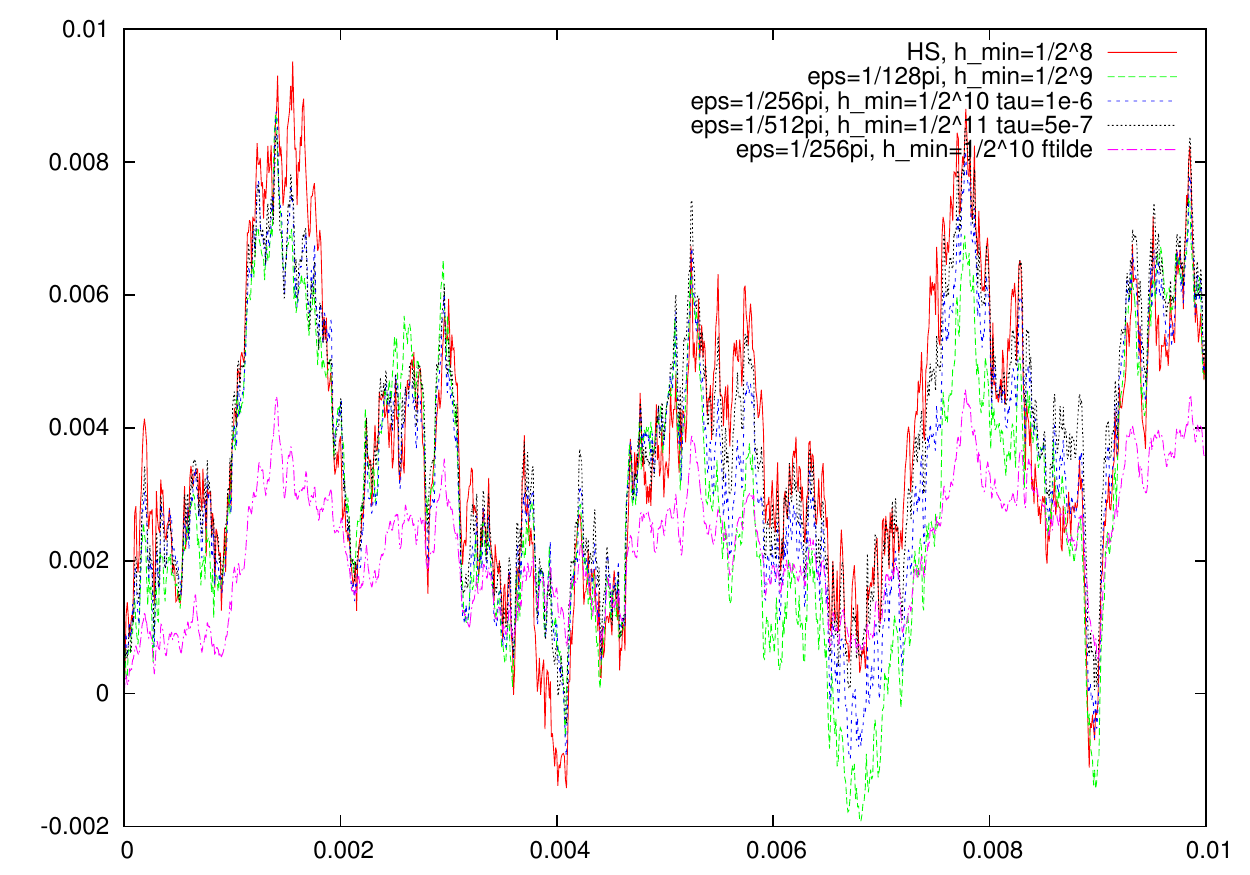}
\includegraphics[width=0.48\textwidth]{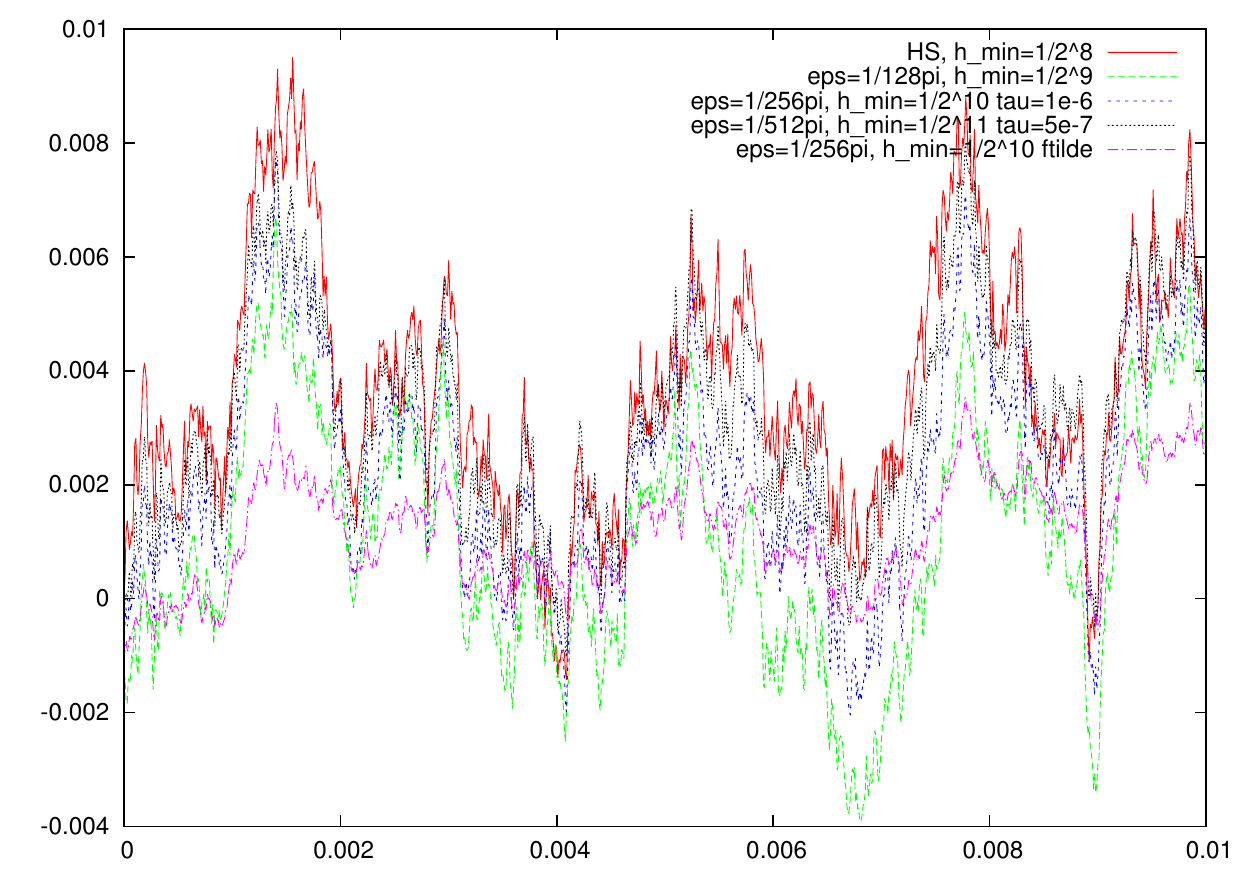}
\caption{Oscillations of the ''compensated'' level-set along the $x$-axis $(x,0)$ 
with adaptive mesh refinement with $h_{\mathrm{max}}=2^{-6}$
for stochastic Cahn-Hilliard equation with $\eps = 1/(32\pi)$, $h_{\mathrm{min}} = 2^{-7}$, $\eps = 1/(64\pi)$,
$h_{\mathrm{min}} = 2^{-8}$,  $\eps = 1/(128\pi)$, $h_{\mathrm{min}} = 2^{-9}$,
and the stochastic Mullins-Sekerka problem with $h_{\mathrm{min}} = 2^{-8}$, $h_{\mathrm{max}}=2^{-6}$
with the noise (\ref{noise_n64}) (left picture); evolution of the corresponding zero level-set (right picture).}
\label{fig_cmp_n64_adapt}
\end{figure}


\section*{Acknowledgement}
Financial support by the DFG through the CRC ''Taming uncertainty and profiting from
randomness and low regularity in analysis, stochastics and their applications'' is acknowledged.
The authors would like to thank the referees for helpful remarks and suggestions.



{\black{

\end{document}

}}

\end{document}